\documentclass[12pt]{article}

\usepackage{amsmath,amssymb,amsthm,lscape,young,hyperref,cancel,makeidx}
 \usepackage{verbatim}

\usepackage{geometry}
\usepackage{graphicx}
\usepackage{epstopdf}
\usepackage{mathabx}
\usepackage[all,arc]{xy}

\newtheorem{theorem}{Theorem}[section]
\newtheorem{proposition}[theorem]{Proposition}
\newtheorem{lemma}[theorem]{Lemma}

\theoremstyle{definition}

\theoremstyle{remark}
\newtheorem{remark}[theorem]{Remark}

\numberwithin{theorem}{section}

\font\cyr=wncyr10
\newcommand{\Sh}{\hbox{\cyr X}}
\renewcommand{\i}{\mathrm{i}}

\newcommand{\cD}{\mathcal{D}}

\def\eps{{\epsilon}}

\def\bE{{\mathbb {E}}}

\def\bZ{{\mathbb {Z}}}

\def\bR{{\mathbb {R}}}
\def\bC{{\mathbb {C}}}
\def\bP{{\mathbb {P}}}
\def\bO{{\mathbb {O}}}

\def\bT{{\mathbb {T}}}
\def\bI{{\mathbb {I}}}

\def\pB{{\mathcal B}}

\def\pE{{\mathcal E}}
\def\pF{{\mathcal F}}

\def\pI{{\mathcal I}}

\def\pL{{\mathcal L}}

\def\pO{{\mathcal O}}
\def\pP{{\mathcal P}}
\def\pQ{{\mathcal Q}}

\def\pU{{\mathcal U}}
\def\pV{{\mathcal V}}

\DeclareMathOperator{\diag}{diag}

\newcommand{\D}[2]{\frac{D #1}{d#2}}

\renewcommand{\ker}{\mathop{\mathrm{Ker}}\nolimits}

\begin{document}
\title{Monads for Instantons and Bows}

\author{Sergey A. Cherkis\thanks{
Department of Mathematics, University of Arizona, 
Tucson, AZ 85721-0089, USA, 
\tt cherkis@math.arizona.edu}
\and
Jacques Hurtubise\thanks{
Department of Mathematics,
McGill University,
Burnside Hall,
805 Sherbrooke St.W.,
Montreal, Que. H3A 2K6,
Canada,
\tt jacques.hurtubise@mcgill.ca}}
\date{}
\maketitle

\abstract{Instantons on the Taub-NUT space are related to `bow solutions' via a generalization of the ADHM-Nahm transform. Both are related to complex geometry, either via the twistor transform or via the Kobayashi-Hitchin correspondence. We explore various aspects of this complex geometry, exhibiting equivalences. For both the instanton and the bow solution we produce two monads encoding each of them respectively.  Identifying these monads we establish the one-to-one correspondence between the instanton and the bow solution.}

  \newpage\tableofcontents\newpage

\section{Introduction} 
The general paradigm in the understanding of any anti-self-dual (ASD)  field equation  on a compact K\"ahler manifold $X$ is given by the Kobayashi-Hitchin correspondence, which goes, roughly, as follows:
$$\begin{matrix} \begin{pmatrix} {\text{Solutions to the}}\\ {\text{anti-self-duality}}\\ \ \text{equations}\\ \ \text{on}\ X \end{pmatrix} & \leftrightarrow & \begin{pmatrix} \text{Holomorphic stable}\\ \ \text{vector bundles} \\ \ \text{on}\ X  \end{pmatrix}\end{matrix}$$
The anti-self-duality equations contain the conditions for integrability for a complex structure; thus the rightward arrow above is simply a forgetful map. The leftward arrow involves solving a variational problem, and the solubility of this problem requires stability. This, being the rough picture, in particular cases, there are often associated auxiliary structures, which modify both the equations and the stability condition.

When the base manifold $X$ is hyperk\"ahler, there is a third element, which is the twistor space $Z$ of the manifold $X$. This space is diffeomorphic to the product $X\times \bP^1$; the $\bP^1$ parametrizes the various complex structures of the hyperk\"ahler structure, and the projection  $X\times \bP^1\rightarrow \bP^1$ is holomorphic. The correspondence expands to:

$$\begin{matrix} \begin{pmatrix} {\rm \ Sols\ to\ the}\\ {\rm \ ASD }\\ \ {\rm equations}\\ \ {\rm on\ X}  \end{pmatrix} & \leftrightarrow & \begin{pmatrix} {\rm \ Holomorphic}\\ \ {\rm vector\ bundles} \\ \ {\rm on\ X}\times \bP^1 \\ \ {\rm (plus\ conditions)}\end{pmatrix}& \leftrightarrow & \begin{pmatrix} {\rm \ Holom.\ stable}\\ \ {\rm vector\ bundles} \\ \ {\rm on\ X} \end{pmatrix}\end{matrix}$$
The righthand map is simply restriction to a fibre. 

We are interested in the picture that holds for the Taub-NUT manifold. This manifold is not compact; the boundary conditions  will dictate the compactification geometry that appears on the holomorphic side of the correspondence.  The purpose of this paper  is to exhibit the trio of data above in the case of the Taub-NUT manifold, and to link it to another trio of data, by showing the equivalence between ASD instantons on the Taub-NUT manifold and the `bow solutions' of \cite{Cherkis:2010bn} (containing certain solutions to Nahm's equations). A persistent theme throughout is the encoding of the structures into various versions of a monad, an algebro-geometric structure that gives the relevant bundles as a cohomology: one has a sequence of  bundles
$$ A{\buildrel{\alpha}\over{\longrightarrow}}B{\buildrel{\beta}\over{\longrightarrow}}C$$
with $\beta\circ\alpha = 0$; the relevant bundle is the quotient
$\ker(\beta)/{\mathrm {Im}}({\alpha})$

In its holomorphic version, the instanton-bow equivalence has on the  instanton side a holomorphic bundle, though this time equipped with some extra data coming from the boundary behaviour. On the bow side, the holomorphic data is a solution to (part of)  Nahm's equations, along, again, with some extra linear data; these are versions of the Nahm complexes introduced by Donaldson for $SU(2)$ monopoles \cite{Donaldson:1985id}, and extended to other gauge groups in \cite{Hurtubise:1989wh}, and then to the case of calorons in \cite{Charbonneau:2006gu}. This  expands the correspondence
$$\begin{matrix} \begin{pmatrix} \text{Solns to}\\ \text{the ASD }\\ \ \text{equation}\\ \ \text{on}\ X  \end{pmatrix} & \leftrightarrow & 
\begin{pmatrix} \text{Holom.}\\  \text{vector bundles} \\ \ \text{on}\ X\times \bP^1 \\ \ \text{(+\ conditions)}\end{pmatrix}& \leftrightarrow & 
\begin{pmatrix} \text{Holom. stable}\\ \ \text{vector bundles} \\ \ \text{on}\ X \end{pmatrix}\\ \\
\quad \text{Up}\uparrow\ \ \downarrow\text{Down}&& \updownarrow&&\updownarrow\\ \\
\begin{pmatrix} \text{Bow Solns:}\\ \text{Solns to Nahm's Eqs}\\ \text{+\ Linear Data} \end{pmatrix} & \leftrightarrow & \begin{pmatrix} {\rm \ Spectral }\\ \ {\rm data\  } \\ \ {\rm on\ }  T\bP^1 \\ \ {\rm (+\ conditions)}\end{pmatrix}& \leftrightarrow & \begin{pmatrix} {\rm \ Holomorphic\  }\\ \ {\rm \ bow \ complex} \\ \ {\rm data\ on\ \bP^1}\end{pmatrix}
\end{matrix}$$
In the various portions of this diagram, we shall encounter monads; these will encode our objects as subquotients of simple objects, typically related by  matrices satisfying certain constraints. This point of view \cite{Horrocks, ADHM,Nahm:1979yw,Nahm1983} is quite fruitful. The right hand side of the diagram is holomorphic, and somewhat simpler. We will first explore this part of the picture, but beforehand, give more precise definitions of the picture's components.  

We will throughout this paper, concentrate on the case of $SU(2)$ and $U(2)$ instantons. The cases of unitary groups of higher rank has some supplementary complications, which in essence are already present in the study of $SU(n)$ and $U(n)$ monopoles, as in \cite{HurtubiseMurray}. One can find a treatment in  \cite{Takayama:2016}. The moduli spaces of bows and instantons on the Taub-NUT space are isometric to the Coulomb branches of the moduli spaces of vacua of quantum field theories \cite{Nakajima:2016guo}.  This leads to potential applications of our work in quantum field theory \cite{Nakajima:2016guo, Braverman:2016wma},  gauge theory mirror symmetry \cite{Gaiotto:2008ak,Cherkis:2011ee}, brane dynamics \cite{Witten09}, and geometric Langlands correspondence \cite{Witten:2009at,Braverman:2017ofm}.

\section{Objects of Study}

\subsection{The Taub-NUT manifold}

The Taub-NUT manifold $X_0$ is a hyperk\"ahler manifold diffeomorphic to $\bR^4$.
Its geometry is   closely tied to that of the Hopf map $\bR^4\xrightarrow[]{\pi} \bR^3$ (a circle bundle away from the origin): fixing a complex structure and linear complex coordinates $(\xi, \psi)$ on $\bR^4\simeq\mathbb{C}^2$, the Hopf map is given by 
$$\pi:(\xi,\psi)\mapsto (t_1+it_2, t_3) = \left(\xi\psi, \frac{|\psi|^2-|\xi|^2}{2}\right).$$

The  fibres of this map  are orbits under the action by complex scalars of unit length: $(\xi,\psi)\mapsto(\lambda\xi,\lambda^{-1}\psi),\ |\lambda|=1.$ Away from the origin these fibres are circles. 

The Taub-NUT metric has  these orbits tending  asymptotically to circles of a constant length: explicitly there are local coordinates\footnote{
The local coordinate $\theta$ is identified with  $ \theta+2\pi$ and is a local coordinate over a cone in $\mathbb{R}^3$ spanned by a connected contractible region in its unit sphere.  Note, that although $\theta$ and $\vec{\omega}\cdot d\vec{t}$ are local, the one-form $d\theta+\vec{\omega}\cdot d\vec{t}$ is a global one-form dual to the isometry generating vector field $\frac{\partial}{\partial\theta}.$
} 
$(t_1,t_2, t_3,\theta)\in \bR^3\times [0,2\pi)$ in which the action of $S^1$  by a linear shift in $\theta$ is isometric,  and the metric is locally given by the Gibbons-Hawking ansatz:
\begin{align}\label{TNmetric}
ds^2 = V(t) d\vec{t}\phantom{|}^{2} + \frac{(d\theta+\vec{\omega} \cdot d\vec{t}\,)^2} {V(t)},
\end{align}
with 
$$V(t) = \ell + \frac{1} { 2|\vec{t}|},$$
where $\ell>0$ is a fixed parameter determining the asymptotic size of the $S^1$
and the local one-form $\vec{\omega}d\vec{t}$ appearing in the metric is related to $V$ by 
$\frac{\partial}{\partial t_i}  V = \epsilon_{ijk}\frac{\partial}{\partial t_j}\omega_k.$

  Away from the origin, there is a complex version of this picture, when one fixes one of the complex structures on $X_0= \bR^4$: one can project $X_0\setminus\{0\}$ further to the unit sphere, and the orbits there are given by the action of $\bC^*$; it is simply the map 
  $\bC^2\backslash 0 \rightarrow \bP^1$.
  
  \subsubsection{The twistor space}\label{Sec:TwistorSpace}
    
  The Euclidean three-space has a minitwistor space given by the total space $\bO(2)$ of the bundle 
$\pO(2)$ over $\bP^1$. If $\zeta$ is the natural coordinate on $\bP^1$ and $\eta\frac{\partial}{\partial\zeta}\in T\mathbb{P}^1$, then $\eta$ is the corresponding fibre coordinate on $\bO(2).$  The twistor correspondence between the points $(t_1,t_2,t_3)$ in $\mathbb{R}^3$ and real sections of $\pO(2)$ called the twistor lines is given by 
$$\eta = (t_1+\i t_2) - 2  t_3\zeta - (t_1-\i t_2)\zeta^2.$$
The variable $\zeta$ parametrises the oriented directions in $\bR^3$; it also parametrizes the complex structures on the Taub-NUT manifold. These twistor  lines are invariant under the real structure involution $(\eta, \zeta)\mapsto (-\bar\eta/\bar\zeta^2,-1/\bar\zeta)$.

The space $\bO(2)$ has over it a family of line bundles $L^\ell$,  with exponential transition function $\exp(-\ell\eta/\zeta)$, so that one has coordinate patches ${V_0= \{\zeta\neq\infty\}}$ and   $V_1=  \{\zeta\neq 0\}$ on $\bO(2)$, with respective coordinates $(\mu, \eta,\zeta)$ and  $(\mu',\eta',\zeta')$ related by
$$(\mu', \eta',\zeta') = ( \exp(-\ell \eta/\zeta) \mu, -\eta/\zeta^2, 1/\zeta).$$ 
 The twistor space $Z_0$ of the Taub-NUT manifold $X_0$ is the sub-bundle of conics $\xi \psi=\eta$ in the bundle $L^\ell(1) \oplus L^{-\ell}(1)$ over $\bO(2)$, with $\xi, \psi$ the tautological sections of  $L^\ell(1),L^{-\ell}(1)$, see \cite{Hitchin79}. One has projections:
\begin{align} 
Z_0&\rightarrow \bO(2) \rightarrow \bP^1, \\
( [\xi,\psi], \eta,\zeta)&\mapsto (\eta,\zeta)\mapsto\nonumber \zeta.
\end{align}
Over $\eta\neq 0$, the fibre of the first map is a $\bC^*$; over $\eta=0$, the fibre is   a chain of two complex lines.

 There is a  fibrewise compactification $Z'_0$ of $Z_0$ over $\bO(2)$ that can be  obtained in three ways:

1) From $\bP(L^\ell(1)  \oplus \pO) $ by blowing up the line $\xi =0, \eta =0$. One obtains a bundle of $\bP^1$s over the complement of $\eta =0$, and over $\eta=0$, there is a bundle of chains of two $\bP^1$s intersecting at a point.
\medskip

1') From $\bP(L^{-\ell}(1)  \oplus \pO) $ by blowing up the line $\psi=0, \eta =0$. One obtains a bundle of $\bP^1$s over the complement of $\eta =0$, and over $\eta=0$, there is a bundle of chains of two $\bP^1$s intersecting at a point.
\medskip

2) As a family of quadrics $\xi \psi= \eta  z^2$ in $\bP(L^\ell(1) \oplus L^{-\ell} (1) \oplus \pO)$ over $\bO(2)$.
\medskip

These are isomorphic; for example, one can get from 1) to 1') by the map 
\begin{align} 
\bP(L^\ell(1)  \oplus \pO) &\rightarrow \bP( L^{-\ell}(1) \oplus \pO)= \bP(\pO(2) \oplus L^{\ell}(1)),\nonumber\\
 ( [\xi,a], \eta,\zeta)&\mapsto (  [ \psi, \hat a], \eta,\zeta))= ([\eta a , \xi], \eta,\zeta)).\nonumber 
 \end{align}

We now have two spaces $Z'_0$ and $ Z_0$ over $\bO(2)$. The complement $ {Z'_0}\backslash Z_0$ is a  union of two divisors $\Gamma_0$, defined by $\psi^{-1} = 0$, and $\Gamma_\infty$, defined by $\xi^{-1} = 0$; each of these divisors maps isomorphically to $\bO(2)$. The divisor defined by $\eta=0$ is the union of two $\bP^1$-bundles $\Delta_\psi, \Delta_\xi$ over the $\zeta$-line, defined respectively on $Z_0$  by $\psi = 0, \xi =0.$ Note, that on $Z_0'$, $\Delta_\psi + \Gamma_\infty$ is linearly equivalent to $\Gamma_0$; likewise, $\Delta_\xi + \Gamma_0$ is linearly equivalent to $\Gamma_\infty.$

\subsubsection{Fixing a complex structure} \label{twist}
The twistor space $Z_0$ of the Taub-NUT is diffeomorphic to $X_0\times \bP^1$.  Let us now fix a complex structure, say $\zeta = 0$, i.e. restrict our attention to $X_0 = \{\zeta= 0\} \subset Z_0$. This amounts to  fixing the vertical direction in $\bR^3$ corresponding to $\zeta=0$. 

Choosing such a direction gives the parallel family of lines $\bR_\eta=\{ (t_1,t_2,t_3)\in \bR^3| \eta= t_1+\i t_2\}$, and, over them in $X_0= \bR^4=\mathrm{TN}\rightarrow\mathbb{R}^3$, for $\eta\neq 0$, the family of cylinders (conics) $\bC^*_\eta$ given in complex coordinates $(\xi,\eta)\in \bC^2=\bR^4$ by $\eta = \xi\psi$. For $\eta = 0$, the fibre is an intersecting pair  of two complex lines.   
 
On the partial compactification $X'_0$, there are two divisors, which we  denote
$C_0$ and $C_\infty$, the proper transforms of the divisors $\xi =0$ and $\xi =\infty$. The fibre over $\eta = 0$ is two  projective lines  $D_\xi$ (which intersects $C_0$)  and $D_\psi$ (which intersects $C_\infty$).  For $\eta\neq 0$, the curves $\bC^*_\eta$ extend to projective lines $B_\eta$.

This can be compactified to a surface $X$. The way we choose to do this is to  add on a chain of two projective lines $F_\xi\cup F_\psi$ over the point $\eta =\infty.$ Alternately, to obtain the same surface $X$, take the surface $\bP^1\times \bP^1$, with coordinates $\eta, \xi$, and blow up two points  $(0,0)$ and $(\infty,\infty)$.

We then have a hexagon of $-1$ curves: 
\begin{itemize}
\item $D_\psi$: $\eta =0$, $\psi = 0$, $\xi$ varying, self intersection $-1$;
\item $D_\xi$: $\eta =0$, $\xi = 0$, $\psi$ varying, self intersection $-1$;
\item $C_0$: $\eta$ varying, $\xi = 0$, $\psi = \infty$, self intersection $-1$;
\item $F_\psi$: $\eta= \infty$, $\xi$ varying, $\psi =\infty$, self intersection $-1$;
\item $F_\xi$: $\eta= \infty$, $\psi$ varying, $\xi =\infty$, self intersection $-1$;
\item $C_\infty$: $\eta$ varying, $\xi = \infty$, $\psi=0$, self intersection $-1$.
\end{itemize}
\begin{figure}[htb]
\begin{center}
    \includegraphics[width=0.35\textwidth]{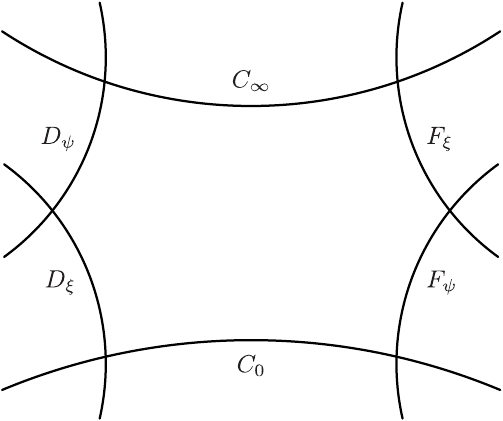}
\label{Inter}
\end{center}
\end{figure}
These curves intersect each other in a cycle, in the order given, with multiplicities one. The complement of $X_0$ in $X$ is the union
$C_0\cup F_\xi  \cup F_\psi \cup C_\infty$. The divisor ot $\eta$ is $D_\psi+D_\xi - F_\psi-F_\xi$; that of $\xi$ is $D_\xi-F_\xi+C_0-C_\infty$; and that of $\psi$ is $D_\psi-F_\psi-C_0+C_\infty$.

\subsection{Instantons on the Taub-NUT}\label{Sec:ITN}

\subsubsection{Charges and degrees}

We consider $SU(2)$ instantons on the Taub-NUT. These are $SU(2)$ bundles on $X_0$, equipped with  an  $SU(2)$ connection $\nabla$, whose curvature has finite $L^2$ norm and  satisfies the anti-self-duality equation $*F=-F$. We have one additional assumption that there is a ray in the base $\mathbb{R}^3$ of the fibration $X_0\rightarrow\mathbb{R}^3$ along which the holonomy of the connection $\nabla$ around the circle fibre is asymptotically generic, i.e. with distinct eigenvalues at least $\epsilon>0$ apart.  As proved in \cite[Thm.22]{Cherkis:2016gmo}, this implies that the asymptotic holonomy around the circle fibre is the same in all directions, up to conjugation, with its eigenvalues
\begin{align}
&\exp\left(\mp 2\pi\i\mu(\vec{t}\,) \right);&
&\text{with}&
\mu(\vec{t}\,) &=\frac{\hat{\lambda} }{\ell}-\frac{\hat{\rho}/\ell}{2|\vec{t}\,|}+O(|\vec{t}\,|^{-2}),
\end{align}
and some real constants $\hat{\lambda}$ and $\hat{\rho}.$ 
Note, that according to this the quantity $\hat\lambda/\ell$ is defined up to an integer.  
However, the next order term $\hat{\rho}/(2\ell{|\vec{t}\,|})$ is a invariantly defined.     Furthermore, $\hat{m}:=\hat{\rho}- \hat{\lambda} /\ell$ is an integer, in essence a first Chern class of the line bundle $\pO(\hat{m})$ one gets on a two-sphere near infinity by parametrizing the bundle over the fibers by sections in which the $\theta$-component of the connection is $i\mu$. This depends on the $\mu$ chosen, so that $\lambda/\ell, m$ are only defined up to integer shifts:
  $(\lambda/\ell, m)\mapsto (\lambda/\ell +N, m -N)$.
  (Note, that for the Hopf fibration $\pi:S^3\rightarrow S^2$ the lift $\pi^*(\mathcal{O}(m))\rightarrow S^3$ of any line bundle $\mathcal{O}(m)$ on the two sphere is trivial. Another way of saying this is that any trivialization of the line bundle over $S^3$ along each fiber of the Hopf fibration represents that bundle as a pullback $\pi^*(\mathcal{O}(m)).$ 
Changing these trivializations shifts the value of $m$ by an integer.)
We normalise by \footnote{
Here $\{a\}=a-\lfloor a\rfloor$ signifies the difference between $a$ and the floor of $a$, i.e. 
the largest integer $\lfloor a\rfloor$ not exceeding $a$.} 
$\lambda :=\ell\{\hat{\rho} \}$ and $m :=\lfloor\hat{\rho}\rfloor.$  Our genericity condition ensures that   $\lambda/\ell$ is neither an integer, nor a half integer. If $\lambda/\ell<1/2$, we stop. If $\lambda/\ell>1/2$, we shift further by $\lambda/\ell\mapsto \lambda/\ell-1$, and so $-\lambda/\ell\mapsto -\lambda/\ell +1$. There is now one asymptotic eigenvalue of the connection between $0$ and $1/2$, and it is that one that we call $\lambda/\ell$.
  
 Using the trivialization corresponding to this choice, the asymptotic form of the connection \cite[Thm. 22]{Cherkis:2016gmo} is 

$$\nabla_\theta = \frac{d}{d\theta} + \i \ \frac{\lambda -\frac{m }{2|\vec{t}\,|}}{\ell+\frac{1}{2|\vec{t}\,|}} {\rm diag}(-1,1)  + O(|\vec{t}\,|^{-2}),$$
with $\lambda/\ell\in(0,\frac12)$ and $m\in \bZ$ are now uniquely defined. The integer $m$ will be called the {\it magnetic charge}.

The process of going to this asymptotic  trivialization is of interest in itself.  Its various ambiguities, their physics interpretation, and associated topological invariants are discussed from the string theory point of view in \cite{Witten09}. Let us begin by noting that our bundle over $\bR^4$ is trivial, since $\mathbb{R}^4$ is contractible. The eigenlines of the holonomy over the sphere near infinity have distinct  eigenvalues, tending to constants at infinity; choosing one of them gives a map $S^3\rightarrow \bP^1(\bC)$,
 i.e. an element of $\pi_3(S^2) = \bZ$, as well as determining an eigenbasis. Under the Hopf map $S^3 = SU(2)\rightarrow S^2= \bP^1(\bC)$, one has an isomorphism $\pi_3(S^3) = \pi_3(S^2)$, and so this passage to an eigenbasis (a diagonal basis) can be made effective, as a map from  (a neighbourhood of) the three-sphere  at infinity to $SU(2) = S^3$. We thus have a natural topological invariant $n_0$, which we will call the {\it instanton number}, given by this element of $\pi_3(S^3)$. We note that further transformations, in an already diagonal basis, are essentially achieved by maps $S^3\rightarrow S^1$, and so are homotopically trivial. Thus the instanton charge $n_0$ remains invariant under these gauge transformations and is, therefore, well defined.
  
\subsubsection{Abelian solutions and compactification} 
We note that, unlike flat space, the Taub-NUT comes equipped with a family of $U(1)$ instantons which are not flat. One has a basic instanton with a globally defined connection one-form
$$a_s =  s \frac{d\theta +\omega}{V},$$
for $s\in \bR$. Again, choices of trivialization  at infinity allow integer shifts in $s$, so that one can  enforce $s\in (-\ell/2,\ell/2]$, while introducing   a magnetic charge, gives a family of connections of the form:
 $$ a_{s,m} = (s - \frac{m}{2 |\vec{t}\,|}) \frac{d\theta +\omega}{V} +m \omega.$$
(Note that $a_s = a_{s,0}$.) Now let us consider the   complex structure in the hyperk\"ahler family of the Taub-NUT, given as the surface $X_0= \bC^2$, with coordinates $(\xi,\psi)$. As we saw, $X_0$ was fibered by a family of curves $L_\eta$ over $\bC$ given by $\xi\psi = \eta$; for $\eta\neq 0$ this is a cylinder $\bC^*$, and for $\eta= 0$ this is a union of two lines $\xi=0$ and $\psi = 0$. These compactify in $X_0'$ by adding two points to each fibre, corresponding respectively to   limits $\xi\rightarrow 0$, $\psi\rightarrow 0$.

A  connection $ a_{s,m} $ defines an  integrable $\overline\partial$-operator, and so one gets from it a  holomorphic bundle over $X_0$; we want to see how it  extends to $X_0'$. Let us consider what happens near $\xi = 0$. One does this, following, e.g., Biquard \cite{Biquard97}. The leading asymptotic of this unitary connection on $\xi\neq 0$ has the form $(s/\ell) d\theta$, where $\theta = \arg(\xi)$; one can pass to a holomorphic gauge by a gauge transformation $g = (\xi\overline\xi)^{s/2\ell}$, which we use as a clutching function to extend the bundle to $\xi = 0$.  At $\psi = 0$ (and so, generically, at $\xi = \infty$), the clutching function, in a similar fashion is $g = (\xi\overline\xi)^{-s/2\ell}$. The end result is a   line bundle over $X_0'$. The result is trivial over each $B_\eta$ for $\eta$ large, as the induced curvature is small; it is  then trivial over each line $B_\eta,\eta\neq 0$.  We then extend to $X$ in the following fashion: we extend from $C_\infty\cap X_0' = \bC $ to $ C_\infty=\bP^1$ as a trivial bundle, and choose a trivialization; taking this trivialization, we can glue in a trivial bundle given on a neighbourhood of the fibres $F_\psi\cup F_\xi$ at infinity, and so get our bundle on all of $X$. As in \cite[Sec.~5]{Cherkis:2016gmo},   and as we have degree 0 on $C_\infty, F_\psi,F_\xi$, this means that we have degree $-m$ along $C_0$.
Line bundles on $X$, a rational surface, form a discrete set, and are classified by their first Chern classes; if one is looking for a line bundle with degree 0 on $B_\eta,\eta\neq 0 , F_\psi, F_\xi$, and degree $-m$ on $C_0$, the only candidate is the line bundle $\pO(-mD_\xi)$; it has degree $ -m$ on $D_\psi$, and degree $ m$ on $D_\xi$.

\subsubsection{Compactifying the $SU(2)$ instanton}

For the $SU(2)$ instanton, we proceed as in the Abelian case, using the asymptotic trivializations in which the  components of the connection are  diagonal up to order $|\vec{t}\,|^{-1}.$  What is new here along $C_0, C_\infty$ is that the $ \lambda_\pm$ eigenspaces give holomorphic sections in the unitary trivialisations with very different growth rates as one goes into $\xi = 0$; the $\lambda_-$ eigenspace gives sections which decay, whereas the $\lambda_+$ eigenspace gives sections which blow up. It is thus only the negative eigenspace  which gives a meaningful subspace, and so the structure one has in the bundle over $C_0$ is a flag rather than a pair of subbundles. Likewise, at infinity, one has a flag, this time corresponding to the $\lambda_+$ eigenspace (of decaying solutions) consisting of a subbundle   over $C_\infty$. Compactifying to all of $X$, as for the line bundle,  one has along $C_\infty$ a trivial bundle, with a trivial subline bundle, and along $C_0$, a bundle with a subbundle of degree $-m$. Finally, we note that the bundle over $L_\eta$ for $\eta$ large is trivial, since a positive subbundle (destabilising subbundle) requires curvature, which is tending to zero.

We thus obtain a bundle over $X$, of degree zero, trivial on the generic lines $B_\eta$, as well as on $C_\infty$, and with flags (subline bundles) of degree $-m$ and $0$ along $C_0$ and  $C_\infty$ respectively. It has  a second Chern class $k$, which, one can show to be  the number of {\it jumping lines} (number of lines at which the holomorphic structure jumps, i.e. is non trivial) in the ruling $B_\eta$, counted with multiplicity. This then tells us that $k$ must be positive.

It is interesting to consider the Abelian $SU(2)$ solution 
$A_{s,m}=\left(\begin{smallmatrix}
a_{s,m}&0\\
0&-a_{s,m}
\end{smallmatrix}\right)
$ 
given in the complex world after compactification as a sum $\pO(mD_\psi)\oplus \pO(-mD_\psi)$ of two $U(1)$ instantons of opposite magnetic charge. This has second Chern class $m^2$; on the other hand, we had already had an instanton number  $n_0$ for $A_{s,m}$, given as   the  degree of the gauge transformation on the infinity of $\bR^4$ that  makes the connection `Abelian near infinity', as the connection is already Abelian,  $n_0(A_{s,m})= 0$.   
 
 \subsection{The bow}
    \begin{figure}[htb]
\begin{center}
    \includegraphics[width=0.45\textwidth]{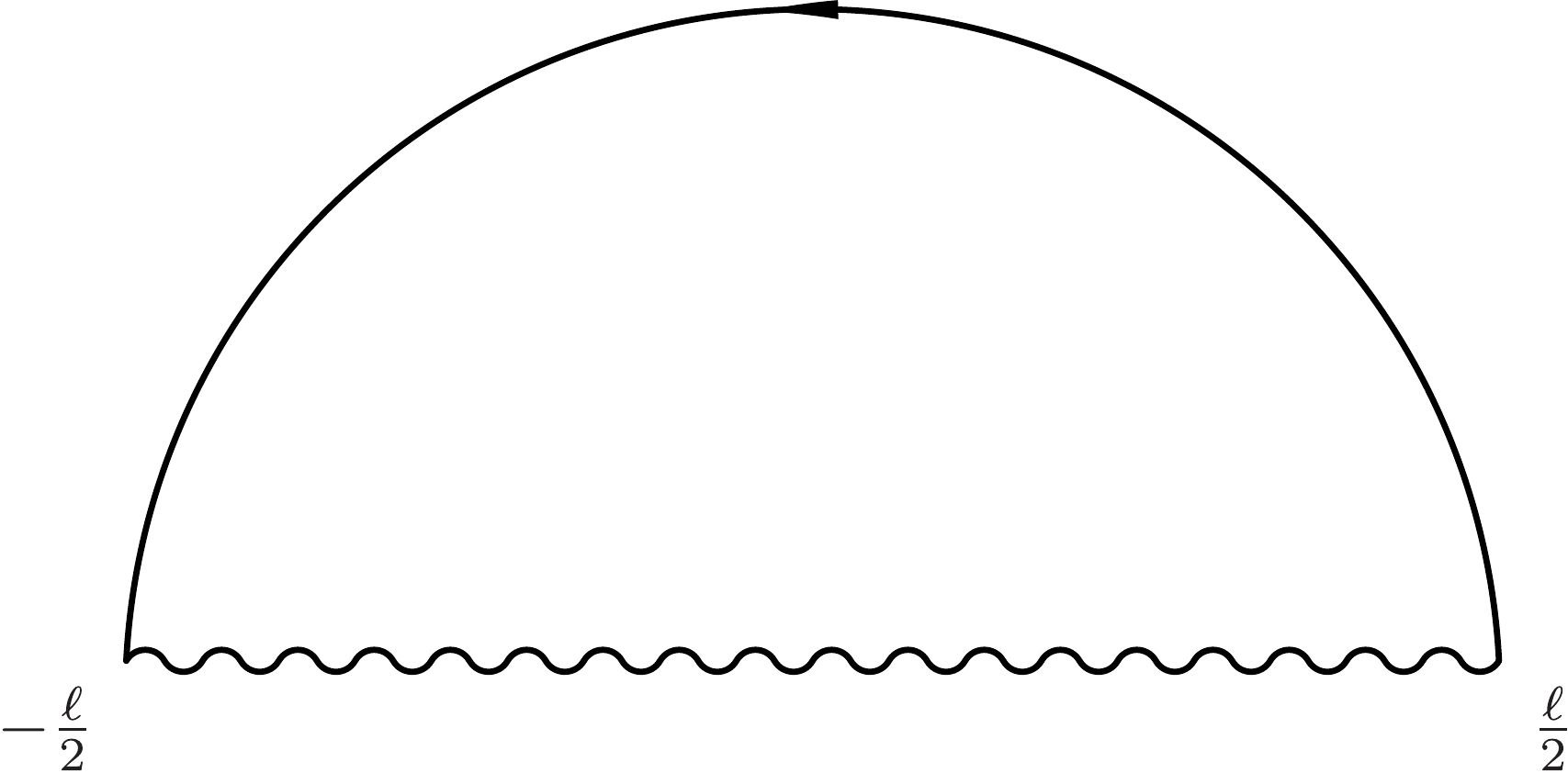}
\caption{The Taub-NUT bow.}
\label{Fig:BowTN}
\end{center}
\end{figure}  
 Bows are generalizations of quivers \cite{Cherkis:2008ip} consisting of oriented intervals and oriented edges, each edge beginning at the end of some interval and ending at the beginning of some (possibly the same) interval. The bow 
relevant for the Taub-NUT space of Eq.~\eqref{TNmetric} consists of a single interval $\left[-\frac{\ell}{2},\frac{\ell}{2}\right]$ parameterized by $s$ and a single edge connecting its ends, with its head $h$ at $s=-\frac{\ell}{2}$ and its tail $t$ at $s=\frac{\ell}{2}$, see Fig.~\ref{Fig:BowTN}.

Any {\em bow representation} \cite{Cherkis:2010bn} corresponding to the $SU(2)$ instanton, studied above, consists of 
\begin{itemize}
\item
two points at $s=\lambda_-:=-\lambda$ and at $s=\lambda_+:=\lambda>0$ on the bow interval, 
\item 
Hermitian bundles $N_0\rightarrow[-\ell/2,\lambda_-]\cup[\lambda_+,\ell/2]$ and $N_1\rightarrow[\lambda_-,\lambda_+]$ of respective ranks $k$ and $k+m,$
\item
chosen Hermitian injections $i_-: N_0|_{\lambda_-}\rightarrow N_1|_{\lambda_-}$ and $i_+: N_0|_{\lambda_+}\rightarrow N_1|_{\lambda_+}$, for $m\geq 0$, or $i_-: N_1|_{\lambda_-}\rightarrow N_0|_{\lambda_-}$ and $i_+: N_1|_{\lambda_+}\rightarrow N_0|_{\lambda_+}$, for $m\leq 0$;
\item
if $m=0,$ one-dimensional\footnote{In general, for higher rank instanton structure group, the dimension of $W_\lambda$ equals to the multiplicity of the corresponding $\lambda$ eigenvalue of the holonomy at infinity.} auxiliary Hermitian spaces $W_{\lambda_-}$ and $W_{\lambda_+}$.
\end{itemize}
\begin{figure}[htb]
\begin{center}
    \includegraphics[width=0.45\textwidth]{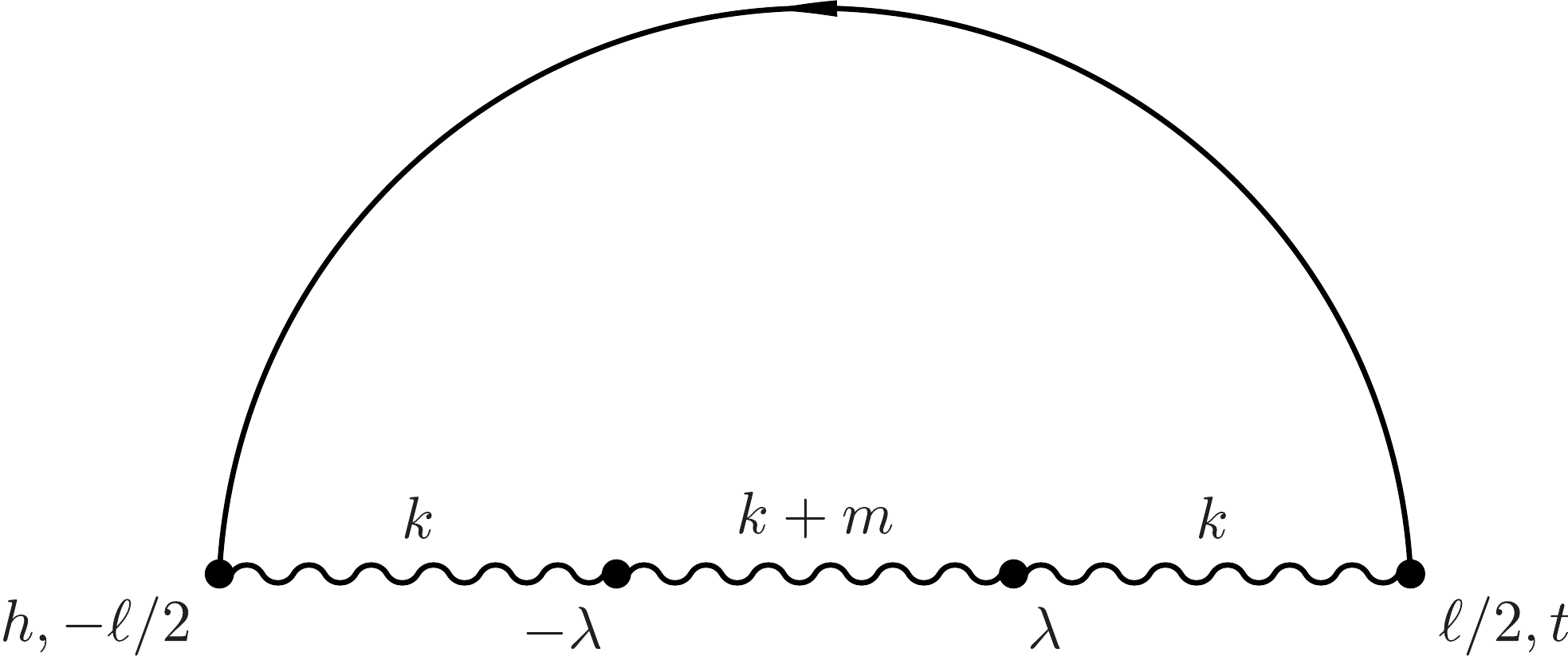}
\caption{Bow representation for the $SU(2)$ instanton on the Taub-NUT.}
\label{Fig:Bow}
\end{center}
\end{figure}

The bow representation has an affine space of {\em bow data} associated with it; this space is a direct sum of 
\begin{description}
\item[Nahm data] associated with the interval.  It is an affine space of quadruplets $(\nabla_s, T_1,T_2,T_3)$ consisting of a connection $\nabla_s$ and three endomorphisms $T_1,T_2,$ and $T_3$ on the Hermitian bundles $N_0$ and $N_1,$
\item[Bifundamental data] associated with the edge, is a vector space of pairs $(B_{t,h},B_{h,t})$ of linear maps 
\begin{align*} 
B_{t,h} : &N_0|_{-l/2}\rightarrow N_0|_{l/2},& 
B_{h,t}:&N_0|_{l/2} \rightarrow N_0|_{-l/2}.
\end{align*}
\item[Fundamental data] associated with each $\lambda$-point $s=\lambda_\pm$ is present only if $m=0.$ It is a vector space of pairs 
$(I_\pm,J_\pm)$ of maps 
\begin{align*}
I_+: &W_{\lambda_+}\rightarrow N_0|_{\lambda_+}=N_1|_{\lambda_+},&
J_+: &N_0|_{\lambda_+}=N_1|_{\lambda_+}\rightarrow W_{\lambda_+},\\
I_-: &W_{\lambda_-}\rightarrow N_0|_{\lambda_-}=N_1|_{\lambda_-},&
J_-: &N_0|_{\lambda_-}=N_1|_{\lambda_-}\rightarrow W_{\lambda_-}.
\end{align*}
\end{description}
The Nahm data satisfy certain analytic conditions as well as matching conditions spelled out in \cite{Second}.   Among all bow data of a given representation of particular importance are the {\em bow solutions}.
  These can be introduced via the hyperk\"ahler reduction (by the action of the gauge group) as the data satisfying the moment map conditions. Equivalently, the data has a Dirac-type operator associated to it and the solution is the data satisfying the  condition of reality of the square of the associated bow Dirac operator \cite{Cherkis:2008ip,Cherkis:2010bn}.  Here, we simply state  these conditions.  These will involve\footnote{
  These conditions can be obtained as a result of the direct `Down transform' of 
\cite{Third}, however, here we take an alternative approach and prove that the monad arising from the bow solution corresponds to that of the instanton.}
\begin{enumerate}
\item
the conditions in the interior of the intervals, given by the Nahm equations,
\item  
the matching conditions at the $\lambda$-points, and
\item 
the conditions at the ends $s=\pm\ell/2$ of the bow interval.  
\end{enumerate}

 \subsubsection{Nahm's equations}

Nahm's equations are ordinary differential equations, discovered by Nahm in his early work on monopoles \cite{Nahm1983}. They relate three functions $T_i(s), i=1,2,3,$ on the line with values in the Hermitian $n\times n$ matrices by 
$\i\frac{dT_1}{ds} = [T_2,T_3],$ 
$\i\frac{dT_2}{ds} = [T_3,T_1],$ 
$\i\frac{dT_3}{ds} = [T_1,T_2].$
One can put in a gauge freedom by replacing the derivative $\frac{d}{ds}$ with a $u(n)$  covariant  derivative
$\nabla$; the equations become 
\begin{equation} \i\nabla T_i =\sum_{j,k} \epsilon_{ijk} T_j T_k.\end{equation}
These are reductions to one dimension of the anti-self-duality equations on $\bR^4$.

One can rewrite these equations with a spectral parameter, which is in fact the twistor parameter $\zeta\in\mathbb{P}^1.$ Setting 
\begin{align*}
A^N(\zeta,s) &= T_1 +\i T_2 -2 T_3\zeta - (T_1-\i T_2)\zeta^2,&
M^N(\zeta,s) & =  -T_3- (T_1-\i T_2)\zeta,
\end{align*}
one has a Lax pair for $\zeta\neq\infty.$ 
Alternatively, one can consider the Lax pair $A^S:=-A^N/\zeta^2$ and $M^S:=M^N-A^N/\zeta$ for $\zeta\neq 0,$ so that
\begin{align*}
A^S(\zeta,s) &= T_1 -\i T_2 +2T_3\frac{1}{\zeta} - (T_1+\i T_2)\frac{1}{\zeta^2},&
M^S(\zeta,s) & =  T_3- (T_1+\i T_2)\frac{1}{\zeta}.
\end{align*}
In either case, the Nahm equations are equivalent to  the Lax equation
$$[\nabla+M(\zeta, s),A(\zeta,s)]  = 0.$$
The moduli of solutions to Nahm's equations encode many different moduli spaces of solutions to the anti-self-duality equations; see e.g. \cite{Jardim} for a review.  Of particular importance are the boundary conditions.  The ones we will want to study arise naturally from the bow and involve the fundamental and bifundamental data. These solutions of the Nahm equations, together with the matching fundamental and bifundamental linear maps comprise the {\em bow solutions} of \cite{Cherkis:2008ip,Cherkis:2010bn}.

\subsubsection{Boundary conditions: Fundamental and Bifundamental Data}\label{Bow}
A bow solution is a decuplet $(\nabla, T_1, T_2, T_3, B_{t,h},B_{h,t},I_-,J_-,I_+,J_+)$ satisfying the following conditions.
\medskip

\noindent {\it Nahm conditions (associated to the subintervals)} 

 \begin{itemize}
\item 
On the Hermitian bundle $N_0$ of rank $k$ over the intervals $[-l/2,  \lambda_-)\cup (\lambda_+, l/2]$, a solution $A^0(\zeta,s)=A^N(\zeta,s)$ to the Nahm equations which  is smooth.
\item 
On the Hermitian bundle $N_1$ of rank $k+m$ over the interval
$(\lambda_-, \lambda_+)$, a solution $A^1(\zeta,s)=A^N(\zeta,s)$ to the Nahm equations is smooth.
\item 
The connection matrices of $\nabla$ are smooth in the interior of both intervals and  have finite limits at the boundary points.
\end{itemize}
\bigskip

\noindent {\it Fundamental conditions (associated to the $\lambda$-points)} 
\medskip

\noindent  {\it When $m>0$:}
 \begin{itemize}
 \item At both boundary points $\lambda_\pm$, the solution $A^N(\zeta,s)$ has a one-sided limit from the `small' (rank $k$) side, and is analytic near the boundary on the `large' (rank $k+m$) side, with at most a simple pole at the boundary point.
 \item   At both boundary points $\lambda_\pm$, an injection
$i_\pm\colon N_0\rightarrow N_1$, respecting the unitary structures,  decomposes $N_1$ at the boundary points into $\mathrm{Im}(i_\pm)\oplus  \mathrm{Im}(i_\pm)^\perp$. We call $\mathrm{Im}(i_\pm)$ the `continuing' components of the solution. One asks that there be an extension of
this decomposition to a unitary trivialization on the interior of the
intervals (in the vicinity of the $\lambda$-point) such that 
\begin{equation*}
T^1_i(s)=\begin{pmatrix}a^1_i(s)                & b^1_i( s)\\
                        c^1_i(s)&-\frac{\i \rho_i}{2(s-\lambda)}+d^1_i(s)                 \end{pmatrix},
                         \end{equation*}
where $\lambda=\lambda_-$ or $\lambda_+.$ The top left blocks are $k\times k$, the bottom right block is
$m\times m$; $a^1_i,b^1_i,c^1_i, d^1_i$ are analytic at $s=\lambda$, and $\{\rho_i\}_{i=1}^3$ form a fixed  $m$-dimensional irreducible representation  of $su(2)$. 
Furthermore, the  solutions on the two intervals match on the continuing components  by 
\begin{equation*}
a^1_i(0) = T^0_i(0).
\end{equation*}
In the same way, the connection coefficients match:
\begin{equation*}
a^1_0(0) = T^0_0(0).
\end{equation*}

\item At both boundary points, some extra data, consisting of a
unitary trivialization $v_\pm$ of the highest weight space of the irreducible representation as follows.
\end{itemize}
We normalise trivializations of $N_0, N_1$ at the boundary points: we assume that the injection of $N_0$ into $N_1$ maps the basis of $N_0$ into the   first $k$ basis vectors of $N_1$ and that under this, our bases match; the unitary complement of $N_0$ in  $N_1$  is then given by the last $m$ vectors of the basis. On this, the irreducible representation in the polar part of $d_i$, plus the trivialization of the highest weight space, means that there is a unique way to trivialize it, so that one has the standard representation of $SU(2)$.  Once this is done, one has a smaller  group of gauge transformations which acts at $s=\lambda_\pm$ by the identity on the last $m$ vectors, and by $U(k)$ on the first vectors, with the group elements matching on $N_0$ and $N_1$. This smaller gauge group is the group of gauge transformations of the bow representation.
\bigskip

\noindent {\it When m=0:}
\medskip

\begin{itemize}
\item At both boundary points $\lambda_\pm$, a unitary isomorphism
$i_\pm\colon N_0\rightarrow N_1$ identifies the two fibres. One asks that the solutions $A^0, A^1$ have respective limits at the $\lambda$-point at $s=0$, where again $s$ is a local parameter with the point $\lambda_\pm$ corresponding to
$s=0.$
\item The additional condition at both boundary points $\lambda_\pm$ is a  decomposition
\begin{align}
A^1(\zeta,\lambda_-)-A^0(\zeta,\lambda_-) &= (I_- -J_-^\dagger\zeta)(J_- +I_-^\dagger\zeta),\\
A^0(\zeta,\lambda_+)-A^1(\zeta,\lambda_+) &= (I_+ -J_+^\dagger\zeta)(J_+ +I_+^\dagger\zeta).
\end{align}
\item The connection is continuous at the boundary under the identification.
\end{itemize}
Again, the fibers of the two vector bundles  at each boundary point $\lambda_\pm$ are identified by $i_\pm$, so that the gauge transformations allowed are continuous at the boundary points.
\bigskip

\noindent {\it When $m<0$:} 
\medskip

One has the same boundary behaviour as for $m>0$, but now  at each boundary point $\lambda_\pm$ the roles of the two intervals are reversed, so that one still has finiteness from the small side, which is now that where the rank is $k+m$, and, from the large side, where the rank  is $k$,  poles with residues forming an irreducible 
representation of $su(2)$ of dimension $-m$, and continuing components of rank $k+m$ which match with the small side.
\bigskip

\noindent{\it Bifundamental conditions (associated to the edges)}
\medskip

The bifundamental data consist of complex linear maps between the fibres of $N_0$ at $-l/2$ and $l/2$:
\begin{align} B_{t,h} : &N_0|_{-l/2}\rightarrow N_0|_{l/2},\\ 
B_{h,t}:&N_0|_{l/2} \rightarrow N_0|_{-l/2}.\end{align}
The notation $t$ (for the edge's `tail') refers to the point $l/2$, and $h$ (for the edge's `head') to the point $-l/2$. With this, the matrices $A^0(\zeta,l/2), A^0(\zeta,-l/2)$ are required to satisfy:
\begin{align}
A^0(\zeta,l/2) =& (B_{t,h} +\zeta B_{h,t}^\dagger)(B_{h,t} -\zeta B_{t,h}^\dagger),\\ A^0(\zeta,-l/2)=& (B_{h,t} -\zeta B_{t,h}^\dagger)(B_{t,h} +\zeta B_{h,t}^\dagger).\end{align}

\subsubsection{Bow complexes}  \label{Bow-complexes}

There is a partial, holomorphic version of the Nahm data, which following Donaldson \cite{Donaldson:1985id},
is referred to as a Nahm complex, and in our case forms a {\em bow complex}. Essentially, one restricts to $\zeta = 0$. In our case, this consists of:
\medskip

\noindent{\it Nahm data}

 \begin{itemize}
\item A complex  bundle $N_0$ of rank $k$ over the intervals $[-l/2,  \lambda_-]\cup [\lambda_+, l/2]$, equipped over
$[-l/2,  \lambda_-)\cup (\lambda_+, l/2]$,  with a connection 
$ \frac{d}{ds} + \alpha^0(s)$, and a section $\beta^0(s)$ of $\mathrm{End}(N_0)$ which is covariant constant 
$$\frac{d\beta}{ds} + [\alpha(s),\beta(s)] = 0;$$
\item A complex bundle $N_1$ of rank $k+m$ over the interval
$(\lambda_-, \lambda_+)$, equipped with a smooth complex connection
$\alpha^{1}(s)$  and a smooth covariant constant section
$\beta^{1}(s)$ of $\mathrm{End}(N_1)$ ; \end{itemize}
\bigskip

\noindent{\it Fundamental data}
\medskip

\noindent {\it When $m>0$}:
\medskip

 \begin{itemize}
 \item At the boundary points $\lambda_\pm$, the connection $\alpha^0(s)$, and a section $\beta^0(s)$ have finite limits, and from the ``large'' side, the connection $\alpha^1(s)$, and a section $\beta^1(s)$ are analytic near the boundary points, with at most a  pole of order one at the boundary point.
 \item At the boundary points $\lambda_\pm$,   injections
$i_\pm\colon N_0\rightarrow N_1$ and   surjections
$\pi_\pm\colon N_1\rightarrow N_0$, such that
$\pi_\pm i_\pm = Id$, so that one can decompose $N_1$ as $
\mathrm{Ker}(\pi_\pm)\oplus \mathrm{Im}(i_\pm)$. One asks that there be an extension of
this decomposition to a trivialization on the interior of the
interval (in the vicinity of each $\lambda$-point) such that one can write the connection $\alpha_{1}(s)$
and the endomorphism $\beta_{1}(s)$ in block form near the boundary points as
\begin{equation*}
\alpha_{1}(s)=\begin{pmatrix}U(s)                & s ^{\frac{m-1} {2}}W(s)\\
                         s^{\frac{m-1} {2}}V(s)&T(s)                 \end{pmatrix},\quad
\beta_{1}(s)=\begin{pmatrix}P(s)                & s^{\frac{m-1} {2}}Q(s)\\
                        s^{\frac{m-1} {2}}R(s)&S(s) \end{pmatrix},
\end{equation*}
where $s$ is a local parameter with the point $\lambda_\pm$ corresponding to
$s=0$. The top left blocks are $k\times k$, the bottom right block is
$m\times m$; $U,W,V,P,Q,R$ are analytic at $s=0$, and $T,S$ are meromorphic
with simple poles at $s=0$, and residues 
\begin{gather}
T_{-1}= \mathrm{diag}\Bigl({\frac{-(m-1)}{2}},{\frac{2-(m-1)}{2}},\dots,{\frac{(m-1)}{2}}\Bigr),\\
S_{-1}=\Sh:= \begin{pmatrix}0&0&0&\dots&0&0\\ 1&0&0&\dots&0&0\\0&1&0&\dots&0&0\\
\dots&\dots&\dots&&\dots&\dots\\
0&0&0&\dots&1&0 \end{pmatrix}.
\end{gather}
Furthermore,
\begin{equation*}
U(0) = \alpha^{0}(0),\quad P(0) = \beta^{0}(0).
\end{equation*}

\item At both boundary points, some extra data, consisting of a
trivialization (choice of vectors $v_-, v_+$) of the $\frac{-(m-1)}{2}$
eigenspace of $T_{-1}$.
\end{itemize}

\bigskip

\noindent {\it When $m=0$}:
 \medskip
 
 \begin{itemize}
 \item At the boundary points $\lambda_\pm$, isomorphisms
$i_\pm\colon N_{ 0}\rightarrow N_1$, $\pi_\pm= i_\pm^{-1}$
with the gluing condition that
$\beta^{0}-\pi_\pm\beta^{1}i_\pm$ has rank one at the boundary point.

\item At the boundary points $\lambda_\pm$, extra data consisting of decompositions $ (I_-, J_-)$, $  (I_+,J_+)$ of the rank one boundary difference matrices  $\beta^{ 0}-\pi_-\beta^{1}i_-$, $\beta^{0}-\pi_+\beta^{1}i_+$ into products of a column and a row vector:
\begin{equation}
\beta^{0}-\pi_-\beta^{1}i_-= I_-\cdot J_-,\quad \beta^{0}-\pi_+\beta^{1}i_+ = I_+\cdot J_+\end{equation}
\item The connection $\alpha$ is continuous at the boundary under the identification.
\end{itemize}

\bigskip

\noindent {\it When $m<0$}:
 \medskip
 
  One has the same boundary behaviour as for $m>0$, but now with at each boundary point $\lambda_\pm$, the roles of the `small' and `large' intervals interchanged in the above above.

\bigskip

\noindent{\it Bifundamental data}
\medskip

The edge data consist of complex linear maps between the fibres of $N_0$ at $-l/2, l/2$:
\begin{align} 
B_{t,h} : N_0|_{-l/2} & \rightarrow N_0|_{l/2},\\ 
B_{h,t}: N_0|_{l/2} & \rightarrow N_0|_{-l/2}.
\end{align}
There are decompositions
\begin{equation}
\beta(-l/2) = B_{h,t}B_{t,h},\quad \beta( l/2) = B_{t,h}B_{h,t}.
\end{equation}

 All of this data is of course to be considered modulo complexified  gauge transformations; again, we can normalise the bases at the boundary points so that $N_0$ injects into $N_1$ as the first $k$ vectors, that there is a well defined complementary space, with a fixed trivialisation on it (exploiting the cyclicity of the residue of $\beta$); the gauge transformations at these boundary points act by $Gl(k,\mathbb{C})$ on the first $k$ vectors, and trivially on the others. 
 
 \subsection{The Bow Monad}
 The bow monad described below is directly related to the Up transform of \cite{Second} via the {\em bow Dirac operator}.  Each bow solution  
 $$(\nabla, T_1, T_2, T_3, B_{t,h},B_{h,t},I_-,J_-,I_+,J_+),$$ 
 gives rise to a family of Dirac-type operators
\begin{multline}
{\mathbf D}_t^\dagger=
\begin{pmatrix}
-D^\dagger & Z^\dagger\\
Z & D
\end{pmatrix}
+\delta(s-\lambda_-)\begin{pmatrix} J_{-}^\dagger \\ I_{-}\end{pmatrix}
+\delta(s-\lambda_+)\begin{pmatrix} J_{+}^\dagger \\ I_{+}\end{pmatrix}\nonumber\\
+\left(\delta(s-h)\begin{pmatrix}\bar{b}_{ht} & B_{th}^\dagger \\  -b_{th} & B_{ht} \end{pmatrix}
+\delta(s-t)\begin{pmatrix}  B_{ht}^\dagger & -\bar{b}_{th} \\ 
 -B_{th} & - b_{ht} \end{pmatrix}\right),
\end{multline}
with $Z:=T_1+\i T_2-t_1-\i t_2$  and $D=\frac{d}{ds}-\i(T_0-t_0)+T_3-t_3.$ 
This family is parameterized by a point on the Taub-NUT with  coordinates $(t_1+\i t_2=\xi\psi,t_3=(|\psi|^2-|\xi|^2)/2,b_{ht}=\xi,b_{ht}=\psi).$
We denote the pair $N_0, N_1$ of bundles with isometric injections by $N$. We also regard $t_1,t_2,$ and $t_3$ as ($s$-independent) endomorphisms of Hermitian line bundle $e$ over the bow (with a fixed trivialization), with $b_{t,h}:e_h\rightarrow e_t$ and $b_{h,t}:e_t\rightarrow e_h.$ 
This Dirac-type operator acts on the direct sum $\Gamma_{L^2}(S\otimes N\otimes e^*)\oplus W_{\lambda_-}\oplus W_{\lambda_+}\oplus N_h\otimes e_t^*\oplus N_t\otimes e_h^*$ of
\begin{itemize}
\item
 $L^2$ sections of the bundle $S\otimes N\otimes e^*$ that are continuous in the continuing components at the $\lambda$-points.  
\item
the auxiliary spaces $W_{\lambda_-}$ and $W_{\lambda_+}$ (these are present only if $m=0$). 
\end{itemize}
The main properties of the operators of this family is that ${\mathbf D}_t^\dagger {\mathbf D}_t$ is proportional to the identity in the $S$ factor and  ${\mathbf D}_t^\dagger {\mathbf D}_t$ is strictly positive (away from the codimension at least two isolated strata in the Taub-NUT space).   Thus, $\mathrm{Ker}\,  {\mathbf D}_t=0.$ The crux of the Up transform is that the instanton bundle $E\rightarrow X_0$ emerges as the index bundle  of this family, i.e. $E=\mathrm{Ker}\,{\mathbf D}_t^\dagger.$

This form of the bow Dirac operator is amenable to the Hodge decomposition $\mathbf{D}_t^\dagger=
\left(\begin{smallmatrix} 
-\delta_0^\dagger  \\  
\delta_1
\end{smallmatrix} \right),
$ with
\begin{align}
\delta_0&=
\begin{pmatrix}
D\\
-Z\\
-J\\
\begin{smallmatrix} 
-b_{ht} & -B_{ht}\\
-B_{th} & b_{th}  
\end{smallmatrix} 
\end{pmatrix}:
\Gamma_{H^1}(N\otimes e^*)\rightarrow 
\Gamma_{L^2}(N\otimes e^*)^{\times 2}\oplus W\oplus (N_h\otimes e_t^*
\oplus N_t\otimes e_h^*),
\\
\delta_1&=(Z, D)+\delta(s-\lambda) I+[\delta(s-h)(-b_{th},B_{ht})-\delta(s-t)(B_{th},b_{ht})]:\nonumber\\
&\Gamma_{L^2}(N\otimes e^*)^{\times 2}\oplus W\oplus (N_h\otimes e_t^* \oplus N_t\otimes e_h^*)\rightarrow\Gamma_{H^{-1}}(N\otimes e^*).
\end{align}
The fact that we began with a solution of a bow representation ensures that $\mathbf{D}_t^\dagger\mathbf{D}_t$ is real (i.e. commutes with the action of quaternionic identities on the $S$ factor) and strictly positive. This, in turn, implies $\delta_1\delta_0=0$ and $\delta_0^\dagger\delta_0=\frac{1}{2}(\delta_0^\dagger\delta_0+\delta_1\delta_1^\dagger)>0$.  Thus, $\mathrm{Ker}\, \mathbf{D}_t^\dagger$ can be identified with the middle cohomology of the Dolbeault complex
\begin{align}\label{FirstBowMonad}
0\rightarrow A\xrightarrow{\delta_0}B\xrightarrow{\delta_1}C\rightarrow 0.
\end{align}
Moreover, since $\delta_0^\dagger\delta_0=\delta_1\delta_1^\dagger=\frac{1}{2}(\delta_0^\dagger\delta_0+\delta_1\delta_1^\dagger)>0,$ $\mathrm{Ker} \delta_0=0=\mathrm{Cok} \delta_1,$ so this complex is exact in its first and last terms. This is the infinite-dimensional monad construction of the bundle $E:=\mathrm{Ker} \mathbf{D}_t^\dagger\rightarrow\mathrm{TN}_k$.

\subsubsection{Twistorial Bow Monad}
The above discussion relates the Dolbeault complex to the Dirac operator of the Up transform.  It is, however, confined to a particular choice $\zeta=0$ of a complex structure on the Taub-NUT space.  In order to make our discussion twistorial, we have to extend it to the full twistor sphere.  To achieve this consider the Hodge  decomposition
\begin{align}
\left(\begin{array}{cc} -1 & -\bar{\zeta}\\ -\zeta & 1 \end{array}\right){\mathbf D}^\dagger_t=\begin{pmatrix} -\delta_0^\dagger \\ \delta_1\end{pmatrix},
\end{align}
now, with $\delta_0$ and $\delta_1$ depend on $\zeta.$
We consider the resulting $\zeta$-dependent Dolbeault complex:
\begin{equation}\label{Eq:Complex}
{\cal C}_\zeta: 0\rightarrow {\cal A}^0\xrightarrow{\delta_0}{\cal A}^1\xrightarrow{\delta_1}{\cal A}^2\rightarrow 0,
\end{equation}
with 
${\cal A}^0=\Gamma(S\otimes N\otimes e^*), {\cal A}^1=\Gamma(S\otimes N\otimes e^*)\oplus (N_t\otimes e_L^*)\oplus (N_h\otimes e_R^*)\oplus W,$ and ${\cal A}^2=\Gamma'(S\otimes N\otimes e^*).$ The latter space consists of distributions of the form $f(s)+\sum_{\lambda\in\Lambda_0}\delta(s-\lambda) a_\lambda.$

Now, the differentials $\delta_0$ and $\delta_1$ are first order in $\zeta,$ while the spaces ${\cal A}^0, {\cal A}^1,$ and ${\cal A}^2$ are (at this stage) still $\zeta$-independent.

Our remaining goal is to find a finite-dimensional versions of the above monads \eqref{FirstBowMonad} and \eqref{Eq:Complex} and compare them with those directly associated with the instanton.
 
\section{Holomorphic Geometry of Calorons}
 
 Our complex space $X$ is a blow-up of $\bP^1\times\bP^1$. The space $\bP^1\times\bP^1$ in turn, is a compactification of $\bR^3\times S^1$, and instantons on this space, with boundary conditions similar to those for the Taub-NUT, called calorons, are closely related to instantons on the Taub-NUT. In both cases, we have, asymptotically at least, circle bundles, though with one being a Hopf bundle and the other a trivial bundle.  We will see that on the holomorphic level, our instantons on the Taub-NUT are in some sense and in a first approximation obtained by  identifying two calorons over an open set.  
 
 Likewise, from the Nahm's equations point of view, the solutions associated to the Taub-NUT are very similar to the caloron solutions; the main distinction being that the caloron solutions do not have the bifundamental data, but are simply solutions on the circle.
 
 As a warm-up (and because we will be using the geometry of the caloron bundles later on) we discuss these  solutions to the ASD equations  on the direct  product $\bR^3\times S^1$ with flat metric $\ell d\vec{t}^2+d\theta^2/\ell$ (instead of a Taub-NUT). This will allow us to recapitulate some relevant material, in particular from  \cite{Charbonneau:2006gu, Charbonneau:2007zd}.  It will give us some insight into a relatively simpler case, allowing a certain simplification of the more roundabout approach of \cite{Charbonneau:2006gu}, where the problem is studied in detail.  After this warmup, we will return to the Taub-NUT case in the next section, highlighting the differences, and in particular, for the bow side of the picture, incorporating the bifundamental data.
 
 \subsection{The bundle on twistor space}
 
As for the Taub-NUT, in the present case of the caloron, we  just discuss the $SU(2)$ case.  For calorons, we have a twistor space $\bT$ that is a $\bC^*$ bundle over the total space $\bO(2)$ of the line bundle $\pO(2)$ over $\bP^1$. Concretely, if $\zeta$ is the natural holomorphic parameter for the projective line, we can cover $\bO(2)$ by two open sets $U_0= \{\zeta\neq\infty\}$ and $U_1=  \{\zeta\neq 0\}$; one has then on $U_0$ coordinates $(\eta,\zeta)$ and on $U_1$ coordinates $(\eta',\zeta')$ related by 
$$( \eta',\zeta') = (-\eta/\zeta^2, 1/\zeta)$$ 
on the overlap.  If $\ell$ is a positive real constant (with $\sqrt{\ell}$  reciprocal of the radius of the circle), the bundle $\pi: \bT\rightarrow \bO(2)$  is the complement of the zero section of the line bundle $L^{\ell}, $ with exponential transition function $\exp(-\ell\eta/\zeta)$, so that one has coordinate patches $V_0= \{\zeta\neq\infty\}$ and   $V_1=  \{\zeta\neq 0\}$ on $\bT$, each isometric to $\mathbb{C}^*\times\mathbb{C}\times\mathbb{C}$, with  coordinates $(\xi, \eta,\zeta)$ and  $(\xi',\eta',\zeta')$ related by
$$(\xi', \eta',\zeta') = ( \exp(-\ell \eta/\zeta) \xi, -\eta/\zeta^2, 1/\zeta).$$ 

The  complex structures on $\bR^3\times S^1$ are parametrized by $\zeta\in \bP^1$. 
Now consider an $SU(2)$ instanton bundle on $\bR^3\times S^1$, equipped with a chosen framing at infinity in, say, the positive $t^3$-direction in $\bR^3\times S^1$, with instanton charge $k$ and monopole charge $m$. The instanton, since it has anti-self-dual curvature, gives an integrable holomorphic structure for the fibre over each $\zeta\in \bP^1$, and globally, a holomorphic bundle $\pE$ on the twistor space; this bundle is equipped with a real involution lifting a natural real involution on $\bT$.  Following, e.g., Biquard \cite{Biquard91,Biquard97}, and as we shall see for the Taub-NUT, the boundary conditions allow an extension of the holomorphic bundle $\pE$ to a  (fibrewise over $\bO(2)$, and so partial) compactification  $\overline{\bT} = \bP(L^{\ell}\oplus \pO)\rightarrow \bO(2)$ given by adding  two natural divisors $\Gamma_0=\{\xi=0\}$, $\Gamma_\infty=\{\xi=\infty\}$, as in \cite{Charbonneau:2006gu}.  As we shall see below, the bundle $\pE$  over the compactification then has  a flag of subbundles $0=\pE^0_0\subset \pE^0_1\subset \pE^0_2 =\pE$ over  $\Gamma_0$ and $0=\pE_\infty^0\subset \pE_\infty^1\subset \pE_\infty^2 =\pE$ over  $\Gamma_\infty$; here the index $i$ in $\pE^{\cdot}_i$ denotes the rank of $\pE^{\cdot}_i$. This compactification follows, in essence, from the work of Biquard \cite{Biquard91}. Define sheaves of meromorphic sections 
\begin{equation} 
\pE_{p,q}^{s,t}=\left\{\sigma\Big|
\begin{array}{c}\xi^{ p}\sigma\ {\rm finite\ at}\ \Gamma_0\ {\rm with\ values\ in}\ \pE^0_q,\\ \xi^{-s} \sigma\ {\rm finite\ at}\ \Gamma_\infty\ {\rm with\ values\ in}\ \pE_\infty^t
\end{array}
\right\}.
\end{equation}

We then look at  the bundle $\pF$ on $\bO(2)$ obtained as the direct image of the bundle $\pE$,  projecting from the $\bC^*$ bundle $\bT$ over $\bO(2)$ as opposed to the $\bP^1$-bundle; $\pF$ is of infinite rank. One   can use the
flags $\pE^0_q $ along $\Gamma_0$, $\pE_\infty^n$ along $\Gamma_\infty$
to define for $p\in \bZ$ and $q= 0,1$ subbundles $\pF^0_{p,q},
\pF_\infty^{m,n}$ of $\pF$ as
\begin{align*}
\pF^0_{ p,q} &=\{\sigma\in \pF\mid \xi^{p}\sigma\text{ finite at }{C}_0
                         \text{ with value in } \pE^0_q  \}, \\
\pF_\infty^{s,t} &=\{\sigma\in \pF\mid \xi^{-s}\sigma \text{ finite at } {C}_\infty
                           \text{ with value in } \pE_\infty^t \}.
\end{align*}
We now have infinite flags
\begin{equation}\label{infiniteflags}\begin{gathered}
\cdots\subset \pF^0_{-1,0}\subset \pF^0_{-1,1} \subset \pF^0_{0,0} \subset \pF^0_{0,1} \subset \pF^0_{1,0} \subset \pF^0_{1,1} \subset\cdots\phantom{-.}\\
\cdots\supset \pF_\infty^{2,0}\supset \pF_\infty^{1,1} \supset \pF_\infty^{1,0} \supset \pF_\infty^{0,1} \supset \pF_\infty^{0,0} \supset \pF_\infty^{-1,1} \supset\cdots.
\end{gathered}\end{equation}

We summarize some results from  \cite{Charbonneau:2006gu}: The direct images $R^1\pi_*(\pE_{p,q}^{s,t})$ can be computed as the quotients $\pF/(\pF^0_{p,q}+\pF_\infty^{s,t})$. The direct images  
$$R^1\pi_*(\pE_{p,0}^{-p+1,0})=\pF/(\pF^0_{p,0}+\pF_\infty^{-p+1,0}), \ R^1\pi_*(\pE_{p,1}^{-p,1}) = \pF/(\pF^0_{p,1}+\pF_\infty^{-p,1})$$ are supported respectively over two {\it spectral curves}  $S_0$, $S_1$ in $\bO(2)$. When $\pE$ is twisted by $L^s$, for $s$ in an interval, these sheaves are the sources of the flows for Nahm's equations, by the well known correspondence of solutions to Lax pair type equations with flows of line bundles on a curve \cite{Nahm:1982jt, Hitchin:1983ay, Griffiths}; these flows are the Nahm transform of the caloron. Generically, one has a partition of the intersection $S_0\cap S_1$ of the two curves into two divisors $S_{10}  $ and $S_{01}$, and an identification of our quotients:
\begin{align*}
{\scriptstyle \pF}/{\scriptstyle(\pF^0_{p,0}+\pF_\infty^{-(p-1),0})} &=
 L^{p\ell+\lambda_-}(2k+m)[-S_{10}]|_{S_0} =
L^{(p-1)\ell+\lambda_+}(2k+m)[-S_{01}]|_{S_0},\\
{\scriptstyle \pF}/{\scriptstyle(\pF^0_{p,1}+\pF_\infty^{-p,1})} &=
L^{p\ell+\lambda_+}(2k+m)[-S_{01}]|_{S_1} =
L^{p\ell+\lambda_-}(2k+m)[-S_{10}]|_{S_1},\\
\end{align*}
The quotients  fit into a description of $\pF$ by the exact sequence
\begin{equation}\label{description-of-F}
 \begin{matrix}
0&\rightarrow&\pF&\rightarrow&
\begin{matrix}\vdots\\ \pF/(\pF^0_{p-1,1}+\pF_\infty^{-p+1,0}) \\ \oplus\\
\pF/(\pF^0_{p,0}+\pF_\infty^{-p,1}) \\ \oplus\\ \pF/(\pF^0_{p,1}+\pF_\infty^{-p,0}) \\ \vdots
\end{matrix}
&\rightarrow &
\begin{matrix}
\vdots\\ \oplus\\ \pF/(\pF^0_{p,0}+ \pF_\infty^{-(p-1),0})\\
\oplus\\ \pF/(\pF^0_{p,1}+\pF_\infty^{-p,1})\\ \oplus\\ \vdots
\end{matrix}
&\rightarrow &0.
\end{matrix}
\end{equation}
Generically, these become 
\begin{equation*}\label{sequence}
0\rightarrow
\pF\rightarrow\begin{matrix}
\vdots\\ L^{(p-1)\ell+\lambda_+}(2k+m)\otimes{\cal I}_{S_{01}}\\ \oplus\\
L^{p\ell+\lambda_-}(2k+m)\otimes{\cal I}_{S_{10}}\\ \oplus\\
L^{p\ell+\lambda_+}(2k+m)\otimes{\cal I}_{S_{01}}\\ \vdots\end{matrix}\nonumber
\rightarrow
\begin{matrix}
\vdots\\ L^{(p-1)\ell+\lambda_+}(2k+m)[-S_{01}]|_{S_0}\\
=L^{p\ell+\lambda_-}(2k+m)[-S_{10}]|_{S_0}\\ \oplus\\
L^{p\ell+\lambda_+}(2k+m)[-S_{01}]|_{S_1}\\ = L^{p\ell+\lambda_-}(2k+m)[-S_{10}]|_{S_1}\\
\vdots\end{matrix}
\rightarrow0.
\end{equation*}
Here $ \pI_{S_{0,1}},  \pI_{S_{1,0}}$ are the ideal sheaves of $S_{0,1}, S_{1,0}$. 

There is a natural shift operator $\Xi$ on this sequence, given by multiplication by the coordinate $\xi$ (more properly by the tautological section $\hat \xi$ of $L^{\ell}$); it acts on $\pF$ as an automorphism, and acts on the quotients by moving them two steps down, changing $p-1$ to $p$, and inducing isomorphisms, since  on the compactification   $\hat \xi$ 
has a zero over the divisor $\Gamma_0$, and a pole over the divisor $\Gamma_\infty$. 

These bundles will correspond to solutions of Nahm's equations on a circle; in this picture,  a shift in $p$ corresponds to a flow around the circle for Nahm's equations.  Equivalently,  the Nahm flow around the full circle shifts the relevant line bundles on the spectral curve by $L^{\ell}$; that this closes into a flow on the circle requires a multiplication by the tautological section $\hat \xi$ of $L^{-\ell}$.
  
We note that we can rebuild $\pE$ from $\pF$ and the shift operator   $\Xi$; along the surfaces $\xi = c,\ c\neq 0,\infty$, for example, one has 
$$\pE|_{\xi= c} = \pF/ (\Xi-c\bI)\pF.$$

\subsection{Restricting to a fibre $\zeta = 0$ of the twistor space: from bundles to monads, $m>0$ case}

\subsubsection{A chain of equivalent objects}

Let us now restrict to the surface $\zeta = 0$ in twistor space. The general Hitchin-Kobayashi, or Narasimhan-Seshadri, correspondence should tell us that the moduli of solutions on the full twistor space will correspond to moduli of holomorphic objects on this fibre. We had a shift operator $\Xi$, acting on $\pF$ by automorphisms. There is also a `half shift', moving the sequence down by one step, which corresponds to a Hecke transform of the bundle $\pE$ both at $\xi = 0$ and $\xi=\infty$. It interchanges the magnetic charge by $m\mapsto -m$, and so we need only consider $m\geq 0$. We consider first  the case $m>0$, as the data for $m>0$ and $m=0$  are somewhat different. 
 Our purpose in this section is to exhibit a chain of equivalences:

\begin{theorem} \cite{Charbonneau:2006gu,Charbonneau:2007zd,Nye-Singer} One has sets of equivalent data:
\begin{enumerate}
\item Framed holomorphic bundles $E$ on $\bP^1\times \bP^1$, as above;
\item Sheaves $P^{i,j}_{k,l}, Q^{i,j}_{k,l}$ on $\bP^1$;
\item  A 7-tuple of matrices $A, B, C, D_2, A',B',C'$ (modulo the action of $Gl(k,\bC)$), satisfying algebraic relations \eqref{monad-cal} and nondegeneracy conditions \eqref{gencon1},\eqref{gencon2},\eqref{gencon3}, and \eqref{gencon4};
\item  A monad $ V_1 {\buildrel{\alpha }\over{\rightarrow} }V_2 {\buildrel{\beta}\over{\rightarrow}}  V_3$ of standard vector bundles on $\bP^1\times \bP^1$, whose cohomology $\ker( \beta)/{\rm {Im}}(\alpha)$ is the bundle $E.$
\end{enumerate}

\end{theorem}

The precise definitions of the items on this list are given below.

\subsubsection{Holomorphic Data I: Bundles $E$ on $\bP^1\times \bP^1$}

We begin with the first item on the list, the bundle corresponding to a caloron; this is a holomorphic bundle over $\bC\times \bC^*$.
The restriction on the twistor space to $\zeta = 0$  corresponds to fixing a complex structure $\bR^3\times S^1 = \bC\times \bC^*\ni(\eta=t_1+i t_2,\xi=\exp(t_3+i\theta))$. As we have seen, the twistor space has a partial compactification to a $\bP^1$-bundle over $\bO(2)$, giving on  $\zeta = 0$, a product $\bP^1\times \bC$; the limits $\xi = 0, \xi =\infty $ in the $\bP^1$ correspond, respectively, to the limits $t_3\rightarrow -\infty, +\infty$ in $\bR^3$.  One is compactifying a cylinder by adding two points;  in the neighbourhood of one of these points, say as $t_3\rightarrow -\infty,$ one again copies the approach of e.g., Biquard \cite{Biquard91}, finding solutions to the Cauchy-Riemann equations which are asymptotic to a constant at $t_3= -\infty$, i.e. at $\xi=0$. This extends $E$ to a bundle at the punctures.  The asymptotics of the instanton tell us in addition that there is a sub line bundle $E^0_1$ along the added divisor $C_0$ corresponding to  the negative eigenbundle of the asymptotic connection component matrix  
$A_\theta$. In the same way, at the other end of the cylinders, one extends along the divisor $C_\infty$, obtaining a bundle with a subline bundle $E_\infty^1$ corresponding to the positive eigenvalue of the component $A_\theta.$

Our bundles also came with an asymptotic framing at $t_3\rightarrow \infty$, giving a   trivialization of the bundle $E$ along $C_\infty$ (the divisor cut out on $\bP^1\times \bP^1$ by $\xi = \infty$). This is compatible with the subbundle, so one can suppose that the subbundle $E_\infty^1$  corresponds to the first vector of the framing. 

Following \cite{Charbonneau:2006gu}, we compactify further to $\bP^1\times\bP^1$ by going to $\eta = t_1+ it_2 = \infty$. This is done in a way which respects the framing along $C_\infty$, extending the trivialization to $C_\infty\cup  \{\eta = \infty\}$. The flag along $C_\infty$ extends, in such a way that the degree of $E_\infty^1$ is zero; on the other hand $E^0_1$ has degree $-m$.  This is where an asymmetry between the divisors $C_0, C_\infty$ is introduced. Let $D , F$ denote the divisors $\eta= 0$, $\eta = \infty$.

In terms of our spaces (and their coordinates) 
\begin{equation}
\begin{matrix} \bP^1\times \bP^1; \  (\eta, \xi)& \hookleftarrow &\overline{\bT}|_{\zeta = 0}; \ (\eta,\xi)&\hookrightarrow& \overline{\bT}; \ (\eta,\xi, \zeta)\\
\downarrow\pi&&\downarrow\pi&&\downarrow\pi\\
\bP^1; \ (\eta)&\hookleftarrow &\bC;\ (\eta)&\hookrightarrow &\bO(2); \ (\eta,  \zeta).\end{matrix}\end{equation}

 We have, on the twistor side, our first set of holomorphic data \cite{Charbonneau:2006gu}. 
This consists of:

\begin{itemize} \item A holomorphic bundle $E$ on $\bP^1\times \bP^1$, with $c_1(E) = 0, c_2(E) = k$; 
\item  Sub-bundles $E^0_1 = \pO(-m) \rightarrow E$ along $C_0$, and $E_\infty^1= \pO  \rightarrow E$ along $C_\infty$.
\item A trivialisation of $E$ along $C_\infty\cup  F$, such that along $C_\infty$, the subbundle $E_\infty^1$ is the span of the first subspace of the trivialisation, and at $C_0\cap  F$, the   subbundle $E^0_1$ is the second vector of the trivialization.
\end{itemize}

\subsubsection{Holomorphic Data II: Sheaves   on $\bP^1$}

For the second item, let us now look at the restriction of the infinite flags, and their quotients. Set 
\begin{align}  P_{p,0}^{-p,1}&= R^1 \pi_*(E_{p,0}^{-p,1}) = F/(F^0_{p,0}+F_\infty^{-p,1}),\\
 P_{p,1}^{-p,0}& = R^1 \pi_*(E_{p,1}^{-p,0})= F/(F^0_{p,1}+F_\infty^{-p,0}), \nonumber\\
Q_{p,0}^{-p+1,0}  &= R^1 \pi_*(E_{p,0}^{-p+1,0})= F/(F^0_{p,0}+F_\infty^{-p+1, 0}), \nonumber\\
 Q_{p,1}^{-p ,1}&= R^1 \pi_*(E_{p,1}^{-p ,1})=  F/(F^0_{p,1}+F_\infty^{-p,1}).\nonumber\end{align}
 
 The diagram (\ref{description-of-F}) above restricts over $\bP^1$ to:

\begin{equation}\label{sheaf-sequences} 
\begin{matrix}
0&\rightarrow& \pO &\rightarrow& P_{p,0}^{-p,1}& \rightarrow& Q_{p,0}^{-p+1,0}&\rightarrow&0,\cr
0&\rightarrow& \pO(-m)&\rightarrow& P_{p,0}^{-p,1}& \rightarrow& Q_{p,1}^{-p ,1}&\rightarrow&0,  \cr
0&\rightarrow& \pO &\rightarrow& P_{p,1}^{-p,0}&  \rightarrow& Q_{p,1}^{-p ,1}&\rightarrow&0, \cr
0&\rightarrow& \pO(m) &\rightarrow& P_{p,1}^{-p,0}&  \rightarrow& Q_{p+1,0}^{-p,0}&\rightarrow&0.  \cr
\end{matrix}
\end{equation}

The $Q$s are torsion sheaves, supported away from $\eta=\infty$; the sheaves $P$ are then of rank one, though they may have torsion at the support of the $Q$s. Note, that the shift homomorphism $\Xi$ maps   $Q_{p+1,i}^{-p-1 ,i}$ to $Q_{p,i}^{-p ,i}$  isomorphically; likewise, it 
maps   $P_{p+1,i}^{-p-1,j}$ to  $P_{p,i}^{-p,j}$ isomorphically. In addition, since $E$ itself is locally free, there is a property of {\it irreducibility} of the sheaves in (\ref{sheaf-sequences}): 
\bigskip

\noindent{\bf Irreducibility Condition}\label{IrredCond} (\cite[page following lemma 9]{Charbonneau:2006gu}) 

1. {\it There are no skyscraper subsheaves $\bC_x$ of the $P $, $Q$ mapped to themselves by the maps above, and}

2. {\it there are no subsheaves of $P$, $Q$ mapped to themselves by the maps above, with common   skyscraper quotients $\bC_x$.}

\medskip

In short, the diagram does not have a `triangular structure', with either subobjects or quotient objects that  are resolution diagrams of torsion sheaves. The reason is that the existence of these would yield    sheaves $E$ which are not locally free, but are torsion free; the triangular structure arises from the sequence $E\rightarrow E^{**}\rightarrow E^{**}/E.$

Summarizing  from \cite{Charbonneau:2006gu}, one thus has our second set of holomorphic data:
\begin{itemize} \item Sheaves $P_{p,0}^{-p,1},P_{p,1}^{-p,0} ,Q_{p,0}^{-p+1,0}, Q_{p,1}^{-p,1}$ on $\bP^1$, fitting into sequences (\ref{sheaf-sequences}), with   $Q_{p,0}^{-p+1,0}$, $Q_{p,1}^{-p ,1}$ torsion, of length $k, k+m$, respectively, and supported away from infinity. The  $P,Q$ satisfy appropriate irreducibility conditions, given above.
\item A shift isomorphism $\Xi$ inducing isomorphisms between the  $P_{p+1,0}^{-p-1,1},P_{p+1,1}^{-p-1,0}$, $Q_{p+1,0}^{-p ,0}, Q_{p+1,1}^{-p-1,1}$ and $P_{p,0}^{-p,1}, P_{p,1}^{-p,0},$ $Q_{p,0}^{-p+1,0},$ $Q_{p,1}^{-p,1}$, commuting with the natural maps.
\item A trivialization of $P_{p,0}^{-p,1}$ and of $P_{p,1}^{-p,0}$ along $\eta = \infty$.
\item A genericity condition: the maps
\begin{equation}
 P^{0,0}_{0,1}\oplus P^{0,1}_{0,0} \xrightarrow{\begin{pmatrix}  r_{+,0}& -\xi'\Xi^{-1}r_{-,0}\\r_{+,1}&-r_{-,1}\end{pmatrix}} Q_{1,0}^{0,0}\oplus Q_{0,1}^{0,1}
 \end{equation} 
 are surjective for all $\xi'$.
\end{itemize}

To see how to get the equivalence between I and II, one has a sequence 
$$ 0\rightarrow F\rightarrow \mathop{\oplus}_{p\in \bZ} P_{p,0}^{-p,1}\oplus P_{p,1}^{-p,0}\rightarrow  \mathop{\oplus}_{p\in \bZ} Q_{p,0}^{-p+1,0}\oplus Q_{p,1}^{-p ,1}\rightarrow 0$$ defining $F$; sections of $F$ are then sequences 
$$ (..., s_{p-1,0},s_{p-1,1},s_{p,0},s_{p ,1},...)\label {sequence-p}$$
of the sum of the $P_{p,0}^{-p,1}\oplus P_{p,1}^{-p,0}$ which match when one maps them to the sum of the $ Q_{p,0}^{-p+1,0}\oplus Q_{p,1}^{-p ,1}$ under (\ref{sheaf-sequences}). The subspaces $F_\infty^{p,j}$ are then obtained as terminating sequences (i.e., zero after a certain point as one increases $p$); the subspaces $F^0_{p,j}$ are then obtained as initiating sequences (i.e., zero after a certain point as one decreases $p$).

One can  obtain 
$E$ along the line $\xi= \xi', \xi'\neq 0,\infty$ as 
$$E|_{\xi= \xi'} = F/ (\Xi-\xi'\bI)F.$$
Along $\xi=\infty$, $E$ is the quotient $F_\infty^{1,0}/F_\infty^{0,0}$, with subline bundle $F_\infty^{0,1}/F_\infty^{0,0}$; along $\xi=0$, $E$ is the quotient $F^0_{1,0}/F^0_{0,0}$, with subline bundle $F^0_{0,1}/F^0_{0,0}$. Over $\bP^1\times \bP^1$, one can obtain $E|_{\xi=\xi'}$ from 
\begin{equation}\label{caloron-sequence}
0\rightarrow  E^{0,0}_{0,1}\oplus E^{0,1}_{0,0} 
\xrightarrow{\begin{pmatrix}    i_{+,0}& -\xi'\xi^{-1} i_{-,0}\\i_{+,1}&-i_{-,1}\end{pmatrix}} E_{1,0}^{0,0}\oplus E_{0,1}^{0,1}\rightarrow E|_{\xi= \xi'}\rightarrow 0,
\end{equation} where the $i_{\pm,0}$ and $i_{\pm,1}$ are the natural inclusions. Taking direct images,  the sequence for $F$ `folds up' into  :
\begin{equation}\label{caloron-sequence-projected}
0\rightarrow  E|_{\xi= \xi'}\rightarrow P^{0,0}_{0,1}\oplus P^{0,1}_{0,0} \xrightarrow{\begin{pmatrix}  r_{+,0}& -\xi'\Xi^{-1}r_{-,0}\\r_{+,1}&-r_{-,1}\end{pmatrix}} Q_{1,0}^{0,0}\oplus Q_{0,1}^{0,1}.
\end{equation} 
In particular, this diagram gives us the genericity property.

Globally, all these fit together as follows: one has a variety $V$ defined as 
$$V = \{ (\eta, \xi,\xi')\in \bP^1\times \bP^1 \times \bP^1|\xi=\xi'\},$$
denoting by $\widetilde E,\ \widetilde E^{ij}_{kl}$, etc. the lifts   of $E,\ E^{ij}_{kl}$, etc. to $ \bP^1\times \bP^1 \times \bP^1$ via the projection onto the first two factors, we have
\begin{equation}\label{caloron-sequence-2}
0\rightarrow  \widetilde E^{0,0}_{0,1}\oplus \widetilde E^{0,1}_{0,0} \xrightarrow{\begin{pmatrix}r_{+,0}& -\xi'\xi^{-1}r_{-,0}\\r_{+,1}&-r_{-,1}\end{pmatrix}} \widetilde E_{1,0}^{0,0}([\bP^1 \times \bP^1\times \{\infty\}])\oplus \widetilde E_{0,1}^{0,1}\rightarrow \widetilde E|_V\rightarrow 0,
\end{equation} 
and taking direct images  to $\bP^1\times \bP^1$ (the last two factors of $\bP^1 \times \bP^1\times\bP^1$), we obtain
\begin{equation}\label{caloron-sequence-projected-2}
0\rightarrow  E \rightarrow \pi^*P_{0,0}^{0,1}\oplus \pi^*P_{0,1}^{0,0} \rightarrow \pi^*Q_{1,0}^{0,0}(C_\infty)\oplus \pi^*Q_{0,1}^{0,1},
\end{equation} 
where $\pi(\eta,\xi') = \eta$. From \cite[Theorem 7]{Charbonneau:2006gu}, we have:
\begin{proposition}  Holomorphic data I and II are equivalent.\end{proposition} 

\subsubsection{Holomorphic Data III: Matrices, up to the action of $Gl(k)$} 

To go on to our third set of holomorphic data, we use a natural resolution of the diagonal $\Delta$ in $\bP^1\times \bP^1$:
\begin{equation} \label{resolution-diagonal}
0\rightarrow \pO(-1,-1)\rightarrow \pO\rightarrow \pO_\Delta\rightarrow 0.
\end{equation}
Lifting our sheaves $P, Q$ to the diagonal and pushing down, we have resolutions:
\begin{equation} \label{resolution-PQ}\xymatrix{ &....&...& &....&\\
\pO(-1)^{k }\ \  \ar[dr]^{X_{+,1}  } \ar[r]^{W_+}\ar[ur]^{X_{+,0} }&   \pO^{k+1}\ar[dr]^{Y_{+,1}}\ar[ur]^{Y_{+,0}}\ar[rr]&   &P_{p,0}^{-p,1}\ar[dr]^{r_{+,1}}\ar[ur] &&\\  
 & \pO(-1)^{k+m} \ar[r]^{Z_1 } &\pO^{k+m}\ar[rr]& &Q_{p,1}^{-p,1}\\
 \pO(-1)^{k+m}\ \  \ar[dr]^{X_{-,0} }\ar[r]^{W_-}\ar[ur]^{X_{-,1} }& \pO^{k+m+1} \ar[dr]^{Y_{-,0} }\ar[ur]^{Y_{-,1} }\ar[rr]&& P_{p,1}^{-p,0}\ar[ur]_{r_{-,1}}\ar[dr]^{r_{-,0}} \\ 
 & \pO(-1)^{k } \ar[r]^{Z_0} &\pO^{k}\ar[rr]& &Q_{p+1,0}^{-p ,0}\\
 \pO(-1)^{k }\ \ \ar[dr]^{X_{+,1}  } \ar[r]^{W_+}\ar[ur]^{X_{+,0} }&   \pO^{k +1}\ar[dr]^{Y_{+,1}}\ar[ur]^{Y_{+,0}}\ar[rr]&   &P_{p+1,0}^{-p-1,1}\ar[dr] \ar[ur]_{r_{+,0}}&&\ar[uuuu]_{\Xi}\\
 &....&...&  &....&
}
\end{equation}
Again, on this diagram, there is a shift isomorphism $ \Xi$, which moves the diagram two steps up. The entries of the maps $W, X,Y, Z$ are  matrices, that can be normalized (see \cite{Charbonneau:2006gu}):
\bigskip

\begin{equation} 
\begin{matrix}
X_{+,1}&= & \begin{pmatrix}A\\A'\end{pmatrix},&&Y_{+,1}&=& \begin{pmatrix}A&C_2\\A'&C'_2\end{pmatrix},\\ \\
W_+&=& \begin{pmatrix}\eta-B\\-D_2\end{pmatrix},&&Z_1 &=&\begin{pmatrix}\eta-B&C_1e_+\\ (e_-)^TB'&(\eta-\Sh)+C'_1e_+\end{pmatrix},\\ \\
X_{-,1}&=& \begin{pmatrix}1&0\\0&1\end{pmatrix},&& Y_{-,1}&=& \begin{pmatrix}1&0&-C_1\\0&1&-C'_1\end{pmatrix},\\ \\
X_{-,0}&=& \begin{pmatrix}1&0\end{pmatrix},&&Y_{-,0}&=& \begin{pmatrix}1&0&0\end{pmatrix},\\ \\
W_-&=& \begin{pmatrix}\eta-B&0\\ -(e_-)^TB'&\eta-\Sh\\0&-e_+ \end{pmatrix},&&Z_0&=&\begin{pmatrix}\eta-B\end{pmatrix},\\ \\
X_{+,0}&= &\begin{pmatrix}\xi\end{pmatrix},&&Y_{+,0}&=& \begin{pmatrix}\xi&0\end{pmatrix}. 
\end{matrix}\label{matrices-caloron}
\end{equation}

Here $A,B,C,D_2,A',B',C'$ are matrices of size $k\times k,k\times k, k\times 2,2\times k, m\times k,1\times k, m\times 2$ respectively. Subscripts denote columns or rows, where appropriate.
We let   $\Sh$ denote the downward $k \times k$ shift matrix, with ones just below the diagonal;  $e_-=(1,0,...,0)$, and  $e_+=(0,0,...,0,1)$. Setting $D_1 = e_+A'$, the commutativity of the diagram (\ref{resolution-PQ}) expresses the monad conditions for the original bundle $E$:

\begin{align}
  [A,B] + CD&=0,                     \nonumber                      \\
  \label{monad-cal}
  (e_-)^TB'A + \Sh A' - A'B -C'D&=0,   \\
   -e_+ A'+   \begin{pmatrix}1&0\end{pmatrix}   D &=0.   \nonumber                                  
  \end{align}
There are, in addition, following non-degeneracy conditions; these are the same as for the monads for $E$, $K_0$, $K_{0\infty}$,$K_\infty$ of Charbonneau-Hurtubise \cite[Theorem 5]{Charbonneau:2006gu}:
 
 \begin{align}
&\begin{pmatrix}A-\xi\\B-\eta\\D\end{pmatrix}
             \text{ injective for all }\xi,\eta\in \bC,\label{gencon1}\\
&\begin{pmatrix}\eta-B&A-\xi& C\end{pmatrix}
           \text{ surjective for all }\xi,\eta\in \bC,\label{gencon2}\\
            \nonumber\\
&\begin{pmatrix}Y_{+,1},Z_1\end{pmatrix}
             \text{ surjective for all }\eta \in \bC,\label{gencon3}\\
\nonumber\\
&\widetilde N= \begin{pmatrix}\begin{pmatrix}A\\A'\end{pmatrix}&\begin{pmatrix}C_2\\ C'_2\end{pmatrix}&M\begin{pmatrix}C_2\\ C'_2\end{pmatrix}&\cdots&
           M^{m-1}\begin{pmatrix}C_2\\ C'_2\end{pmatrix}\end{pmatrix} \text{ is an isomorphism},\label{gencon4}
\end{align}
where
\begin{equation}
M= -Z_{-,1}(0)= \begin{pmatrix} B& -C_1e_+   \\
                        (e_-)^TB' &\Sh-C_1'e_+\end{pmatrix}.
\end{equation}
The first two conditions are linked to the irreducibility of complex of $P,Q$ and so, to the eventual local freeness of the sheaf $E$. The third is linked to the surjectivity of the map $P_{p,0}^{-p,1}\rightarrow Q_{p,1}^{-p ,1}.$ (The other surjectivities are automatic). The invertibility of the final matrix $\tilde N$ is linked to the fact that the map $P_{p,0}^{-p,1}(m-1)\rightarrow Q_{p,1}^{-p ,1}$ should induce an isomorphism on sections. See \cite[Lemma 7]{Charbonneau:2006gu}.

The various normalizations involved in the process use the framing condition present in the previous sets of data, and reduce the freedom of choice to an action of $Gl(k)$. 

\begin{proposition} \cite[Theorem 5]{Charbonneau:2006gu}. Holomorphic data II and III are equivalent.\end{proposition}

 \subsubsection{Holomorphic data IV: Monads over $\bP^1\times \bP^1$}

Of course this implies that data I and III are equivalent; one can  see this directly from how the $P$ and $Q$ resolutions   give a monad for $E$, that is a complex $A\buildrel{a}\over{\rightarrow} B \buildrel{b}\over{\rightarrow}C$ with $E$ identified as $\mathrm{Ker}(b)/\mathrm{Im}(a)$. Recall that the sections of $E$ along $\xi=c$ should be given as sections 
 $$ (..., p_{k-1,0}^{-k+1,0},p_{k-1,1}^{-k+1,0},p_{k,0}^{-k,1},p_{k ,1}^{-k,0},...)$$ 
 of $F$ that under a shift $k-1\rightarrow k$ are scaled by $c$.  This can be represented as sections  $p_{ +},p_{ -}$ of $P_+ = P_{0,0}^{0,1}, P_-= P^{0,0}_{0,1}$ lying in the kernel of the restriction map $r$ to  $Q_0= Q_{0,0}^{1,0},Q_1= Q_{0 ,1}^{0,1}$ given by:
 $$r(p_{ +},p_{ -})= ( r_{+,0}(p_{ +})-cr_{-,0}(p_{ -}),r_{+,1}(p_{ +})-r_{-,1}(p_{ -}))$$
 Varying $c$ and  replacing $c$ by $\xi$ amounts to lifting to $\bP^1\times\bP^1$.
 Let us write our resolutions of  both $P$'s and $Q$'s and the maps between them   induced by the  $r_{\pm, 0}, r_{\pm, 1}$  schematically as 
$$ \begin{matrix} \pO(-1)\otimes (U_{P_+}\oplus U_{P_-} )& \buildrel {W}\over {\longrightarrow} & \pO\otimes(V_{P_+}\oplus V_{P_-} )&\longrightarrow&P_+\oplus P_-\\
\downarrow X&&\downarrow Y &&\downarrow r\\
 \pO(-1)\otimes(U_{Q_0}\oplus U_{Q_1})& \buildrel {Z}\over {\longrightarrow} & \pO\otimes(V_{Q_0}\oplus V_{Q_1})&\longrightarrow &Q_0\oplus Q_1\end{matrix}.$$
 A section in the kernel of $r$ (that is, a section of the bundle $E$) gets represented by a section $v_P$ of $ \pO\otimes(V_{P_+}\oplus V_{P_-})$ which is mapped by $Y$ not necessarily to zero, but to an element $Z(u_Q)$ in the image of $Z$; i.e. $Y(v_P)-Z(u_Q) = 0$. These must then be considered modulo trivial $(v_P, u_Q)$, which are of the form
 $(W(u_P), X(u_P))$. In short, and more properly putting in the twists of equation (\ref{caloron-sequence-2}), sections of $E$ are represented by a monad on $\bP^1\times \bP^1$ :
 
\begin{equation}   \label{agmonad} \begin{matrix}   U_{P_+}(-1,0)\\\oplus \\  U_{P_-}(-1,0) \end{matrix}\xrightarrow{\begin{pmatrix}W_+&0\\0&W_-\\X_{+,0}&X_{-,0}\\X_{+,1}&X_{-,1}\end{pmatrix}}
 \begin{matrix}  V_{P_+}(0,0)\\ \oplus \\V_{P_-}(0,0)\\ \oplus  \\  U_{Q_0}(-1,1)\\ \oplus\\    U_{Q_1} (-1,0)\end{matrix}\xrightarrow{\begin{pmatrix}Y_{+,0}&Y_{-,0}&-Z_0&0\\Y_{+,1}&Y_{-,1}&0&-Z_1\end{pmatrix}}\begin{matrix} V_{Q_0}(0,1)\\ \oplus\\  V_{Q_1}(0,0)\end{matrix} .\end{equation}
  Here, if $V$ is a vector space, $V(i,j)$ denotes $V\otimes \pO(i,j)$; the matrices are those of (\ref{resolution-PQ}).  Expanding, as in (\ref{matrices-caloron}), 
 $${\begin{pmatrix}W_+&0\\0&W_-\\X_{+,0}&X_{-,0}\\X_{+,1}&X_{-,1}\end{pmatrix}} = \begin{pmatrix} \eta-B&0&0\\-D_2&0&0\\ 0& \eta-B&0 \\ 0& (e_-)^TB'& \eta-\Sh   \\0&0&-e_+ \\ \xi &1&0\\A& 1&0 \\A'&0&1   \end{pmatrix},$$
 
\begin{multline*}
\begin{pmatrix}Y_{+,0}&Y_{-,0}&-Z_0&0\\Y_{+,1}&Y_{-,1}&0&-Z_1\end{pmatrix}=\\
= \begin{pmatrix}\xi&0&1&0&0& -\eta+B&0&0\\ A&C_2 &1&0&-C_1&0& -\eta+B&-C_1e_+\\
 A'&C'_2&0&1&-C'_1& 0& (e_-)^TB'&\Sh-\eta -C'_1e_+\end{pmatrix}.\end{multline*}
 The matrices satisfy the monad relations (\ref{monad-cal}) and the genericity constraints \eqref{gencon1},\eqref{gencon2},\eqref{gencon3},\eqref{gencon4}.
 Essentially by row-reducing and column-reducing, one can show that this monad is equivalent to the smaller monad, which is more or less the standard one for bundles on $\bP^1\times\bP^1$ which are trivial on the lines $\{\infty\}\times\bP^1$ and $\bP^1\times \{\infty\}:$
 $$\xymatrix{ \pO(-1,0)^k\ar[rr]^-{\begin{pmatrix} \xi - A\\ \eta -B\\ D\end{pmatrix}}&{\phantom{A}}&\pO(-1,1)^{k}\oplus \pO^{k+2}\ar[rr]^-{\begin{pmatrix} \eta-B& A-\xi&C \end{pmatrix}}&{\phantom{\eta-B A-\xi }} 
 &\pO(0,1)^k}.$$
  
 In a similar way, one can recreate a subsheaf $E_0$ of sections of $E$ with values in the first subspace of the flag along $C_0$, and similarly a subsheaf $E_\infty$ of sections of $E$ with values in the first subspace of the flag along $C_\infty$; this gives our flags along $C_0, C_\infty$.  One can also recreate the trivialization, and get our holomorphic data I.

 \subsection {Nahm complex over the circle, and monads}
 
  The final set of holomorphic data that can be derived from the bundle $E$ is a Nahm complex: following \cite{Charbonneau:2006gu},  the Nahm complexes that we  consider over the circle, viewed as the real line with $s, s+l$ identified, are defined as in subsection \ref{Bow-complexes}, 
with the difference that there is no bifundamental data.  Instead, the fibres at $s$ and at $s+\ell$ are identified.  Thus, the Nahm complex solution is defined on the circle.

 The Nahm complex admits an action of a group of gauge transformations $g$ which are smooth away from the jump points $\lambda_-$ and $\lambda_+,$  and satisfy the appropriate compatibility conditions at the boundary points; using these, one can put the Nahm complex locally, into a  normal form:
 
\begin{lemma}\cite[Prop 1.15]{Hurtubise:1989wh}
\label{lemma:normalform}
\begin{itemize}  
\item Away from the $\lambda$-points, one can gauge  the connection and endomorphism to $\alpha = 0, \beta =B$ constant. This extends to the boundary points, if one is
on $N_{ 0}$ and, in the cases $m=0,1$, on $N_1$  also. 
\item At the $\lambda$-point (translated to $s=0$), over $N_1$, for $m>1$, one can gauge transform  the connection and endomorphism to the
block form
\begin{align}\label{alphaForm}
\alpha_{1} &= \frac{1}{s}\begin{pmatrix}0&0\\0&\mathrm{diag}({\frac{-(m-1)}{2}},{\frac{2-(m-1)}{2}},\dots,{\frac{(m-1)}{2}})\end{pmatrix},\\
\label{betaForm}
\beta_{1}&=\begin{pmatrix}B&s^{\frac{m-1}{2}}C_1e_+\\
s^{\frac{m-1}{2}}(e_-)^TB'&-s^{-1}\Sh + \tilde C'_1e_+\end{pmatrix}.
\end{align}
Here $B$ is $k\times k$, $B'$ is $1\times k$, $C_1 $ is $k\times 1$ ; all these are constant matrices. Also 
$ \tilde C'_1$ is $m\times 1$, with $ (\tilde C'_1)_i= s^{m-i}(C'_1)_i$, and $(C'_1)_i$ constants.
\end{itemize}
\end{lemma}

Let us denote the union of our two vector bundles $N_0$ and $N_1$ and their glueings at the boundary as one rather unusual bundle $N$ over the circle, whose rank happens to change across $\lambda_\pm$, so that there is a `large' interval $[\lambda_-, \lambda_+]$, and a `small' interval $[\lambda_+, \lambda_-+l]$. The Nahm construction, in its holomorphic geometric version, gives an infinite dimensional monad
\begin{equation} 
\widetilde{H}_{1}(V) \xrightarrow{\cD_1 = \left(\begin{smallmatrix}  d_s + \alpha -(t_3+\i\theta)\\  \beta- \eta\end{smallmatrix}\right)} L^2(V)^{\oplus 2}\xrightarrow{\cD_2 = \left(\begin{smallmatrix}  \eta-\beta,&d_s + \alpha -(t_3+\i\theta)\end{smallmatrix}\right)} \widetilde{H}_{-1}(V). \end{equation}
The function spaces much be chosen with a bit of care, so that both the derivative operator and multiplication by a function that has a pole at the $\lambda$-points are well defined. The  space $\widetilde{H}_{1}$ is thus a subspace of the standard Sobolev space ${H}_{1}$, for example. The Nahm complex through this infinite dimensional monad encodes a bundle $\widetilde E$ over $\bC\times \bC$, which is invariant under $\theta\mapsto \theta + 2\pi/\ell$, and so descends to  the quotient $\bC\times \bC^*$ by this action. 
\begin{proposition}\cite{Charbonneau:2007zd}

 1) Holomorphic data I-IV and V are equivalent.

2) Under this equivalence, the bundles $E$ and $\widetilde E$ that their respective monads encode are isomorphic.

\end{proposition}

\begin{proof} The first part is covered in   Charbonneau-Hurtubise \cite[section 5]{Charbonneau:2007zd}. The sheaves $P,Q$ contain all the information for the Nahm complex: 
\begin{itemize}
\item The bundles $N_0$, $N_1$ are simply $H^0(\bP^1, Q_0), H^0(\bP^1, Q_1)$, with the natural trivial connection; 
\item In this trivialization, the matrices $\beta_i$ are simply the matrices $Z_i$ above.
\item The glueing of $N_0, N_1$ is effected at $\lambda_\pm$   by the maps $$H^0(\bP^1, Q_0)\leftarrow H^0(\bP^1, P_\pm)\rightarrow H^0(\bP^1, Q_1).$$
\end{itemize}

We note, that 
the normal form becomes $\alpha_1= 0$, and $\beta_1 = Z_1$, 
if one acts by the singular gauge transformation 
$\diag(1,...,1,   z^{\frac{-(m-1)}{2}}, z^{\frac{2-(m-1)}{2}}, \dots,z^{\frac{(m-1)}{2}})$.  In this gauge $\beta_1 = Z_1$  equals to the matrix defined in (\ref{matrices-caloron}). This of course takes us out of the framework of the our monad of $L^2$ function spaces. 
The poles  here  appear to be  essentially put in by hand; the complex geometrical reason for having them only appears when one goes to the full twistor space, and is linked to the geometry of sections of $L^t$ on the spectral curves, as $t$ tends to zero. This is discussed in Hurtubise and Murray \cite{HurtubiseMurray}.

   There remains the global monodromy of  the connection $\alpha$.  For $m>1$, the normal forms at both ends of the `large' interval (on which the bundle has rank $k+m$) are conjugates by the matrix $\widetilde N$ defined above in equation \eqref{gencon4}; thus the parallel transport for the connection on the large interval will be $\widetilde N$.  The matrix $\widetilde N$ conjugates
 \begin{equation}\label{left-normal} Z_1 = {\begin{pmatrix} -B&C_1e_+\\(e_-)^TB'&-\Sh+C'_1e_+\end{pmatrix}}\end{equation} to
 \begin{equation}\label{right-normal}{\begin{pmatrix} -B&\tilde C_1e_+\\  -(e_-)^TD_2&-\Sh+\tilde C'_1e_+\end{pmatrix}}\end{equation}
 On the small interval one simply takes the identity map as the parallel transport.
 
 The case $m= 1$ is treated similarly. Conversely, given a Nahm complex, one can extract the matrices from the normal forms at the singular points, and the monodromy of the connection.

We exhibit  how the two sets of data define isomorphic bundles $E$ and $\widetilde E$. This is done in \cite{Charbonneau:2007zd}, but we now revisit it from a monad point of view. As we saw, the bundle $E$ was given as the cohomology of a monad (\ref{agmonad}). We want to show that the bundles $E$ and $\widetilde E$ are equivalent by exhibiting some morphisms of respective monads.
 To do this, we first reduce the infinite dimensional monad that encodes the  bundle $\widetilde E$   to a finite dimensional one, quite similar to  (\ref{agmonad}) that arises from the algebraic geometry. We begin by considering the situation at $t_3+i\theta = 0$.
 
We first note that in our complex, the elements $(\chi^1, \chi^2)$ in the kernel of $\cD_2$ can be modified by a coboundary so that $\chi^1$ is compactly supported in the intervals, at a distance $2\eps$ from the boundary. This amounts to solving $\chi^1  = (d_s  +  \alpha)u$ near the boundary, which one can do \cite{Donaldson:1985id}, as we argue momentarily, in spite of the pole of $\alpha$,  and then applying $\cD_1$ to $-u$ times an appropriate bump function.
 
Consider the spaces:
\begin{itemize}
\item $\widetilde V_+$ of solutions to $ (d_s  +  \alpha)v_+ = 0$  on the interval $(\lambda_--\ell,\lambda_-)$ which satisfy the boundary conditions at $\lambda_+$;
\item  $\widetilde V_-$  of solutions to $ (d_s  +  \alpha)v_- = 0$  on the interval $(\lambda_+ - \ell,\lambda_+ )$ which satisfy the boundary conditions at $\lambda_-$;
\item Subspaces $\widetilde  U_\pm$ of $\widetilde  V_\pm$, such that not only $v_\pm$, but also $\beta(v_\pm)$, satisfy the boundary conditions at $\lambda_\pm$; as $\beta$ has a pole at the boundary, this gives a space that is one dimension smaller.
\item $\widetilde U_0=\widetilde V_0$ of values $u_0 = u(\lambda_--\epsilon)$  at   a fixed point $\lambda_--\epsilon $  (outside of the support of $\chi^1$) of solutions on $(\lambda_+ -\ell , \lambda_-)$ to the 
equation 
\begin{equation}(d_s  +  \alpha)u (s) = \chi^1,\label{chione}\end{equation}  
with $u(\lambda_+ - \ell  +\epsilon) =0$; $\widetilde U_0$, of course, is just the fibre of $N$ at $\lambda_--\epsilon$;
\item $\widetilde U_1= \widetilde V_1$ of values $u_1 = u (\lambda_+-\epsilon)$ of  solutions  on $(\lambda_-, \lambda_+)$ to the 
equation $(d_s  +  \alpha)u(s) = \chi^1 $, with $u(\lambda_-+\epsilon) =0$. 
 
\end{itemize}
 
This assembles into a finite dimensional monad to which our infinite dimensional monad reduces; it is basically the same as (\ref{agmonad}), and indeed yields the same bundle. To see this, let us begin over $t_3+i\theta = 0$. Consider a solution   $(\chi^1, \chi^2)$ to $\cD_2(\chi^1, \chi^2) =0$ satisfying the boundary conditions, with $\chi^1 = 0$ near the boundary points. Write this as  $(\chi^1_0, \chi^2_0)$ on $(\lambda_+-l, \lambda_-)$ and as $(\chi^1_1, \chi^2_1)$ on $(\lambda_-, \lambda_+)$. The section  $  \chi^2_0 $  can be written as 
 $v_+ +(\beta_0-\eta) u $  on $(\lambda_+-l, \lambda_-)$ with $u$ solving \eqref{chione}; for it to extend past $\lambda_-$, one needs
 $$ev_{+,0} (v_+) + (\beta_0-\eta) u_0 = -ev_{-,0}(v_-)$$
 for some $v_-\in \widetilde V_-$; here the $ev$ denote evaluation, and $\beta_0$ is the evaluation of $\beta$ at the reference point in the interval.  In a similar vein, writing $\chi^2_1 $ as $v_-(s) +(\beta_1-\eta)u(s)$; again, for it to extend past $\lambda_+$, one has
 $$ev_{-,1}(v_-) + (\beta_1-\eta)u_1 =  -ev_{+,1}(v_+),$$
where of course we are extending our solutions periodically.  The effect of $\cD_1$ on this is to modify the $v_-, v_+$ by solutions of the form $(\beta-\eta)u_-, (\beta-\eta) u_+$ defined on the same intervals as $v_-, v_+$. $\cD_1$ also modifies the $(u_0, u_1)$ by 
adding to it $(-ev_{+,0}(u_+) - ev_{-,0}(u_-), -ev_{-,1}(u_-)- ev_{+,1}(u_+))$. 

Let us now return to 	demonstrating that one can solve  $(d_s+\alpha)u=\chi^1$ with $\chi^1$ in $L^2$, for a $u$   in the desired Sobolev space $H^1$ i.e. with a  $u$ and its derivative that are indeed square integrable at the $\lambda$-point. 
This is a three step argument employing the frame of \cite{Donaldson:1985id} adapted to the Nahm residue: 
1. The form \eqref{alphaForm} of the pole of $\alpha$ implies that there is a unique solution $w_0$  of $(d_s+\alpha)w=0$ of order $s^{\frac{m-1}{2}}$ at the $\lambda$-point.  The leading pole \eqref{betaForm} of $\beta$ in turn implies that $w_1:=\beta w_0, w_2:=\beta^2 w_0, \ldots, w_{m-1}:=\beta^{m-1}w_0$ are solutions of $(d_s+\alpha)w=0$ of respective orders $s^{\frac{m-3}{2}}, s^{\frac{m-5}{2}},\ldots,s^{-\frac{m-1}{2}}$ at the $\lambda$-point. 
2. Thus, the union of the set $\{v_j:=s^{-\frac{m-1-2j}{2}}w_j\}_{j=0}^{m-1}$  and by any frame in continuing components, form a completely regular frame at the $\lambda$-point.  
3. If $\chi^1_0, \chi^1_1, \ldots, \chi^1_{m-1}$ are the components of $\chi^1$ in this frame, then the components of $u$ are given by 
$u_j=s^{\frac{m-1-2j}{2}}\int_0^s t^{-\frac{m-1-2j}{2}} \chi_j(t)dt$ if $m-1-2j\leq0$ and
$u_j=s^{\frac{m-1-2j}{2}}\int_\epsilon^s t^{-\frac{m-1-2j}{2}} \chi_j(t)dt$ if $m-1-2j>0.$ Therefore, if $\chi^1$ is $L^2$ at the $\lambda$-point, then $u$ is indeed in $H^1.$

Moving this picture to an arbitrary $t_3+i\theta$ modifies our spaces in a simple fashion, in a way that only depends on $\xi = e^{t_3+i\theta}$: the solutions $v_\pm$ for general $\xi$ are just the solutions for $\xi= 0$ multiplied by $\xi$. This changes nothing in our formulae except the monodromy, inserting a factor of $\xi$ in one of our evaluation maps. Thus, we have  a finite dimensional 
monad, equivalent to the infinite-dimensional Nahm monad: 
 
$$  \begin{matrix} \widetilde U_+ \\ \oplus\\  \widetilde  U_-\end{matrix}\xrightarrow{\begin{pmatrix}\beta -\eta&0\\ 0&\beta -\eta \\ -\xi ev_{+,0}&-ev_{-,0}\\   -ev_{+,1}&-ev_{-,1}\end{pmatrix}}  \begin{matrix}\widetilde  V_{+}\\ \oplus\\ \widetilde  V_{-}\\ \oplus\\\widetilde   U_{0}\\ \oplus\\\widetilde     U_{1} \end{matrix} \xrightarrow{\begin{pmatrix}\xi ev_{+,0}& ev_{-,0}& \beta_0-\eta&0\\
  ev_{+,1}&ev_{-,1}&0&\beta_1-\eta\end{pmatrix}}
 \begin{matrix}  \widetilde  V_{ 0} \\ \oplus\\ \widetilde  V_{ 1} \end{matrix}. $$
 
 Now we can put in some trivializations and see what these maps become. The maps $ev_{\pm, 0}:  \widetilde V_\pm\rightarrow  \widetilde  U_0$ are onto, and so can be put in a standard form $(1,0)$. Likewise,  the map 
 $ev_{-,1}:  \widetilde V_-\rightarrow  \widetilde  U_1$ is an injection, and can be put in a standard form $(1,0)^T$. On the other hand, the remaining map  $ev_{+,1}$ then has to contain the information of the monodromy of the connection. We are then getting a monad that is quite close to that of (\ref{agmonad}).
  
Unfortunately, the spaces $\widetilde V_\pm$ both have dimension $k+ \mathrm{Int}(m /2)$, where $\mathrm{Int}$ denotes the integer part, while the spaces $  V_\pm$ have dimensions $k+1, k+m+1$. The link can be understood as follows. Recall that $V_\pm$ were the spaces of sections of $P_\pm$; these fit into   exact sequences 
 \begin{equation} 
\begin{matrix}
0&\rightarrow& \pO &\rightarrow& P_+& \rightarrow& Q_0&\rightarrow&0,\\
0&\rightarrow& \pO(m) &\rightarrow& P_-&  \rightarrow& Q_0&\rightarrow&0.  
\end{matrix}
\end{equation}
 These engender a whole sequence of extensions
  \begin{equation} 
\begin{matrix}
0&\rightarrow& \pO(\nu) &\rightarrow& P_{+,\nu}& \rightarrow& Q_0&\rightarrow&0,\\
0&\rightarrow& \pO(m+\nu) &\rightarrow& P_{-,\nu}&  \rightarrow& Q_0&\rightarrow&0, \end{matrix}
\end{equation}
 as sheaves of sections of $ P_{\pm}$ with poles of order $\nu$ allowed at infinity for $\nu>0$ or zeroes of order $-\nu$ forced at infinity for $\nu<0$. We note that the $P_{+,\nu}$ are all isomorphic, away from infinity, and in a fairly natural way.  Their spaces of sections $V_{+, \nu}$ are nested: $V_{+, \nu}\subset  V_{+, \nu+1}$, with the difference being just one extra pole allowed at infinity. The same naturally holds for the $P_{-,\nu}, V_{-, \nu} $. In this vein,  $\widetilde V_+$ should be identified with $V_{+, \mathrm{Int}(m/2)-1}$, and $\widetilde V_-$ should be identified with $V_{-, -\mathrm{Int}(m/2)-1}$. In short, 
 while $E$ is defined on $\bP^1\times \bC$ as a bundle by 
\begin{equation} 
0\rightarrow  E \rightarrow \pi^*P_{+}\oplus \pi^*P_{-} \rightarrow \pi^*Q_{0}(C_\infty) \oplus \pi^*Q_{ 1}, 
\end{equation} 
 the bundle $\widetilde E$ should be thought of as being defined by the (isomorphic over  $\bP^1\times \bC$) sequence
  \begin{equation} 
0\rightarrow \widetilde E \rightarrow \pi^*P_{+, \mathrm{Int}(m/2)-1}\oplus \pi^*P_{-, -\mathrm{Int}(m/2)-1} \rightarrow \pi^*Q_{0}(C_\infty) \oplus \pi^*Q_{ 1}. \end{equation} 
The difference between the two would only emerge on a compactification, where of course many choices are possible.
  
  On the level of monads, the isomorphism  of $E$ and $\widetilde E$ is mediated by the maps
\begin{equation}\label{ajustment} \begin{matrix} 
  V_+&\rightarrow&  V_{+, \mathrm{Int}(m/2)-1}&\leftarrow& V_{+, \mathrm{Int}(m/2)-1}= \widetilde V_+,\\
  V_-&\rightarrow& V_-&\leftarrow& V_{-, -\mathrm{Int}(m/2)-1}= \widetilde V_-,\\
  U_+&\rightarrow& U_{+, \mathrm{Int}(m/2)-1}&\leftarrow& U_{+, \mathrm{Int}(m/2)-1}= \widetilde U_+,\\
  U_-&\rightarrow& U_-&\leftarrow& U_{-, -\mathrm{Int}(m/2)-1}= \widetilde U_-,
  \end{matrix}\end{equation}
 with the maps on $U_0, U_1, V_0, V_1$ simply being the identity. Once one does this, and adjusts for choices of sign, and remembers the normal forms for the Nahm complex, the monads coincide. \end{proof}
 
 \subsection{The case m=0}
 
 This case is somewhat simpler: the first set of data is essentially the same:
 \subsubsection{Holomorphic Data I: Bundles $E$ on $\bP^1\times \bP^1$} 
 This consists of:

\begin{itemize} 
\item A holomorphic bundle $E$ on $\bP^1\times \bP^1$, with $c_1(E) = 0, c_2(E) = k$; 
\item Sub-bundles $E^0_1 = \pO\rightarrow E$ along $C_0$, and $E_\infty^1= \pO  \rightarrow E$ along $C_\infty$;  
\item A trivialisation of $E$ along $C_0\cup  \{\eta = \infty\}$, such that along $C_0$, the subbundle $E^0_1$ is the span of the first vector of the trivialisation, and at the intersection of $\{\eta = \infty\}$ with $C_\infty$, the subbundle $E_\infty^1$ is the span of the second vector.
\end{itemize}

Similarly, the passage to the second set of data is identical, as follows.  
\subsubsection{Holomorphic Data II: Sheaves   on $\bP^1$}

\begin{itemize} 
\item Sheaves $ P_{p,0}^{-p,1}, P_{p,1}^{-p,0}, Q_{p,0}^{-p+1,0}, Q_{p,1}^{-p,1}$ on $\bP^1$, fitting into sequences 
\begin{equation}\label{sheaf-sequences-2.1} 
\begin{matrix}
0&\rightarrow& \pO &\rightarrow& P_{p,0}^{-p,1}&\buildrel{r_{+,0}}\over{\rightarrow}& Q_{p,0}^{-p+1,0}&\rightarrow&0,\\
0&\rightarrow& \pO&\rightarrow& P_{p,0}^{-p,1}&\buildrel{r_{+,1}}\over{\rightarrow}& Q_{p,1}^{-p,1}&\rightarrow&0,\\
0&\rightarrow& \pO &\rightarrow& P_{p,1}^{-p,0}&\buildrel{r_{-,1}}\over{\rightarrow}& Q_{p,1}^{-p,1}&\rightarrow&0,\\
0&\rightarrow& \pO &\rightarrow& P_{p,1}^{-p,0}&\buildrel{r_{-,0}}\over{\rightarrow}& Q_{p+1,0}^{-p,0}&\rightarrow&0,
\end{matrix}
\end{equation}
with   $Q_{p,0}^{-p+1,0}$, $Q_{p,1}^{-p ,1}$ torsion, both of length $k$, satisfying the same irreducibility conditions as in the $m>0$ case.
\item  A shift isomorphism $\Xi$ inducing isomorphisms between the  $P_{p,0}^{-p,1}$, $P_{p,1}^{-p,0},$ $Q_{p,0}^{-p+1,0},$ $Q_{p,1}^{-p,1}$ and $P_{p+1,0}^{-p-1,1},P_{p+1,1}^{-p-1,0} ,Q_{p+1,0}^{-p ,0}, Q_{p+1,1}^{-p-1,1},$ respectively.
\item A trivialization of $P_{p,0}^{-p,1}$ and of $P_{p,1}^{-p,0}$ along $\eta = \infty$.
\end{itemize}
 Again, one can take   resolutions, and obtain  a diagram (\ref{resolution-PQ}); the matrices are given by the following:
 \subsubsection{Holomorphic Data III: Matrices, up to the action of $Gl(k)$} 
\begin{equation}
\begin{matrix}
X_{+,1}&= & \begin{pmatrix}A \end{pmatrix},&&Y_{+,1}&=& \begin{pmatrix}A&C_2 \end{pmatrix},\\ \\
W_+&=& \begin{pmatrix}\eta-B_0\\-D_2\end{pmatrix},&&Z_1 &=&\begin{pmatrix}\eta-B_1\end{pmatrix},\\ \\
X_{-,1}&=& \begin{pmatrix}1 \end{pmatrix},&& Y_{-,1}&=& \begin{pmatrix}1& -C_1 \end{pmatrix},
 \\ \\
X_{-,0}&=& \begin{pmatrix}1 \end{pmatrix},&&Y_{-,0}&=& \begin{pmatrix}1& 0\end{pmatrix},\\ \\
W_-&=& \begin{pmatrix}\eta-B_0 \\  -D_1 A^{-1}\end{pmatrix},&&Z_0&=&\begin{pmatrix}\eta-B_0\end{pmatrix},\\ \\
X_{+,0}&= &\begin{pmatrix}1 \end{pmatrix},&&Y_{+,0}&=& \begin{pmatrix}1&0\end{pmatrix}.\end{matrix}\label{matrices-caloron-2}\end{equation}

Note, that $A$ is invertible; this corresponds to the fact that the bundle $E$ is trivial along $C_0$. The matrices are determined up to a common action of $Gl(k,\bC)$. The commutativity of the diagram gives the constraint, equivalent to the monad condition:
 $$  [A,B_0] + CD=0. $$
 Also, one has that the matrices $B_0,B_1$ giving the sheaves $Q_0,Q_1$ differ by a matrix of rank one:
 $$B_1 =B_0 - C_1D_1.$$
One has genericity conditions; in addition to asking that $A$ be invertible, one stipulates:
  \begin{align}
\begin{pmatrix}A-\xi\\B-\eta\\D\end{pmatrix}
            &\text{ injective for all }\xi,\eta\in \bC,\label{gencon1.1}\\
\begin{pmatrix}\eta-B,&A-\xi,& C\end{pmatrix}
            &\text{ surjective for all }\xi,\eta\in \bC.\label{gencon2.1}
\end{align}
Again this is linked to the irreducibility of the $P,Q$s, and eventually to the local freeness of the bundle $E$.

 \subsubsection { Holomorphic data IV: Monads}
 
As in the case of $m>0$, a monad is built out of the resolutions for $P, Q$, and with matrices above. The formulae are the same as for $m>0$.
 
  \subsubsection { Holomorphic data V: a Nahm complex over the circle}
 
For $m= 0$, the constraints are simpler: the Nahm complexes over the circle that we consider are defined by
\begin{itemize}
\item A bundle $N_0$ of rank $k$ over the interval
$[\lambda_-, \lambda_+]$, equipped with a smooth connection
$d_{\alpha_{0}}$, and a
covariant constant smooth section $\beta_{0}$ of $\mathrm{End}(V_{ 0})$.
\item A bundle $N_1$ of rank $k$ over the interval
$[\lambda_+, 2\pi+ \lambda_-]$, equipped with a smooth connection
$d_{\alpha_{1}}$ and a covariant constant smooth section
$\beta_{1}$ of $\mathrm{End}(N_1).$
\item At the boundary points $\lambda_\pm$, isomorphisms
$i_\pm\colon N_0\rightarrow N_1$, $\pi_\pm= i_\pm^{-1}$
with the gluing condition that
$\beta_{0}-\pi_\pm\beta_{1}i_\pm$ has rank one at the boundary.
\item At both boundary points, extra data consisting of decompositions  of the rank one boundary difference matrices  
$\beta_{ 0}-\pi_-\beta_{1}i_-$, $\beta_{0}-\pi_+\beta_{1}i_+$ into products of pairs of a column and a row vector $(I_-, J_-)$ and $ (I_+,J_+)$:
\begin{equation}
\beta_{0}-\pi_-\beta_{1}i_-= I_-\cdot J_-,\quad \beta_{0}-\pi_+\beta_{1}i_+ = I_+\cdot J_+.\end{equation}
\end{itemize}

The procedure for passing from our other holomorphic data to the Nahm complex is similar to the case of $m>0$, but again simpler: the sections of $Q_i$ are associated to covariant constant sections of $V_i$; the sections of $P_\pm$ are associated to the covariant constant sections near the boundary points $\lambda_\pm$, with the sections of $P_\pm(-1)$ mediating the isomorphisms between $N_0$ and $N_1$ at these points:
$$H^0 (\bP^1, Q_0) \simeq H^0 (\bP^1, P_\pm(-1))\simeq H^0 (\bP^1, Q_1).$$
The maps  $H^0 (\bP^1, P_\pm )\rightarrow H^0 (\bP^1, Q_i)$ on the one hand, have one-dimensional kernels $V_i$; on the other hand $H^0 (\bP^1, P_\pm(-1))$ sits naturally inside $H^0 (\bP^1, P_\pm )$ as sections vanishing at infinity. We then have a decompositions $H^0 (\bP^1, P_\pm ) = H^0 (\bP^1, P_\pm(-1))\oplus V_i $ and so projections $\pi_{\pm,i}: H^0 (\bP^1, P_\pm ) \rightarrow H^0 (\bP^1, P_\pm(-1))$, with kernels $V_i$, and with $\pi_{\pm,1}=\pi_{\pm,0}$ on $H^0 (\bP^1, P_\pm(-1))$, so that $\pi_{\pm,1}-\pi_{\pm,0}$ is of rank one. Multiplication $m_\eta$ by the coordinate $\eta$ defines a map 
$$H^0 (\bP^1, P_\pm(-1))\rightarrow H^0 (\bP^1, P_\pm),$$
 and one has that the covariant constant sections $\beta_i$ of the Nahm complex are defined by $\beta_i = \pi_{\pm,i}\circ m_\eta$. This is the source of the rank one jumps from $\beta_0$ to $\beta_1$.

 One does not need to worry about the poles of covariant constant sections, since the Nahm data is regular.

In terms of matrices, the matrices $B_i$ then get translated into the covariant constant endomorphisms $\beta_i$. The rank one jumps of $\beta$ at the $\lambda_\pm$ are then given by the matrices $C_i, D_i$. For the connection on the circle, the sole invariant is the global holonomy, and this is given by the matrix $A$.

\section{Holomorphic Data for the Taub-NUT}

As noted above, the geometry of the Taub-NUT manifold $\bR^4$ is more closely tied to the Hopf map $\bR^4\rightarrow \bR^3$ (a circle bundle away from the origin) than to the trivial circle bundle over $\bR^3$. Let us recall the geometry of the Taub-NUT manifold  $X_0= \bR^4=\bC^2$ under the $\zeta = 0$ complex structure. 
First, the Hopf map to $\bR^3$ is given by 
$$(\xi,\psi)\mapsto (t_1+\i t_2, t_3) = (\xi\psi, \frac{|\psi|^2-|\xi|^2}{2}).$$
The  fibres of this map  away from the origin are orbits under the action by complex scalars of unit length.
In particular, fixing the direction in $\bR^3$ corresponding to $\zeta=0$, one has the parallel family of lines $L_\eta=\{ (t_1,t_2,t_3)\in \bR^3| \eta= t_1+\i t_2\}$, and, over them in $\bR^4$, the family of conics (cylinders for $\eta\neq 0$) given in complex coordinates $(\xi,\eta)\in \bC^2=\bR^4$ by $\eta = \xi\psi.$ In other words, the restriction $X_0 = \bR^4$ of the twistor space $Z$ to  $\zeta = 0$  is a family of conics $\xi \psi=\eta$. This was compactified above, first into a surface $X_0'$, by adding two points to each conic, so that one has a family of compact conics over $\bC$, and then to a closed surface $X$ by adding a conic  (two lines) over $\eta = \infty$. We refer to subsection \ref{twist}.

\subsection{Restricting to a fibre: from bundles to monads, the case $m>0$}
As for the caloron, our aim will be to exhibit a chain of equivalences:

\begin{theorem} One has equivalent set of  data:
\begin{enumerate}
\item Holomorphic bundle $E$ on $X$;
\item A collection of sheaves $P^{i,j}_{k,l}, Q^{i,j}_{k,l}$ on $\bP^1$;
\item  A tuple of matrices $A, B, C, D_2, A',B',C', B_{h,t}, B_{t,h}$;
\item  A monad $ V_1 {\buildrel{\alpha }\over{\rightarrow} }V_2 {\buildrel{\beta}\over{\rightarrow}}  V_3$ of standard vector bundles on $X$, whose cohomology $\ker( \beta)/\mathrm{Im}(\alpha)$ is the bundle $E.$
\end{enumerate}

\end{theorem}

Again, the precise description follows.

\subsubsection{Holomorphic Data I: Bundle $E$ on $X$}

Over the  Taub-NUT, a solution to anti-self-duality equations on a bundle $E$  gives us an integrable complex structure on  $E$  over $X_0$; as for the caloron, the asymptotic behaviour of the connection and its curvature give us an extension of this structure for $E$ over $X'_0$.  In addition, along the divisors $C_0, C_\infty$
one has a holomorphic subbundle corresponding to  the (negative) eigenbundle of the asymptotic  monodromy of the operator 
$\frac{\partial}{\partial \theta} + A_\theta$. 

We extend from $X'_0$ to $X$. Again, this is where an asymmetry is introduced between $C_0$ and $C_\infty$, following the example used in other cases, such as monopoles or calorons; one takes a trivialization at $t_3=-\infty$ (the points corresponding to the points $\eta\neq \infty$ in $C_\infty$) in which the eigenvectors of the monodromy form a basis, and extend this to $\eta=\infty$ in the divisor $C_\infty$. As the bundle is trivial on  $\pi^{-1}(U')$  for $U' = U-\{\eta =\infty\}$ for some neighbourhood $U$ of  $\eta=\infty$ in $C_\infty$, one then has a natural extension of the bundle as a trivial bundle on $\pi^{-1}(U)$ using the trivialization on $U$.  This gives a bundle that is trivial on $C_\infty$, and a subline bundle $E_\infty^1$ that is a trivial subline bundle. 

  The result is then a bundle on $X$, and  the subbundle $E^0_1$ along $C_0$ becomes an $\pO(-m)$; we can adjust our trivializations so that the subbundle $\pO(-m)$ corresponds to the second vector of our induced trivialization at $C_0\cap F_\xi$.

More specifically, the data consists of 
\begin{itemize}
\item A rank 2 holomorphic vector bundle $E$ on $X$, with $c_1 = 0, c_2 = k$,
\item A subbundle $E^0_1= \pO(-m)\hookrightarrow E$ along $C_0$, and another  subbundle ${E_\infty^1=\pO\hookrightarrow E}$ along $C_\infty,$
\item A trivialization  on $C_\infty\cup F_\xi\cup F_\psi $, with the subline bundle $E^1_\infty$ corresponding to the first vector; the flag $E^0_1$ corresponds to the second vector in the trivialization at 
$C_0\cap F_\xi$.
\end{itemize}

As for calorons, one can define the sheaves $E_{m,n}^{p,q},$ which one can push down by $\pi$ to $\bP^1$. Pushing down $E$ from $X-C_0-C_\infty$ to $\bP^1$, one gets a  sheaf $F$ of infinite rank. It is again filtered as above, by subsheaves $F^0_{p,q},
F_\infty^{p,q}$, and one has the quotients $P_{p,0}^{-p,1},P_{p,1}^{-p,0} ,Q_{p,0}^{-p+1,0}, Q_{p,1}^{-p,1}$ in the same way, fitting into sequences as before:

\begin{equation}\label{sheaf-sequences-2.2} 
\begin{matrix}
0&\rightarrow& \pO &\rightarrow& P_{p,0}^{-p,1}& \buildrel {i_1}\over{\rightarrow}& Q_{p,0}^{-p+1,0}&\rightarrow&0,\cr
0&\rightarrow& \pO(-m)&\rightarrow& P_{p,0}^{-p,1}&\buildrel {i_2}\over{\rightarrow}& Q_{p,1}^{-p ,1}&\rightarrow&0,  \cr
0&\rightarrow& \pO &\rightarrow& P_{p,1}^{-p,0}& \buildrel {i_3}\over{\rightarrow}& Q_{p,1}^{-p ,1}&\rightarrow&0, \cr
0&\rightarrow& \pO(m) &\rightarrow& P_{p,1}^{-p,0}& \buildrel {i_4}\over{\rightarrow}& Q_{p+1,0}^{-p,0}&\rightarrow&0.  \cr
\end{matrix}
\end{equation}

Again, the $Q_{p,0}^{-p+1,0}$ are supported over $k$ points, counted with multiplicity, and the $Q_{p,1}^{-p ,1}$ over $k+m$ points, also counted with multiplicity; the calculation is an application of the Grothendieck-Riemann-Roch theorem. The big difference is the shift operator $\Xi$, which does not define an isomorphism as it did for calorons. Indeed, let $\xi,\psi$ be our standard holomorphic coordinates on $X$. 
The  function  $\xi$ over  $X$ has a pole along $ C_\infty\cup F_\xi$, and a zero along $C_0\cup D_\xi$; likewise $\psi$ has a pole along $ C_0\cup F_\psi$, and a zero along $C_\infty\cup D_\psi$. Multiplication by  $\xi$ and by $\psi$ induce respective  morphisms 
\begin{align} \Xi: E_{m,n}^{p,q}\rightarrow& E_{m-1,n}^{p+1,q}( -D_\xi+ F_\xi),\\
\Psi:E_{m,n}^{p,q}\rightarrow& E_{m+1,n}^{p-1,q}( -D_\psi+ F_\psi).
\end{align}
Thus, the multiplication operators induce twists by divisors located above $\eta = \xi\psi = 0,\infty$; they no longer induce shift isomorphisms $p\rightarrow p\pm 1$ on the $P,Q$.  
\begin{equation}\label{shift3}\xymatrix { &Q_{p+1,0}^{-p,0}&\\
^{(-D_\xi)}Q_{p ,0}^{-p+1 ,0}\ar[ur]^{\Xi^{-1}}\ar[dr]^I &&^{(-D_\psi)}Q_{p+1,0}^{-p ,0} \ar[ul]^{I}\ar[dl]_{\Psi^{-1}}\\ 
&  Q_{p ,0}^{-p+1,0}&}\end{equation}
 Here $I$ denotes the map given by the natural inclusions.  The sheaves
 $^{(-D_\xi)}Q_{p ,0}^{-p+1 ,0}$,
 $^{(-D_\psi)}Q_{p+1,0}^{-p ,0}$ are the direct images $R^1\pi_*(E_{m-1,n}^{p+1,q}( -D_\xi+ F_\xi))$, $R^1\pi_*(E_{m+1,n}^{p-1,q}( -D_\psi+ F_\psi))$, respectively.
 For the $Q$s, the maps $\Xi^{-1}$, $\Psi^{-1}$ are isomorphisms, as the bundle $E$ is trivial over $\eta =\infty$, and   the $Q$s are supported away from infinity. One then has maps:
$$\widehat B_{h,t} = I\circ \Xi: Q_{p+1,0}^{-p,0}\rightarrow Q_{p ,0}^{-p+1,0},\quad \widehat B_{t,h} = I \circ\Psi: Q_{p ,0}^{-p+1,0}\rightarrow Q_{p+1,0}^{-p,0}.$$
The compositions $i\circ \Xi\circ i\circ\Psi$, $i\circ\Psi\circ i\circ \Xi$ are multiplication by $\eta$.
 
Now consider the diagram of maps
\begin{equation}\label{updown}\begin{matrix} \bP^1&\leftarrow&\bP^1\times X&\rightarrow&X\\
\downarrow &&\downarrow&&\downarrow\\
 \bP^1&\leftarrow&\bP^1\times \bP^1&\rightarrow&\bP^1\end{matrix}.\end{equation}
 One has the sheaves $P$ and $Q$, on the left hand side; pull them back to the central terms, and denote them by the same symbols. One can build over $\bP^1\times X$ a diagram
  \begin{equation} \label{sheaf-diagram}\xymatrix{ &Q_{0,1}^{0,1}(-F)\ar[r]^{\eta-\eta'}&Q_{0,1}^{0,1}\\P^{0,0}_{0,1}(-F)\ar[r]^{\eta -\eta' }\ar[ur]^{-i}\ar[dr]^{-i}&P^{0,0}_{0,1}\ar[ur]\ar[dr]  \\&Q_{ 1,0}^{0,0} (-F)\ar[r]^{\eta-\eta'}&Q_{ 1,0}^{0,0} \\
  Q_{ 1,0}^{0,0}( -C_0-F)^{k} \ar[ur]^{-i}\ar[ddr]^(0.7){ -\widehat B_{h,t}}\ar[r]_(0.5){\psi\bI} &  Q_{ 1,0}^{0,0}(-C_\infty ) \ar[ur]^{\xi\bI}\ar[dddr]^{\widehat B_{h,t}} &   \\
 &&   \\
Q_{0 ,0}^{ 1,0} (-C_\infty-F) \ar[dr]_{-i} \ar[uur]_(0.7){ -\widehat B_{t,h}}\ar[r]^(0.5){\xi\bI} &  Q_{0 ,0}^{ 1,0} (-C_0 ) \ar[uuur]_{\widehat B_{t,h}}\ar[dr]_{\psi\bI} \\
  &Q_{0,0}^{1,0}(-F)\ar[r]^{\eta-\eta'}&Q_{0,0}^{1,0}\\P^{0,1}_{0,0}(-F)\ar[r]^{\eta-\eta'}\ar[ur]^{-i} \ar[dr]^{-i} &P^{0,1}_{0,0}\ar[ur] \ar[dr] \\&Q_{0,1}^{0,1}(-F)\ar[r]^{\eta-\eta'}&Q_{0,1}^{0,1}.
}
\end{equation} 
\begin{remark} \label {monad diagrams} These diagrams will recur, and so some explanation of the notation is in order. We will think of them as defining monads.
Each bundle  of the monad is a direct sum of the bundles in a given column of \eqref{sheaf-diagram}.
 Accordingly, each arrow represents an entry in the matrix representing a map from the sum of each column to the sum of the next; all other entries are zero. The last line is a repeat of the first, and should be identified with it; this was done to  avoid too many crossing arrows. The unmarked arrows are the map $i$ induced by inclusion on the level of the $E^{i,j}_{k,l}$.  The coordinate $\eta'$ is the coordinate on the $\bP^1$ factor, and $\eta$ on $X$ factor. Note that the composition of the maps from the left-hand column to the middle column with the map from the middle column to the right column is zero. We will eventually see that the cohomology of this diagram at the middle column is the bundle $E$.\end{remark}

\begin{proposition} For the sheaves  of (\ref{sheaf-diagram}), the map between the  left hand side  and the middle is an   injection of sheaves, and  the restriction of the map between the middle column and the right hand column to  $P^{0,0}_{0,1}\oplus Q_{ 1,0}^{0,0}(-C_\infty )\oplus Q_{0 ,0}^{ 1,0} (-C_0 )\oplus P^{0,1}_{0,0}$ is a  surjection  for all $\psi, \xi,\eta, \eta'$ with $\xi\psi =\eta $.
 \end{proposition} 
\begin{proof} For the first statement, we have that the horizontal maps are themselves injections, so the result follows. For the second, if we have $\eta\neq \eta'$, the result follows immediately. More generally, we note that the map is $R^1\pi_*$ of a map of sheaves over $X\times X$ (distinguishing the first $X$, and the objects on it, by a prime)
$$E_{0,1}^{0,0}    \oplus E_{1,0}^{0,0} (-C_\infty-F_\psi -F'_\xi ) \oplus E_{0,0}^{1,0}(-C_0-F_\xi- F'_\psi)    \oplus E_{0,0}^{0,1} \buildrel {\mu}\over{\longrightarrow}
E_{0,1}^{0,1}  \oplus E^{0,0}_{1,0}  \oplus E^{1,0}_{0,0}$$
with 
$$\mu = \begin{pmatrix}   1&  0&0 &-1\\ -1 &\xi&\psi'& 0\\ 0& \xi'&\psi &1\end{pmatrix}.$$
The quotient of the second term by the image of $\mu$ has discrete support on the fibres of $X\times X\rightarrow \bP^1\times X$, and so $R^1\pi_*$ of the quotient is zero. This then tells us that the induced map  on $R^1\pi_*$ of the two terms above is surjective.
\end{proof}

 \subsubsection{Holomorphic Data II: Sheaves   on $\bP^1$}
We then have:
\begin{itemize} 
\item Sheaves $P_{0,0}^{0 ,1}, P_{0,1}^{0 ,0}, Q_{0,0}^{1,0} ,Q_{1,0}^{0,0}$, $Q_{0,1}^{0 ,1}$ on $\bP^1$ fitting into sequences (\ref{sheaf-sequences-2.2}), with $Q_{0,0}^{1,0} ,Q_{1,0}^{0,0}$, $Q_{0,1}^{0 ,1}$ torsion, of length $k, k, k+m$, respectively, and supported away from infinity, with  $Q_{0,0}^{1,0} ,Q_{1,0}^{0,0}$ having the same support, and indeed being isomorphic away from $\eta =0$. 
\item Shift maps 
$$\widehat B_{h,t}: Q_{ 1,0}^{ 0,0}\rightarrow Q_{0 ,0}^{1,0},\quad \widehat B_{t,h}: Q_{0,0}^{1,0}\rightarrow Q_{1,0}^{0,0},$$
such that the  compositions $\widehat B_{h,t}\circ \widehat B_{t,h}$, $\widehat B_{t,h}\circ \widehat B_{h,t}$ are multiplication by   $\eta$.
\item A trivialization of $P_{0,0}^{0 ,1}, P_{0,1}^{0 ,0}$ at $\eta =\infty$.
\item An irreducibility condition, which is the same as given in section 3.2.3, page \pageref{IrredCond}.
\item A genericity condition on maps between the sheaves  of (\ref{sheaf-diagram}), ensuring that the left hand side maps injectively to the middle, and that the  restriction of the map between the middle column and the right hand column to  $P^{0,0}_{0,1}\oplus Q_{ 1,0}^{0,0}(-C_\infty )\oplus Q_{0 ,0}^{ 1,0} (-C_0 )\oplus P^{0,1}_{0,0}$ is surjective  to the right hand side, for all $\psi, \xi,\eta, \eta'$ with $\xi\psi =\eta $.
 \end{itemize}

 The modification of the shift maps indicates  that obtaining a monad from the exact sequence for $E$ along a line is not as  straightforward  as in \eqref{caloron-sequence-projected}.  Indeed, while there are monads for bundles on these blown up surfaces (see Buchdahl \cite{Buchdahl}), they are not adapted to our purposes. Rather, we note that there are two families of lines $L^\xi_a :\{\xi= a\}$ and  $L^\psi_b: \{\psi= b\}$, each of them filling out a dense subset of the surface, and we will obtain a monad from each family, then `fuse' the two monads together. This will amount to considering the two blowdown projections 
 $$\bP^1\times \bP^1 \buildrel{\mu_1}\over{\longleftarrow} X \buildrel{\mu_2}\over{\longrightarrow} \bP^1\times \bP^1,$$
 with 
 $$  \mu_\xi(\eta, \xi,\psi) = (\eta, \xi ) \quad \mu_\psi(\eta, \xi,\psi) = (\eta,  \psi), $$
 considering monads for the pushdown $E_\xi =   (\mu_\xi)_*E$ and  $E_\psi =   (\mu_\psi)_*E$, and glueing the two.
 
 One has the resolutions, in which the maps are plus or minus the natural maps, unless otherwise indicated.
 
  \begin{equation}\label{TN-resolution-2}
  \xymatrix{ 
  & &E_{0,1}^{0,0} \ar[r]\ar[ddr] &E_{0,1}^{0,1}\\
  &&\oplus&\oplus \\
   0 \ar[r]  & &E_{1,0}^{0,0} (-F_\xi)  \ar[r]_(0.6){a } \ar[ddr]^{-\xi}&E_{1,0}^{0,0}&\ar[r]&E|_{L^\xi_{a}} \ar[r]&0\\
     &&\oplus&\oplus \\
   &&E_{0,0}^{0,1}  \ar[r]\ar[uuuur] &E_{0,0}^{1,0}} \end{equation}
 and
   \begin{equation}\label{TN-resolution-3}\xymatrix{ 
  & &E_{0,1}^{0,0} \ar[r]\ar[ddr] &E_{0,1}^{0,1}\\
  &&\oplus&\oplus \\
      0\ar[r]& &E_{0,0}^{1,0}(-F_\psi)\ar[r]_(0.6){-\psi}\ar[ddr]^(0.5){b}&E_{1,0}^{0,0}&\ar[r]&E|_{L^\psi_{b}} \ar[r]&0.\\
     &&\oplus&\oplus \\
   &&E_{0,0}^{0,1}  \ar[r]\ar[uuuur] &E_{0,0}^{1,0}
} 
\end{equation}
More globally, set
 \begin{align}
 V_\xi &=\{(\xi',p)\in \bP^1 \times X| \xi(p)= \xi'\}, \\
  V_\psi &=\{( \psi',p)\in \bP^1 \times X| \psi(p)= \psi'\}.
 \end{align}
 Denoting by $\widetilde E, \widetilde E^{ij}_{kl}$ the lifts of $ E,  E^{ij}_{kl}$ to $\bP^1 \times X,$ we have  a sequence 
\begin{equation}\label{TN-resolution-4}\xymatrix{ 
  & \widetilde E_{0,1}^{0,0} \ar[rr]\ar[ddrr] &&\widetilde E_{0,1}^{0,1}\\
  &\oplus&&\oplus \\
0 \ar[r]  & \widetilde E^{0,0}_{1,0}(-F_\xi-[\xi'=\infty])  \ar[rr]_(0.7){\xi'} \ar[ddrr]^{-\xi}&&\widetilde E_{1,0}^{0,0}&\ar[r]&\widetilde E|_{V_\xi } \ar[r]&0\\
     &\oplus&&\oplus \\
   &\widetilde E_{0,0}^{0,1}  \ar[rr]\ar[uuuurr] &&\widetilde E_{0,0}^{1,0}} \end{equation}
 and
\begin{equation}\label{TN-resolution-5}\xymatrix{ 
  & \widetilde E_{0,1}^{0,0} \ar[rr]\ar[ddrr] &&\widetilde E_{0,1}^{0,1}\\
  &\oplus&&\oplus \\
      0\ar[r]& \widetilde E^{1,0}_{0,0}(-F_\psi-[\psi'=\infty])\ar[rr]_(0.7){-\psi} \ar[rrdd]^(0.5){ \psi'}&&\widetilde E_{1,0}^{0,0}&\ar[r]&\widetilde E|_{V_\psi} \ar[r]&0.\\
     &\oplus&&\oplus \\
   &\widetilde E_{0,0}^{0,1}  \ar[rr]\ar[uuuurr] &&\widetilde E_{0,0}^{1,0}} \end{equation}
 
For the first resolution (\ref{TN-resolution-4}), we take a direct image via $\pi^\xi(\eta,\xi,\xi') =  (\eta,\xi')$ onto $\bP^1\times \bP^1$. The line $L^\xi_{\xi'}$ projects isomorphically to $\bP^1$, for $\xi'  \neq 0, \infty$ and so one is essentially getting $E$, for $\xi'\neq 0$, as well as for $\xi'=0, \eta\neq 0$ and for $\xi'=\infty, \eta\neq \infty$. In fact since the projection from $V_\xi$ to $\bP^1\times \bP^1$ is  the blowdown of $D_\xi$ and $F_\xi$, we are getting the pushdown $  \pi^\xi_* (E)$ from $X$ to $\bP^1\times \bP^1$ by $\pi^\xi(p) = (\eta,\xi)$; remembering that the bundle is trivial over $F_\infty$, we obtain  (substituting $\xi$ for $\xi'$, and with the abuse of notation that $P_{i,j}^{k,l}, P_{i,j}^{k,l}$ denote both the sheaves on $\bP^1$ and  their lifts to $\bP^1\times \bP^1$ or $X$):
\begin{equation}\label{pushdownxi}
0\rightarrow  \pi^\xi_* (E) \rightarrow  
	{\scriptstyle   
P_{0,1}^{0,0}\oplus Q_{1,0}^{0,0}(-C_\infty)\oplus P_{0,0}^{0,1}  \xrightarrow{\left(\begin{smallmatrix}1&0&-1\\-1&\xi&0\\ 0&-\widehat B_{h,t}&1\end{smallmatrix}\right)}  Q_{0,1}^{0,1}\oplus Q_{1,0}^{0,0}\oplus  Q_{0 ,0}^{ 1,0} 
	}
\rightarrow R^1\pi^\xi_*(E).
\end{equation}
We note that the support of $R^1\pi^\xi_*(E)$, if it is non-empty, is at $\xi = 0$, and so
\begin{proposition} \label{gen1} The map 
$$  P_{0,1}^{0,0}\oplus Q_{1,0}^{0,0}(-C_\infty)\oplus P_{0,0}^{0,1} \xrightarrow{\begin{pmatrix}1&0&-1\\-1&\xi&0\\ 0&-\widehat B_{h,t}&1\end{pmatrix}} Q_{0,1}^{0,1}\oplus Q_{1,0}^{0,0}\oplus  Q_{0 ,0}^{ 1,0}$$
arising from a bundle $E$ is surjective away from $\xi= 0$.
\end{proposition}

In turn, taking the second resolution (\ref{TN-resolution-5}), we get  sheaves and maps:

\begin{equation}\label{pushdownpsi}
\scriptstyle
0\rightarrow    \pi^\psi_* (E) \rightarrow  P_{0,1}^{0,0}\oplus Q^{1,0}_{0,0}(-C_0)\oplus P_{0,0}^{0,1}  \xrightarrow{\left(\begin{smallmatrix}1&0&-1\\-1&-\widehat B_{t,h}&0\\ 0&\psi&1\end{smallmatrix}\right)}  Q_{0,1}^{0,1}\oplus Q_{1,0}^{0,0}\oplus  Q_{0 ,0}^{ 1,0} \rightarrow R^1\pi^\psi_*(E).
\end{equation}
Again, the support of $R^1\pi^\psi_*(E)$, if it is non-empty, is at $\psi = 0$, and so
\begin{proposition} \label{gen2} The map 
$$  P_{0,1}^{0,0}\oplus Q^{1,0}_{0,0}(-C_0)\oplus P_{0,0}^{0,1}  \xrightarrow{\begin{pmatrix}1&0&-1\\-1&-\widehat B_{t,h}&0\\ 0&\psi&1\end{pmatrix}}  Q_{0,1}^{0,1}\oplus Q_{1,0}^{0,0}\oplus  Q_{0 ,0}^{ 1,0},$$
arising from a bundle $E,$ is surjective away from $\psi= 0$.
\end{proposition}

Now, let us take resolutions of the $P, Q$, as for the caloron. Lifting back to $X$, this gives the following commutative diagram, where $D= D_\xi+D_\psi$, $F = F_\xi + F_\psi$, and where one remembers that the $Q$s are supported away from $F$:
\begin{equation} \label{resolution-PQ-TN-4}
    \xymatrix@C=1pc{  
 & \pO(-F)^{k+m} \ar[r]^-{Z_{ 0,1} } &\pO^{k+m}\ar[rr]& &Q_{0,1}^{0,1}\\
 \pO(-F)^{k+m}\ \  \ar[dr]^{X_{-,0} }\ar[r]^-{W_{ -}}\ar[ur]^{X_{-,1} }& \pO^{k+m+1} \ar[dr]^{Y_{ -,0} }\ar[ur]^{Y_{ -,1} }\ar[rr]&&P_{0,1}^{0,0} \ar[ur]\ar[dr] \\ 
 &  \pO(-F)^{k } \ar[r]^-{Z_{ 1,0}} &\pO^{k}\ar[rr]& &Q_{ 1,0}^{0,0} \\ 
 \pO(-F-C_\infty-F_\xi)^{k}\ar[ur]^-{\xi\bI} \ar[dr]^{ B_{h,t}}\ar[r]^-{Z_{ 1,0}} &  \pO(-C_\infty-F_\xi)^{k}\ar[ur]^{\xi\bI}  \ar[rd]^{ B_{h,t}}\ar[rr] & &Q_{ 1,0}^{0,0}(-C_\infty) \ar[ur]^\xi  \ar[rd]^{ \widehat B_{h,t}} \\
 & \pO(-F)^{k }\ar[r]^-{Z_{0,0}} &\pO^{k} \ar[rr]& &Q_{0 ,0}^{ 1,0} \\
 \pO(-F)^{k }\ \ \ar[dr]^{X_{ +,1}  } \ar[r]^{W_{+}}\ar[ur]^{X_{+,0} }&   \pO^{k+1}\ar[dr]^{Y_{+,1}}\ar[ur]^{Y_{+,0}}\ar[rr]&   &P_{0,0}^{0,1}\ar[dr] \ar[ur]\\
 & \pO(-F)^{k+m} \ar[r]^-{Z_{ 0,1} } &\pO^{k+m}\ar[rr]& &Q_{0,1}^{0,1}.
   }
\end{equation}
As explained in the caloron case, we get a monad from this diagram by summing each of the three columns on the left, and changing signs on the diagonals between the first and second column, so that the diagram is anti-commutative instead of commutative. Let us make these sign changes from now on. The  cohomology of the monad is $(\pi^\xi)^*\pi^\xi_*(E)= E_\xi$,

Again, taking resolutions gives  a diagram, and hence an analogous monad for $E_\psi = (\pi^\psi)^*\pi^\psi_*(E)$
    \begin{equation} \label{resolution-PQ-TN-5}\xymatrix@C=1pc{  
 & \pO(-F)^{k+m} \ar[r]^-{Z_{ 0,1} } &\pO^{k+m}\ar[rr]& &Q_{0,1}^{0,1}\\
 \pO(-F)^{k+m}\ \  \ar[dr]^{-X_{-,0} }\ar[r]^-{W_{ -}}\ar[ur]^{-X_{-,1} }& \pO^{k+m+1} \ar[dr]^{Y_{ -,0} }\ar[ur]^{Y_{ -,1} }\ar[rr]&&P_{0,1}^{0,0} \ar[ur]\ar[dr] \\ 
 &  \pO(-F)^{k } \ar[r]^-{Z_{  1,0}} &\pO^{k}\ar[rr]& &Q_{ 1,0}^{0,0} \\ 
 \pO(-F-C_0-F_\psi)^{k}\ar[dr]^{-\psi\bI} \ar[ur]^{ -B_{t,h}}\ar[r]^-{Z_{0,0}} &  \pO(-C_0-F_\psi)^{k}\ar[dr]^{\psi\bI}  \ar[ru]_{ B_{t,h}} \ar[rr]& &Q^{ 1,0}_{0,0}(-C_0)\ar[dr]^\psi \ar[ru]_{\widehat B_{t,h}} \\
 & \pO(-F)^{k }\ar[r]^-{Z_{0,0}} &\pO^{k} \ar[rr]& &Q_{0 ,0}^{1,0} \\
 \pO(-F)^{k }\ \ \ar[dr]^{-X_{ +,1}  } \ar[r]^{W_{+}}\ar[ur]^{-X_{+,0} }&   \pO^{k+1}\ar[dr]^{Y_{+,1}}\ar[ur]^{Y_{+,0}}\ar[rr]& &P_{0,0}^{0,1}\ar[dr] \ar[ur]\\
 & \pO(-F)^{k+m} \ar[r]^-{Z_{ 0,1} } &\pO^{k+m}\ar[rr]& &Q_{0,1}^{0,1}\\
 }
\end{equation}
 This, over $\eta\neq 0$, is  isomorphic to the monad 
   \begin{equation} \label{resolution-PQ-TN-6}\xymatrix@C=1pc{  
 & \pO(-F)^{k+m} \ar[r]^-{Z_{ 0,1} } &\pO^{k+m}\ar[rr]& &Q_{0,1}^{0,1}\\
 \pO(-F)^{k+m}\ \  \ar[dr]^{-X_{-,0} }\ar[r]^-{W_{ -}}\ar[ur]^{-X_{-,1} }& \pO^{k+m+1} \ar[dr]^{Y_{ -,0} }\ar[ur]^{Y_{ -,1} }\ar[rr]& &P_{0,1}^{0,0} \ar[ur]\ar[dr] \\ 
 &  \pO(-F)^{k } \ar[r]^-{Z_{ 1,0}} &\pO^{k}\ar[rr]& &Q_{ 1,0}^{ 0,0} \\ 
 \pO(-F-C_0-F_\psi)^{k}\ar[dr]^{-\xi^{-1}B_0} \ar[ur]^{ -B_{t,h}}\ar[r]^-{Z_{0,0}} &  \pO(-C_0-F_\psi)^{k}\ar[dr]^{\xi^{-1}B_0}  \ar[ru]_{ B_{t,h}} \ar[rr]& &Q^{ 1,0}_{0,0}(-C_0) \ar[dr]^{ \eta \xi^{-1}} \ar[ru]_{\widehat B_{t,h}} \\
 & \pO(-F)^{k }\ar[r]^-{Z_{0,0}} &\pO^{k}\ar[rr]& &Q_{0 ,0}^{ 1,0} \\
 \pO(-F)^{k }\ \ \ar[dr]^{-X_{ +,1}  } \ar[r]^-{W_{+}}\ar[ur]^{-X_{+,0} }&   \pO^{k+1}\ar[dr]^{Y_{+,1}}\ar[ur]^{Y_{+,0}}\ar[rr]& &P_{0,0}^{0,1}\ar[dr] \ar[ur]\\
 & \pO(-F)^{k+m} \ar[r]^-{Z_{0,1} } &\pO^{k+m}\ar[rr]& &Q_{0,1}^{0,1}.
 }
\end{equation}
The isomorphism is achieved by maps which, on the central (second) column, map 
sections $(u_1, u_2, u_3, u_4,u_5, u_6)$ to $(u_1, u_2, u_3, u_4,u_5 +\xi^{-1} u_4, u_6)$.
The two monads \eqref {resolution-PQ-TN-4} and \eqref{resolution-PQ-TN-6} in turn are the same apart from a central piece. One then has a map between the two which is the identity   on these identical pieces, and on the central piece, corresponding to the sheaves $Q_{ 1,0}^{0,0} ,Q_{0,0}^{ 1,0},$  a morphism


\begin{equation*} \label{resolution-PQ-TN-7.5}
 \xymatrix{      \pO(-F-C_0-F_\psi)^{k} \ar[r]^{Z_{0,0}}  \ar[ddr]^{\xi^{-1}B_{t,h}}
 &\pO(-C_0-F_\psi)^{k}  \ar[ddr]^{\xi^{-1}B_{t,h}}\\  \\
 &\ \pO(-F-C_\infty-F_\xi)^{k}\ar[r]^{Z_{1,0}}&\pO(-C_\infty-F_\xi)^{k}.
}
\end{equation*}
Note, that the $ Q_{ 1,0}^{0,0} ,Q_{0,0}^{ 1,0}$ have the same support: the projections of the lines $\eta$=constant which are jumping lines.  On the level of the $Q$s, $B_{h,t}$ is the map induced on sections  by multiplication by $\xi;$ in turn $B_{t,h}$ is induced by multiplication by  $\psi,$ and $B_0, B_1$ are induced   by $\eta.$ One has that $Z_{0,0}= \eta-B_0$, $Z_{1,0}= \eta-B_1$; since $\eta= \psi\xi$, we have the condition 
\begin{align} B_0 &= B_{h,t}B_{t,h}\\B_1&= B_{t,h}B_{h,t},\end{align} 
ensuring the necessary commutation in the diagrams above.

This monad morphism realizes on $\eta\neq 0 $ the isomorphisms  $E_\xi=E=E_\psi$; indeed $E_\xi$  is isomorphic to  $E$ away from $D_\xi$, and  $E_\psi$ is isomorphic to  $E$ away from $D_\psi$. 

We would like to `fuse' the two monads, to give us $E$. What   works is:

    \begin{equation} \label{resolution-PQ-TN-9}\xymatrix{  
 & \pO(-F)^{k+m} \ar[r]^{Z_{ 0,1} } &\pO^{k+m}\ar[rr]& &Q_{0,1}^{0,1}\\
 \pO(-F)^{k+m}\ \  \ar[dr]^{-X_{-,0} }\ar[r]^{W_{ -}}\ar[ur]^{-X_{-,1} }& \pO^{k+m+1} \ar[dr]^{Y_{ -,0} }\ar[ur]^{Y_{ -,1} }\ar[rr]& &P_{0,1}^{0,0} \ar[ur]\ar[dr] \\ 
 &  \pO(-F)^{k } \ar[r]^(0.6){Z_{ 1,0}} &\pO^{k}\ar[rr]& &Q_{ 1,0}^{0,0} \\ 
  \pO(-F-C_0)^{k}\ar[ur]^{- \bI} \ar[ddr]^(0.7){-B_{h,t}}\ar[r]_(0.5){\psi\bI} &  \pO(-C_\infty-F_\xi)^{k}\ar[ur]^{\xi\bI}  \ar[rddd]^(0.7){ B_{h,t}} \ar[rrd] && \\
 &&&Q \ar[uur]  \ar[rdd]\\
   \pO(-F-C_\infty)^{k}\ar[dr]^{-\bI} \ar[uur]_(0.7){- B_{t,h}}\ar[r]^(0.5){\xi\bI} &  \pO(-C_0-F_\psi)^{k}\ar[dr]^{\psi\bI}  \ar[ruuu]_(0.7){ B_{t,h}}\ar[urr] & \\
 & \pO(-F)^{k }\ar[r]^{Z_{0,0}} &\pO^{k} \ar[rr]& &Q_{0 ,0}^{ 1,0} \\
 \pO(-F)^{k }\ \ \ar[dr]^{-X_{ +,1}  } \ar[r]^{W_{+}}\ar[ur]^{-X_{+,0} }&   \pO^{k+1}\ar[dr]^{Y_{+,1}}\ar[ur]^{Y_{+,0}}\ar[rr]& &P_{0,0}^{0,1}\ar[dr] \ar[ur]\\
 & \pO(-F)^{k+m} \ar[r]^{Z_{ 0,1} } &\pO^{k+m}\ar[rr]& &Q_{0,1}^{0,1}.
 }
\end{equation}

The left hand side is the direct image under $\bP^1\times X\rightarrow X$ of the diagram of sheaves \eqref{sheaf-diagram}. Using the  resolution of the diagonal in $\bP^1\times \bP^1$ by $\pO(-1,-1)\rightarrow \pO\rightarrow \pO|_\Delta$, one has that the terms on the right hand side (except for the fourth term, $Q$) are  indeed  the sheaves 
$Q_{0,1}^{0,1}, P_{0,1}^{0,0} , Q_{ 1,0}^{0,0},Q_{0 ,0}^{ 1,0} , P_{0,0}^{0,1}$; for the  remaining sheaf $Q$, mapping to $  Q_{ 1,0}^{0,0},   Q^{ 1,0}_{0,0}$, we define it as the quotient
  \begin{equation} \label{define-Q2}\xymatrix{  
  Q_{ 1,0}^{0,0}( -C_0)   \ar[ddr]^(0.7){ -\widehat B_{h,t}}\ar[r]_(0.5){\psi\bI} &  Q_{ 1,0}^{0,0}(-C_\infty )  \ar[rd]  \\
 &&Q. \\
Q_{0 ,0}^{ 1,0} (-C_\infty)  \ar[uur]_(0.7){-\widehat B_{t,h}}\ar[r]^(0.5){\xi\bI} &  Q_{0 ,0}^{ 1,0} (-C_0 )  \ar[ur] \\ }
\end{equation}
A bit of diagram chasing shows that it can also be defined  by 
    \begin{equation} \label{define-Q}\xymatrix{  
  \pO(-F-C_0)^{k}  \ar[ddr]^(0.7){ -B_{h,t}}\ar[r]_(0.5){\psi\bI} &  \pO(-C_\infty-F_\xi)^{k} \ar[rd]  \\
&&Q. \\
  \pO(-F-C_\infty)^{k} \ar[uur]_(0.7){ -B_{t,h}}\ar[r]^(0.5){\xi\bI} &  \pO(-C_0-F_\psi)^{k}\ar[ur] \\ }
\end{equation}

The diagram (or monad) (\ref{resolution-PQ-TN-9}) contains the monads \eqref{resolution-PQ-TN-4} and \eqref{resolution-PQ-TN-5} as sub-monads; 
for (\ref{resolution-PQ-TN-4}), one maps the central columns of (\ref{resolution-PQ-TN-4}) to those of (\ref{resolution-PQ-TN-9}) by 
$$(u_1, u_2, u_3, u_4,u_5, u_6)\mapsto (u_1, u_2, u_3 , u_4  , 0, u_5, u_6),$$
and on the first column by 
$$(v_1, v_2, v_3) \mapsto (v_1, \xi v_2, B_{h,t}v_2, v_3)$$
and  similarly for (\ref{resolution-PQ-TN-5}).

Now let us start from   holomorphic data II. One can take the locally free resolutions of $P, Q$ as above, and build sequences (\ref{resolution-PQ-TN-4},\ref{resolution-PQ-TN-5},\ref{resolution-PQ-TN-9}).   One has a proposition that can be proven for the resolutions and the monads
\begin{proposition}\label{genericity-maps}
For the diagram-monad \eqref{resolution-PQ-TN-9}, arising from  holomorphic data II, one has
\begin{itemize}
\item{}The maps $X_{-,1}, X_{-,0}, Y_{-,1}, Y_{-,0}, X_{+, 0} Y_{+,0} $ are surjective.
\item{} The maps $X_{+,1}, Y_{+,1}$ are injective.
\item{} The maps between the second and third terms in the monad  is surjective at every point.
\item{} The map between the first and second term in the monads is injective at every point.
\end{itemize}
\end{proposition}
 
\begin{proof}
For the first two items, one uses the long exact sequence of (\ref{sheaf-sequences-2.2}), and of their twists by $\pO(-1)$. For the third item, we have by our genericity property that the map between the second and third column is surjective on the level of sheaves, i.e. in \eqref{sheaf-diagram}; we want it to be surjective on the level of global sections over $\bP^1$. We note that the map on the level of sheaves can be written schematically as $P\oplus Q(-F)\rightarrow Q\rightarrow 0$; this fits into an exact sequence $0\rightarrow P(-F)\rightarrow P\oplus Q(-F)\rightarrow Q\rightarrow 0$. On the other hand, from the properties (\ref{sheaf-sequences-2.2}), one finds that $H^1(\bP^1,P(-1)) = 0$, guaranteeing surjectivity. For the fourth, one simply notes that one has an injection of sheaves, giving an injection on the level of sections.\end{proof}

Now we consider the sheaf defined by the monad. We note that the surjectivity and injectivity given above show that it is a bundle; furthermore, the fact that it is   $E_\xi$ over $\xi\neq 0$ and $E_\psi$ over $\psi\neq 0$ guarantees that the bundle is isomorphic to $E$, i.e the bundle we started out with. In short:

\begin{proposition} Holomorphic data I and II are equivalent.\end{proposition}

One can work out as for the caloron \eqref{matrices-caloron} and \eqref{matrices-caloron-2} the maps in the corresponding resolutions. One obtains essentially the same expressions, except that the map $Z_{1,0}$  associated to $Q_{ 1,0}^{0,0}$ and the map $Z_{0 ,0}$ associated to $Q_{0 ,0}^{1,0}$ are no longer the same. Rather, one has
$$Z_{0 ,0}= (\eta-B_{0}),\quad  Z_{1,0} = (\eta-B_{1}),$$
with 
$$ B_{0}= B_{h,t}B_{t,h},\quad B_1=B_{t,h}B_{h,t}.$$
This follows from the fact that one has the relation on coordinates $\eta = \xi\psi$. We notice that if the matrices $B_{t,h},B_{h,t}$ are invertible, the $k\times k$ matrices $B_{ 0},B_{1}$ are conjugate. 

More precisely, one can normalize to matrices:

\subsubsection{Holomorphic Data III: Matrices, up to the action of $Gl(k)$}  
\begin{equation} 
\begin{matrix}
X_{+,1}&= & \begin{pmatrix}A\\A'\end{pmatrix},&&Y_{+,1}&=& \begin{pmatrix}A&C_{2}\\A'&C'_{2}\end{pmatrix},\\ \\
W_{+}&=& \begin{pmatrix}\eta-B_0\\-D_{2}\end{pmatrix},&&Z_{0,1} &=&\begin{pmatrix}\eta-B_{1}&C_{1}e_+\\ (e_-)^TB'&(\eta-\Sh)+C'_{1}e_+\end{pmatrix},\\ \\
X_{-,1}&=& \begin{pmatrix}1&0\\0&1\end{pmatrix},&& Y_{-,1}&=& \begin{pmatrix}1&0&-C_{1}\\0&1&-C'_{1}\end{pmatrix},\\ \\
X_{-,0}&=& \begin{pmatrix}1&0\end{pmatrix},&&Y_{-,0}&=& \begin{pmatrix}1&0&0\end{pmatrix},\\ \\
W_{-}&=& \begin{pmatrix}\eta-B_{1}&0\\ -(e_-)^TB'&\eta-\Sh\\0&-e_+ \end{pmatrix},&&Z_{0,0}&=&\begin{pmatrix}\eta-B_{0}\end{pmatrix},\\ \\
X_{+,0}&= &\begin{pmatrix}1 \end{pmatrix},&&Y_{+,0}&=& \begin{pmatrix}1&0\end{pmatrix}. 
\end{matrix}\label{matrices-TN}
\end{equation}

Here $A,B_i,C ,D ,A ',B ',C' $ are matrices of size $k\times k,k\times k, k\times 2,2\times k, m\times k,1\times k, m\times 2$ respectively. For $C,C' ,D$, the subscripts denote columns or rows, where appropriate.

As noted, one has the constraints
\begin{align}
  B_{0}&= B_{h,t}B_{t,h},& B_1&=B_{t,h}B_{h,t} .
\end{align}
In addition, there are constraints given by the commutativity of the diagram (\ref{resolution-PQ-TN-9}): setting as above $D_{ 1} = e_+A' $, one has

\begin{align}\label{monad-Taub}
A B_0-B_{1}A  + C D &=0,                                           \\
  (e_-)^TB'A  + \Sh A ' - A 'B_0 -C'D&=0,   \\
   -e_+ A'+   \begin{pmatrix}1&0\end{pmatrix}   D&=0.                                    
  \end{align}
Again, there are genericity conditions. These arise from two sources: first, the monad (\ref{resolution-PQ-TN-9}) must  have surjective maps from the second column to the third column, and injective maps from the first column to the second. The other is that using the matrices to define sheaves $P, Q$ from the formulae in the resolutions above, one should get sequences (\ref{sheaf-sequences-2.2}). 

\subsubsection{Holomorphic Data IV: Monads}
This data has already been given. It is essentially the diagram  (\ref{resolution-PQ-TN-9}), without the $P$s and $Q$s,  with an implicit sum along every column, and with its top and bottom lines identified.

   \begin{equation} \label{resolution-PQ-TN-10}\xymatrix{  
 & \pO(-F)^{k+m} \ar[r]^{Z_{ 0,1} } &\pO^{k+m}\\
 \pO(-F)^{k+m}\ \  \ar[dr]^{-X_{-,0} }\ar[r]^{W_{ -}}\ar[ur]^{-X_{-,1} }& \pO^{k+m+1} \ar[dr]^{Y_{ -,0} }\ar[ur]^{Y_{ -,1} }\\ 
 &  \pO(-F)^{k } \ar[r]^(0.6){Z_{ 1,0}} &\pO^{k} \\ 
  \pO(-F-C_0)^{k}\ar[ur]^{- \bI} \ar[ddr]^(0.7){-B_{h,t}}\ar[r]_(0.5){\psi\bI} &  \pO(-C_\infty-F_\xi)^{k}\ar[ur]^{\xi\bI}  \ar[rddd]^(0.7){ B_{h,t}}  \\
 &&&\\
   \pO(-F-C_\infty)^{k}\ar[dr]^{-\bI} \ar[uur]_(0.7){- B_{t,h}}\ar[r]^(0.5){\xi\bI} &  \pO(-C_0-F_\psi)^{k}\ar[dr]^{\psi\bI}  \ar[ruuu]_(0.7){ B_{t,h}}  \\
 & \pO(-F)^{k }\ar[r]^{Z_{0,0}} &\pO^{k}   \\
 \pO(-F)^{k }\ \ \ar[dr]^{-X_{ +,1}  } \ar[r]^{W_{+}}\ar[ur]^{-X_{+,0} }&   \pO^{k+1}\ar[dr]^{Y_{+,1}}\ar[ur]^{Y_{+,0}}\\
 & \pO(-F)^{k+m} \ar[r]^{Z_{ 0,1} } &\pO^{k+m}.
 }
\end{equation}

\subsection{Bow complexes and monads}\label{mpos}

Our aim in this section is to show
\begin{theorem} Holomorphic data I-IV above are equivalent to a bow complex.
\end{theorem}

As we saw from the point of view of the bow solution, our instantons are encoded  by a bow complex, as in section \ref{Bow-complexes}, with, as for calorons, the holomorphic bundle being encoded by an infinite dimensional monad. As before, solutions on an interval $[c, c+\ell]$ containing $\lambda_\pm$,  satisfying appropriate continuity constraints at $\lambda_\pm$, which are the same as for the caloron. The difference is that instead of being periodic, the solutions at $c, c+\ell$ are identified by adding four auxiliary spaces $\widehat U_h,\widehat  U_t, \widehat V_h, \widehat  V_t$ of dimension $k$ to the complex:

\begin{equation} \label{infmonad}
 \widetilde{H}_{1}(V)  \xrightarrow{D_1 = \left(\begin{smallmatrix}  d_s + \alpha -(t_3+it_0), \\  \beta-\eta\\ \psi ev_{c}+ B_{h,t} ev_{c+\ell}\\ \xi ev_{c+\ell}+B_{t,h} ev_{c} \end{smallmatrix}\right)}\begin{matrix} L^2(V)\\ \oplus \\ L^2(V) \\ \oplus\\ \widehat U_h\\ \oplus \\ \widehat  U_t\end{matrix}\xrightarrow{D_2 = \left(\begin{smallmatrix}  -\beta+ \eta &d_s + \alpha -(t_3+it_0)&0&0\\0&ev_c&\xi&  -B_{h,t}\\0&ev_c& - B_{t,h}&\psi \end{smallmatrix}\right)} \begin{matrix}\widetilde{H}_{-1}\\ \oplus\\  \widehat V_h\\ \oplus \\ \widehat  V_t \end{matrix}.  \end{equation}
Recalling that 
$$\beta(c) = B_{h,t}\circ B_{t,h},\quad \beta(c+\ell) = B_{t,h}\circ B_{h,t},$$
one finds that this is indeed a monad.

Our aim is to reduce this to a finite dimensional monad, as we did for calorons. Let $(\chi_1, \chi_2, \widehat u_h, \widehat u_t)$ denote a cocycle; again, as for the caloron, modifying it by a coboundary, we can suppose that $\chi_1$ is supported away from $\lambda_+, \lambda_-$, and $c, c+\ell$.

Consider the spaces:
\begin{itemize}
\item $\widetilde V_+$ of solutions to $ (d_s  +  \alpha)v_+ = 0$  on the interval $(\lambda_--\ell,\lambda_-)$ which satisfy the boundary conditions at $\lambda_+$;
\item  $\widetilde V_-$  of solutions to $ (d_s  +  \alpha)v_- = 0$  on the interval $(\lambda_+ - \ell,\lambda_+ )$ which satisfy the boundary conditions at $\lambda_-$;
\item Subspaces $\widetilde  U_\pm$ of $\widetilde  V_\pm$, such that not only $v_\pm$ satisfy the boundary conditions at $\lambda_\pm$ but also $\beta(v_\pm)$; as $\beta$ has a pole at the boundary, this gives a space that is one dimension smaller.
\item $\widetilde U_{0,-}$ of values $u_0 = u(c)$  at   a fixed point $c $ of solutions on $(\lambda_+ -\ell , \lambda_-)$ to the 
equation 
\begin{equation}(d_s  +  \alpha)u (s) = \chi^1,\label{chioneTN}\end{equation}  
with $u(\lambda_-) =0$; $\widetilde U_0$, of course, is just the fibre of $N$ at $c$;
\item $\widetilde U_1$ of values $u_1 = u (\lambda_+-\epsilon)$ of  solutions  on $(\lambda_-, \lambda_+)$ to the 
equation $(d_s  +  \alpha)u(s) = \chi^1 $, with $u(\lambda_-+\epsilon) =0$. 
\item $\widetilde U_{0,+}$ of values $u_0 = u(c+\ell)$  at   a fixed point $c+\ell $    of solutions on $(\lambda_+ -\ell , \lambda_-)$ to the 
equation 
\begin{equation}(d_s  +  \alpha)u (s) = \chi^1,\label{chitwoTN}\end{equation}  
with $u(\lambda_+) =0$; $\widetilde U_0$, of course, is just the fibre of $N$ at $c$;
\item $\widetilde V_{0,\pm} =\widetilde  U_{0,\pm}$, 
\item $\widetilde V_1 =\widetilde  U_1$.
 \item Spaces $\widetilde  U_h, \widetilde  W_h, \widetilde  U_t, \widetilde  W_t$ of dimension $k$.
\end{itemize}

Our reduction of the monad goes by the procedure used for the caloron; we solve $(d_s  +  \alpha)u (s) = \chi^1$ on the intervals, with initial condition $\chi^1 = 0$ on one end of the interval, and set $
\chi_2 =   v_\pm + (\beta-\eta) u$ on the intervals; this gives a monad that is very close to that for the caloron, but for one thing:
  the cyclicity condition for the caloron gets replaced by the glueing condition coming from our infinite dimensional monad:
$$ ev_c(v_-) + ev_c ((\beta-\eta) u) + \xi u_{h} -B_{h,t}( u_{t}) = 0 $$
in $V_{0,-}$, with $u_{h}\in U_{h}, u_{t}\in U_{t}$, and in $V_{0,+}$,
$$ ev_{c+\ell}(v_+) + ev_{c+\ell} ((\beta-\eta) u) + \xi u_{h} -B_{h,t}( u_t) = 0. $$
These can be modified by coboundaries. Writing an element $s$ of the left hand side of the infinite dimensional monad as $u_\pm + r$, with $u_\pm\in U_\pm$ and $r= 0$ at $\lambda_\pm$, and setting $w_h, w_t$ to be $ev_c(s), ev_{c+\ell}(s)$, the coboundary map changes  $v_{\pm}$ by $(\beta-\eta)(u_{\pm})$, and change $u_{0,-}$ by an arbitrary  $w_h \in  \widetilde  W_h $, and $u_{0,+} $ by an arbitrary  $w_t\in  \widetilde  W_t$; but then, however, $u_h, u_t$ get  modified by 
$$ \begin{pmatrix}u_h\\ u_t\end{pmatrix} \mapsto \begin{pmatrix}u_h\\ u_t\end{pmatrix} + \begin{pmatrix}\psi&   B_{h,t}\\ B_{t,h}&\xi \end{pmatrix}\cdot \begin{pmatrix}w_h\\ w_t\end{pmatrix}.$$
Once one does this,  inserting appropriate twists by line bundles, so that the maps remain finite at infinity, one obtains an anti-commutative diagram
   \begin{equation} \label{resolution-PQ-TN-10-twistor}\xymatrix{  
 & \widetilde U_1(-F)\ar[r]^{\beta_1-\eta} & \widetilde V_1\\
\widetilde U_+(-F) \ \  \ar[dr]^{-ev_c }\ar[r]^{\beta-\eta}\ar[ur]^{-ev_1 }& \widetilde V_+ \ar[dr]^{ev_c}\ar[ur]^{ev_1}\\ 
 & \widetilde U_{0,-}(-F)  \ar[r]^(0.6){\beta(c) -\eta} &\widetilde V_{0,-} \\ 
 \widetilde W_h(-F-C_0) \ar[ur]^{  \bI} \ar[ddr]^(0.7){-B_{h,t}}\ar[r]_(0.5){\psi\bI} & \widetilde  W_h(-C_\infty-F_\xi) \ar[ur]^{\xi\bI}  \ar[rddd]^(0.7){ B_{h,t}}  \\
 &&&\\
 \widetilde W_t(-F-C_\infty) \ar[dr]^{ \bI} \ar[uur]_(0.7){- B_{t,h}}\ar[r]^(0.5){\xi\bI} & \widetilde W_t(-C_0-F_\psi) \ar[dr]^{\psi\bI}  \ar[ruuu]_(0.7){ B_{t,h}}  \\
 & \widetilde U_{0,+}(-F)^{k }\ar[r]^{\beta(c+\ell)} & \widetilde V_{0,+}   \\
 \widetilde U_+(-F) \ \ \ar[dr]^{-ev_1 } \ar[r]^{\beta-\eta}\ar[ur]^{-ev_{c+\ell} }&    \widetilde U_-\ar[dr]^{ev_1}\ar[ur]^{ev_{c+\ell}}\\
 & \widetilde U_1(-F)\ar[r]^{Z_{ 0,1} } & \widetilde V_1.} \end{equation}
 
 Again, this is not quite the same monad  as that produced by the algebraic geometry; there is the same issue as for the caloron: the dimensions of $\widetilde  U_\pm, \widetilde  V_\pm$ do not coincide. However, once one adjusts by the maps of (\ref{ajustment}), one can define a monad morphism between the two which yields the same bundle as its cohomology. We would like to remark that this modification, essentially replacing the sheaves $P$ by twists of $P$ along $\eta =\infty$, corresponds essentially to different choices of compactification of 
 the bundle $E$ along the divisor $F$.

 The link between monads and bundles can be viewed in the same way as for bundles as in the caloron case, as expounded in section 3.2.5. On the other hand, there is a more sophisticated way of linking the  two, involving a spectral sequence, which we will briefly sketch. 
Let us rearrange the monad as the following commutative diagram with all horizontal and slanted sequences exact\newline
\begin{align}\label{primitive-monad}
\xymatrix @C=2pc{
 &0\ar[dr]\\
&&{E}\ar[dr]\\
0\ar[r]&H^1(N\otimes e^*)\ar[dr]_-{D}
\phantom{\begin{matrix} L\\W\\ e\\ e\end{matrix}}
\ar@<4.6ex>[r]^-{Z}
\ar@<2.6ex>[r]^-{J}
\ar@<.3ex>[r]^-{B_{ht}}
\ar[r]+(-9,-6)^{-b_{th}}
\ar[]!<0ex,-5.5ex>;[r]!<0ex,-5.5ex>+(-9,6)^(0.9){b_{ht}}
\ar@<-5.5ex>[r]^-{B_{th}}
& \boxed{\begin{matrix} L^2\\W\\N_h\otimes e_t^*\\N_t\otimes e_h^*\end{matrix}}\ar[r]
\ar@<4.6ex>[dr]^(0.65){D}
\ar@<1ex>[dr]^-{\delta^\lambda I}
\ar[dr]_-{\begin{smallmatrix}[\delta^h(-b_{th},B_{ht})\\-\delta^t(B_{th},b_{ht})]\end{smallmatrix}}
& P\ar[r]\ar[dr] & 0\\
&0\ar[r] & L^2(N\otimes e^*)\ar[r]_-{Z}&H^{-1}(N\otimes e^*)\ar[r]&Q\ar[r]\ar[dr]&0\\
& &&&&0
}
\end{align}

The identification of $E =\mathrm{Ker}\, (P\rightarrow Q)$ with the middle cohomology of the complex \eqref{FirstBowMonad} is via the standard  diagram chasing:
\begin{align}\label{chasing-monad}
\xymatrix @C=3pc{
&&*+[o][F-]{v} 
\ar@{|->}[]!<0.5ex,-.5ex>;[dr]!<0ex,.5ex>+(-.5,1)\\
&
{ }\POS p+(-5,7) *+{\chi}="chi"
\phantom{\begin{matrix} L\\W\\ e\\ e\end{matrix}}
& \tilde{\Psi}_1:=\boxed{\begin{smallmatrix} \Psi_1
\\f
\\ \upsilon_+
\\ \upsilon_-
\end{smallmatrix}}
\ar@{|->}[r]
\ar@{|->}[dr]^{\tilde{D}}
\POS p+(-10,7) *+{\delta\tilde{\Psi}_1}="deltaPsi1" 
\ar@{|.>}[r]+(-16,7)
& u
\ar@{|..>}[r]
\ar@{|..>}[dr] 
\POS p+(-14,7) *+{0}="uzero" 
& 0\\
&
& \Psi_2 
\ar@{|->}[r]_-{Z}
\POS p+(-10,5) *+{\delta\Psi_2=D\chi}="deltaPsi2" 
& \tilde{D}\tilde{\Psi}_1=Z\Psi_2
\ar@{|..>}[r]
&0.
\ar@{|->} "chi";"deltaPsi1"
\ar@{|->} "chi";"deltaPsi2"
}
\end{align}
Given $v\in E$ its image $u$ in $P$ has some preimage $\tilde{\Psi}_1$ under the horizontal map, that is part of that horizontal exact sequence. In turn, $\tilde{D}\tilde{\Psi}_1$ has a preimage $\Psi_2.$ This ensures that the pair $(\tilde{\Psi}_1,\Psi_2)$ is a middle cocycle of the monad.  At the same time, $\tilde{\Psi}_1$ is only defined modulo some $\delta \tilde{\Psi}_1$ that is annihilated by the horizontal map and, therefore, has to be an image of some $\chi,$ due to exactness. Then, changing  $\Psi_2$  by $D\chi$ produces the correct preimage of the new $\tilde{D}(\tilde{\Psi}_1+\delta \tilde{\Psi}_1).$  Since the lowest horizontal sequence is exact, this change in $\Psi_2$ is unique and is equal to $D\chi.$  Therefore, the cocylce $(\tilde{\Psi}_1,\Psi_2)$ is defined up to a coboundary of the monad.  Thus we have a map from $E$ to the middle cohomology of the monad.  It is injective due to the above argument.  It is surjective thanks to the exactness of the sequence $0\rightarrow E\rightarrow P\rightarrow Q\rightarrow 0.$

Diagram \eqref{primitive-monad} is an unfolding of the monad \eqref{FirstBowMonad}, both consisting of infinite-dimensional spaces.  It allows us to focus individually on each subinterval and reinterpret its cohomology in terms of a finite-dimensional monad.  

One way of viewing it is via the spectral sequence.
\subsubsection{The Spectral Sequence}
Let us view (the upside-down of) diagram \eqref{primitive-monad}  as a part of anticommuting\footnote{We adjust the sign of one of the arrows to change commutativity of \eqref{primitive-monad} to anticommutativity.} double complex $E_0^{\bullet,\bullet}:$
$$
\xymatrix{
\vdots&\vdots&\ddots&\\
0\ar[u]\ar[r]&0\ar[u]\ar[r]&0\ar[u]\ar[r]&\hdots\\
L^2\ar[u]\ar[r]^Z&H^{-1}\ar[r]\ar[u]&0\ar[u]\ar[r]&\hdots\\
H^1\ar[u]^{D_0}\ar[r]^{-\tilde{Z}}&L^2_W\ar[u]^{\tilde{D}}\ar[r]&0\ar[u]\ar[r]&\hdots,
}
$$
with $\tilde{D}=(D_1,I,B,b), \tilde{Z}=
\begin{pmatrix}
Z\\J\\B,b
\end{pmatrix},$ and $L^2_W=L^2\oplus W\oplus N_h\otimes e_t^*\oplus N_t\otimes e_h^*.$  As  argued above, the hypercohomology $H^\bullet(T^\bullet(E))$ is $H^1(T^\bullet(E))=\mathrm{Ker}\, \mathbf{D}^\dagger=:E $ and $H^0(T^\bullet(E))=0=H^2(T^\bullet(E)),$ since, as we established by positivity, $H^0(T^\bullet(E))\oplus H^2(T^\bullet(E))=\mathrm{Ker}\,  \mathbf{D}=0.$

The horizontal leaves of the spectral sequence are 
\begin{align}
{
_>E_1^{\bullet,\bullet}:\xymatrix{
 & & \\
 \mathrm{Ker}_{L^2}Z\ar[u]&Q:=\mathrm{Cok}_{H^{-1}}Z\ar[u]\\
 \mathrm{Ker}_{H^1}\tilde{Z}\ar[u]^\alpha& P:=\mathrm{Cok}_{L^2_W}\tilde{Z}\ar[u]^\beta&,
 }
 }\\
\nonumber\\
{
_>E_\infty^{\bullet,\bullet}=_>E_2^{\bullet,\bullet}:\xymatrix{
& & & \\
& \mathrm{Cok}\, \alpha&\mathrm{Cok}\, \beta\\
& \mathrm{Ker}\, \alpha\ar[uul]& \mathrm{Ker} \beta\ar[uul]\\
 &&\ar[uul]&.
 }
 }
\end{align}
Since $\mathrm{Ker}_{L^2} Z=0$ due to the non-degeneracy condition (no continuous eigensections of both $(T_1+\i T_2)_M$ and $(T_1+\i T_2)_{L,R}$).  Also, since hypercohomology is concentrated in degree one, $\mathrm{Ker}\, \alpha=0=\mathrm{Cok}\,\beta.$

We conclude that $E =\mathrm{Ker}\,\mathbf{D}^\dagger=H^1(T^\bullet(E))=\mathrm{Ker} \beta:P\rightarrow Q,$ as argued earlier via  diagram chasing.

The vertical leaves of the spectral sequence are
\begin{align}
_\wedge E_\infty^{\bullet,\bullet}=_\wedge E_1^{\bullet,\bullet}:&\xymatrix{
0=\mathrm{Cok}_{L^2} D_0\ar[r]&\mathrm{Cok}_{H^1}\tilde{D}\\
0=\mathrm{Ker}_{H^1} D_0\ar[r]&\mathrm{Ker}_{L^2_W}\tilde{D}.
}&
\end{align}
Generic holonomy around the bow (or a circle) implies that $D_0f=F$ equation can be solved for any $F.$  It also implies that $D_0f=0$ has no global continuous solutions.  Therefore, we have another identification of $E =\mathrm{Ker}_{L^2_W}\tilde{D}=\mathrm{Ker} (D_1,I,B,b): L^2\oplus W\oplus N_h\otimes e_t^*\oplus N_t\otimes e_h^*\rightarrow H^{-1}$, which in turn can be identified with the linear space of dimension equal to the number of $\lambda$-points. 

It also follows from the hypercohomology vanishing that $\mathrm{Cok} \tilde{D}=0.$

  \subsection{The case $m=0$}

As in Section~\ref{mpos}, one has  a chain of equivalences:

\begin{theorem} One has equivalent sets of  data:
\begin{enumerate}
\item Holomorphic bundles $E$ on $X$;
\item Sheaves $P^{i,j}_{k,l}, Q^{i,j}_{k,l}$ on $\bP^1$;
\item  A tuple of matrices;
\item  A monad $ V_1 {\buildrel{\alpha }\over{\rightarrow} }V_2 {\buildrel{\beta}\over{\rightarrow}}  V_3$ of standard vector bundles on $X$, whose cohomology $\ker( \beta)/\mathrm{Im}(\alpha)$ is the bundle $E.$
\item A bow complex for $m=0$ (see section 2) .
\end{enumerate}

\end{theorem}

  Let us  be a bit more specific about our data:  it is  fortunately  
very similar to the  $m>0$ case. 
\begin{itemize}
\item The bundle is exactly as stated for the $m>0$ case, setting $m=0$.
\item The   sheaves $P^{i,j}_{k,l}, Q^{i,j}_{k,l}$ on $\bP^1$ are exactly as stated for the $m>0$ case.
\item There is a resolution diagram for the $P,Q$ exactly as in (\ref{resolution-PQ-TN-9}).  The  matrices are those giving the maps in the resolutions,  as for $m>0$.  Their normalisations (choosing bases) will differ; again, there are constraints on the matrices imposed by (anti-)commutation of the diagram, and genericity conditions.
\item The monad, is obtained from the resolution (\ref{resolution-PQ-TN-9}) exactly as above.
\item The bow complex is obtained from the  sheaves $P,Q$ as for $m>0$. The difference for $m=0$ is at the jump points $\lambda_\pm$, where there are the rank one jumps. This already occurs with the caloron, and the mechanism which accomplishes it is the same. See section 3.
\end{itemize}

 \subsection{The case $m<0$}

Again, one has  a chain of equivalences:

\begin{theorem} One has equivalent sets of  data:
\begin{enumerate}
\item Holomorphic bundles $E$ on $X$;
\item Sheaves $P^{i,j}_{k,l}, Q^{i,j}_{k,l}$ on $\bP^1$;
\item  A tuple of matrices;
\item  A monad $ V_1 {\buildrel{\alpha }\over{\rightarrow} }V_2 {\buildrel{\beta}\over{\rightarrow}}  V_3$ of standard vector bundles on $X$, whose cohomology $\ker( \beta)/\mathrm{Im}(\alpha)$ is the bundle $E.$
\item A bow complex for $m<0$ (see section 2).
\end{enumerate}

\end{theorem}

  Again, let us be more specific. Much of the data  is the same as for the $m>0$ case. 
\begin{itemize}
\item The bundle is exactly as stated for the $m>0$ case, setting $m<0$. We note one difference with the $m>0$ case, in that 
the line $C_0$ is automatically a jumping line for the holomorphic structure, that is a line where the holomorphic structure is non-trivial. These lines can be counted, with multiplicity, and the multiplicity here, of the line $C_0=\{ \psi = \infty\}$ is bounded below by $-m$. One has a ruling by lines $\psi = $ constant, 
  and  the number of jumping lines in the  ruling counted with multiplicity equals the second Chern class. As $C_0$ is already  contributing at least $-m$ to the count of jumping lines, one has  $c_2> -m$.
\item The   sheaves $P^{i,j}_{k,l}, Q^{i,j}_{k,l}$ on $\bP^1$ are exactly as stated for the $m>0$ case.
\item There is a resolution diagram for the $P,Q$ exactly as in (\ref{resolution-PQ-TN-9}).  The  matrices are those giving the maps in the resolutions,  as for $m>0$.  Their normalisations (choosing bases) will differ; again, there are constraints on the matrices imposed by (anti-)commutation of the diagram, and genericity conditions.
\item The monad, is obtained from the resolution (\ref{resolution-PQ-TN-9}) exactly as above.
\item The bow complex is obtained from the  sheaves $P,Q$ as for $m>0$. Note that the rank of the bundles on the intervals is reversed, but the procedure for building them is the same.  
\end{itemize}

  \section{On the Twistor Space}
What we have said so far concerns what is happening on a single fibre of the twistor fibration $Z_0\rightarrow \bP^1$, i.e. in one complex structure on the Taub-NUT.  We would like to see how our various correspondences generalize, when considered over the full twistor space.
  
 \subsection{Extending the bundle to a partial compactification, and infinite flags over $\bO(2)$}
Recall from Section~\ref{Sec:ITN}   our definitions of the twistor space.  We had a diagram, with horizontal maps the natural inclusion, 
   $$ \xymatrix{  X_0\ar[r]\ar[d]&Z_0\ar[d]\\ \bC\ar[r] \ar[d]& \bO(2)= T\bP^1\ar[d]\\ pt\ar[r] & \bP^1,}$$  
 whose fibre from the top row to the middle one is generically a $\bC^*$, with exceptional fibres over the zero-section in $\bO(2)$ being a chain of two copies of $\bC$;  these fibres   in $\bR^4\rightarrow \bR^3$ are  the preimages of a family of parallel lines in $\bR^3$.  There is a fibrewise compactification 
$$ \xymatrix{  X_0'\ar[r]\ar[d]&Z_0'\ar[d]\\ \bC \ar[r] \ar[d]& \bO(2)\ar[d]\\ pt\ar[r] & \bP^1,}$$ 
whose fibre from the top row to the middle one is generically a $\bP^1$, with exceptional fibres a chain of 2 copies of $\bP^1$.
 
   In going to the (partial) compactification $Z_0'$ from $Z_0$, one is adding two copies of $\bO(2)$: one, $\Gamma_0$  over  $\psi^{-1} = 0$ (so that $\Gamma_0\cap X_0'$ is the restriction to $X_0'$ of the divisor $C_0$ in $X$  of the previous section  ), and another, $\Gamma_\infty$, at $\xi^{-1}  = 0$, with again $\Gamma_\infty\cap X_0' = C_\infty\cap X_0'$ . The compactification has a real structure $\sigma$ extending the one on $Z_0$, and lifting the standard ones on $\bO(2)$ and $\bP^1$:
   $$\sigma(\xi,\psi,\eta,\zeta) = (-\overline \psi/\overline \zeta,\overline \xi/\overline \zeta,-\overline \eta/\overline\zeta^2,-1/\overline\zeta),$$
 and interchanging $\Gamma_0, \Gamma_\infty$, as well as $\Delta_\psi, \Delta_\xi$.
 
We gave ourselves a full compactification $X$ in the previous section, but this approach does not work well on the full twistor space, as our objects are resolutely non-algebraic on the full twistor space. Still, the extensions  to $X_0'$ go over well here: recall that our  instanton gave us a rank 2 holomorphic vector bundle $E$ on the twistor space $X_0'$, with distinguished sub-line bundles along $C_0\cap X_0' = \bC$, and also along $C_\infty\cap X_0' = \bC$. We now have a similar picture for $Z_0$: a rank 2 holomorphic bundle $\pE$, (in general, we will try to use the same letters, but in script, to denote extensions of our sheaves from $X_0, X_0'$ to the full twistor space $Z_0, Z_0'$)  with distinguished sub-line bundles 
$\pL_0$, $\pL_\infty$ along $\Gamma_0$, $\Gamma_\infty$; one has, as in Hitchin \cite{Hitchin:1983ay}, Garland-Murray \cite {Garland-Murray}, Charbonneau-Hurtubise \cite{Charbonneau:2007zd}, that 
$$\pL_0 = L^{\lambda_+}(-m),\quad  \pL_\infty = L^{ \lambda_-}(-m)$$
where $L$ is the standard line bundle over $\bO(2)$ of section~\ref{Sec:TwistorSpace}, and the twist $(-m)$ refers as usual to twisting by the standard line bundles $\mathcal{O}(-m)$ lifted from $\bP^1$ so that along $\Gamma_0$, $\pE$ is an extension
$$0\rightarrow L^{\lambda_+}(-m)\rightarrow \pE|_{\Gamma_0} \rightarrow L^{ \lambda_-}(m)\rightarrow 0,$$
and along $\Gamma_\infty$,
$$0\rightarrow L^{\lambda_-}(-m)\rightarrow \pE|_{\Gamma_\infty} \rightarrow L^{ \lambda_+}(m)\rightarrow 0.$$
Furthermore, the bundle $\pE$ comes equipped with a quaternionic structure, lifting the real structures $\sigma$ on $Z_0$ and $\bO(2)$: an antiholomorphic bundle map $\tau$, whose square is minus the identity: 
$$ \xymatrix{  \pE\ar[r]^\tau\ar[d]&\pE\ar[d]\\ Z_0 \ar[r]^\sigma&Z_0.}$$
Recall that the (partially compactified) twistor space $Z_0'$ is a bundle of quadrics in the bundle $\bP(L^\ell\oplus L^{-\ell}\oplus \pO)$ over $\bO(2)$; the tautological sections $\xi, \psi$ give identifications
\begin{align} 
L^\ell& \simeq \pO(\Gamma_0 + \Delta_\xi - \Gamma_\infty),&  
L^{-\ell}& \simeq \pO(\Gamma_\infty + \Delta_\psi - \Gamma_0).
\end{align}
Here $\Delta_\xi,\Delta_\psi$ are the two components of $\eta = 0.$

Our subline bundles $\pL_0$, $\pL_\infty$ can be thought of as defining flags 
$0=\pE^0_0\subset \pE^0_1= \pL_0 \subset \pE^0_2 =\pE$ over $\Gamma_0$ and $0=\pE_\infty^0\subset \pE_\infty^1= \pL_\infty \subset \pE_\infty^2 =\pE$ over  $\Gamma_\infty$; here the index $i$ in $\pE^*_i$ denotes the rank.   Define sheaves of meromorphic sections of $\pE$, which are holomorphic sections away from $\Gamma_0$, $\Gamma_\infty$:
\begin{equation} 
\pE_{p,q}^{m,n}=\left\{s\Big|
\begin{array}{c}\xi^{ p}s\ {\rm finite\ at}\  \Gamma_0\ {\rm with\ values\ in}\ \pE^0_q,\\ \xi^{-m} s\ {\rm finite\ at}\ \Gamma_\infty\ {\rm with\ values\ in}\ \pE_\infty^n
\end{array}
\right\}.
\end{equation}

We can then consider, as for the caloron, the infinite rank direct image $\pF$ of $\pE$ under projection  from the generically $\bC^*$-bundle $Z_0 $ (not  $Z_0'$) to $\bO(2)$.  $\pF$ is thus of infinite rank. One   can use the
flags $\pE^0_q $ along $\Gamma_0$, $\pE_\infty^n$ along $\Gamma_\infty$
to define for $p\in \bZ$ and $q= 0,1$ subbundles $\pF^0_{p,q},
\pF_\infty^{m,n}$ of $\pF$ as
\begin{align*}
\pF^0_{ p,q} &=\{s\in \pF\mid \xi^{p}s \text{ finite at }{\Gamma}_0
                         \text{ with value in } \pE^0_q  \}, \\
\pF_\infty^{m,n} &=\{s\in \pF\mid \xi^{-m}s \text{ finite at } {\Gamma}_\infty
                           \text{ with value in } \pE_\infty^n \}.
\end{align*}
We now have, as before, infinite flags
\begin{equation}\label{infiniteflags-1}\begin{gathered}
\cdots\subset \pF^0_{-1,0}\subset \pF^0_{-1,1} \subset \pF^0_{0,0} \subset \pF^0_{0,1} \subset \pF^0_{1,0} \subset \pF^0_{1,1} \subset\cdots\phantom{-.}\\
\cdots\supset \pF_\infty^{2,0}\supset \pF_\infty^{1,1} \supset \pF_\infty^{1,0} \supset \pF_\infty^{0,1} \supset \pF_\infty^{0,0} \supset \pF_\infty^{-1,1} \supset\cdots.
\end{gathered}\end{equation}
For  quotients, one finds line bundles
\begin{align} \pF^0_{ p,1}/ \pF^0_{ p,0} =  L^{p\ell +\lambda_+}(-m),&\quad \pF^0_{ p+1,0}/ \pF^0_{ p,1} =  L^{p\ell +\lambda_-}(m),\\
 \pF_\infty^{p,1}/\pF_\infty^{p,0}= L^{-p\ell+\lambda_-}(-m),&\quad \pF_\infty^{p+1,0}/\pF_\infty^{p,1}= L^{-p\ell+\lambda_+}(m).\end{align}

As for the caloron, the direct images $R^1\pi_*(\pE_{p,q}^{m,n})$ can be computed as the quotients $\pF/(\pF^0_{p,q}+\pF_\infty^{m,n})$. The direct images  
\begin{align}R^1\pi_*(\pE_{p,0}^{-p+1,0})&=\pF/(\pF^0_{p,0}+\pF_\infty^{-p+1,0}){\buildrel{def}\over {=}}\pQ_{p,0}^{-p+1,0} , \\R^1\pi_*(\pE_{p,1}^{-p,1}) &= \pF/(\pF^0_{p,1}+\pF_\infty^{-p,1})){\buildrel{def}\over {=}}\pQ_{p,1}^{-p,1}
\end{align}
are supported respectively over two {\it spectral curves}  $S_0$, $S_1$ in $\bO(2)$, and are, generically, line bundles over these curves. The quotients  
\begin{align}R^1\pi_*(\pE_{p,0}^{-p ,1})&=\pF/(\pF^0_{p,0}+\pF_\infty^{-p,1}){\buildrel{def}\over {=}}\pP_{p,0}^{-p,1} , \\R^1\pi_*(\pE_{p,1}^{-p,0}) &= \pF/(\pF^0_{p,1}+\pF_\infty^{-p,0})){\buildrel{def}\over {=}}\pP_{p,1}^{-p,0}
\end{align}
on the other hand, are supported over the full space and have rank one. A calculation with the Grothendieck-Riemann-Roch theorem tells us that $S_0$, the curve of jumping lines in the fibration $Z_0'\rightarrow \bO(2)$, lies in the linear system of $\pO(2k)$ over $\bO(2)$, while $S_1$ lies in the linear system $\pO(2k+ 2m)$ (see Charbonneau-Hurtubise \cite{Charbonneau:2006gu}). When $\pE$ is  twisted by $L^s$, for $s$ varying in an interval, these sheaves $\pQ_{p,0}^{ -p+1,0},\pQ_{p,0}^{ -p,0}$ are the sources of the flows for Nahm's equations, by the well known correspondence of solutions to Lax pair type equations with flows of line bundles on a curve.

More generally, as for the caloron, one can reconstruct $\pF$ by 
\begin{equation}\label{description-of-F-2}
 \begin{matrix}
0&\rightarrow&\pF&\rightarrow&
\begin{matrix}\vdots\\ \pP_{p-1,1}^{-p+1,0}   \\ \oplus\\
\pP_{p,0}^{-p,1} \\ \oplus\\ \pP_{p,1}^{-p,0}  \\ \vdots
\end{matrix}
&\rightarrow &
\begin{matrix}
\vdots\\ \oplus\\ \pQ_{p,0}^{-p+1,0} \\
\oplus\\ \pQ_{p,1}^{-p,1} \\ \oplus\\ \vdots
\end{matrix}
&\rightarrow &0.
\end{matrix}
\end{equation}
  
 \subsection{Generic structure of the bundles}
  
One can understand some of the geometry of the flags by choosing a local trivialization of $\pF$ in which the flag $\pF^0_{p,q}$ is the standard  flag for the loop group $LGl(2)$; the other flag $\pF_\infty^{m,n}$ then  defines a map into the flag manifold, and the spectral curves are pullbacks of the two  codimension one varieties of the Birkhoff stratification, and the flag $\pF_\infty^{m,n}$ is the pull-back of the tautological flag. The two codimension one Birkhoff strata intersect in two codimension two strata, and accordingly, if one imposes the genericity condition:
\begin{equation} \label{genericity} 
    {\text{The curves}}\, S_0\, {\text{and}}\, S_1\,  {\text{only meet these codimension  one and two strata;}} 
    \end{equation} 
(in other words the image in the flag manifold of $\bO(2)$ is in general position).  then one can write (see \cite{murray-monopoles, Garland-Murray}) 
$$ S_0\cap S_1 = S_{0,1}\cup S_{1,0},$$
where $S_{0,1}$, $S_{1,0}$ are pullbacks of the codimension two strata, with $S_{0,1}$  characterized by 
$\mathrm{dim}\, \pF/(\pF^0_{p,1} + \pF_\infty^{-p+1,0}) =1$,
and $S_{1,0}$ by 
$\mathrm{dim}\, \pF/(\pF^0_{p,0} + \pF_\infty^{-p,1}) = 2.$ 
The two sets are interchanged by the real structure. From the maps 
\begin{align*}   \pF^0_{ p,1}/ \pF^0_{ p,0} =  L^{p\ell + \lambda_-}(-m)& \rightarrow  \pF/(\pF^0_{p,0}+\pF_\infty^{-p ,1})=\pP_{p,0}^{-p ,1},\\
 \pF_\infty^{-p+1,0}/\pF_\infty^{-p,1}= L^{p\ell + \lambda_-}(m) & \rightarrow  \pF/(\pF^0_{p,0}+\pF_\infty^{-p ,1})= \pP_{p,0}^{-p ,1},\\
  \pF^0_{ p+1,0}/ \pF^0_{ p,1} =  L^{p\ell +\lambda_+}(m)&\rightarrow \pF/(\pF^0_{p ,1}+\pF_\infty^{-p ,0})=\pP_{p ,1}^{-p ,0},\\ 
 \pF_\infty^{-p,1}/\pF_\infty^{-p,0}= L^{ p\ell+\lambda_+}(-m)&\rightarrow \pF/(\pF^0_{p ,1}+\pF_\infty^{-p ,0})=\pP_{p ,1}^{-p ,0},
  \end{align*}
one can obtain, as   Garland and Murray do for the caloron case \cite[Section 6]{Garland-Murray}, that, with our genericity assumption,
 \begin{align}\pP_{p,0}^{-p ,1} &= L^{p\ell + \lambda_-}(2k+m)\otimes \pI_{S_{1,0}},\\
\pP_{p ,1}^{-p ,0}&=  L^{p\ell +\lambda_+}(2k+m)\otimes \pI_{S_{0,1}},
 \end{align}
where $ \pI_{S_{0,1}},  \pI_{S_{1,0}}$ are the ideal sheaves of $S_{0,1}, S_{1,0}$. In a similar vein, looking at the maps
\begin{align*} 
\pF^0_{ p,1}/ \pF^0_{ p,0} =  L^{p\ell + \lambda_-}(-m)& \rightarrow  \pF/(\pF^0_{p,0}+\pF_\infty^{-p+1 ,0})=\pQ_{p,0}^{-p+1 ,0}, \\
 \pF_\infty^{-p+1,1}/\pF_\infty^{-p+1,0}= L^{(p-1)\ell + \lambda_+}(-m) & \rightarrow  \pF/(\pF^0_{p,0}+\pF_\infty^{-p+1 ,0})=\pQ_{p,0}^{-p+1 ,0}, \\
  \pF^0_{ p+1,0}/ \pF^0_{ p,1} =  L^{ p\ell +\lambda_+}(m)&\rightarrow \pF/(\pF^0_{p ,1}+\pF_\infty^{-p ,1})=\pQ_{p ,1}^{-p ,1}, \\
 \pF_\infty^{-p+1,0}/\pF_\infty^{-p,1}= L^{ p\ell+\lambda_-}( m)&\rightarrow\pF/(\pF^0_{p ,1}+\pF_\infty^{-p ,1}) =\pQ_{p ,1}^{-p ,1},  
\end{align*}
one obtains, again as in \cite{Garland-Murray}, the sheaves concentrated over the spectral curves
 \begin{align}\label{Qs} \pQ_{p,0}^{-p+1 ,0}  &= L^{p\ell + \lambda_-}(2k+m)|_{S_0}[-S_{1,0}] =  L^{ (p-1)\ell +\lambda_+}(2k+m)|_{S_0}[-S_{0,1}],\\
\pQ_{p ,1}^{-p ,1}&=   L^{ p\ell + \lambda_+}(2k+m)|_{S_1}[-S_{0,1}] 
 =  L^{p\ell +\lambda_-}(2k+m)|_{S_1}[-S_{1,0}].
 \end{align}
 In particular, this gives isomorphisms over the spectral curves: 
 
 \begin{proposition}
 One has over $S_0$
 $$L^{ \lambda_+}  [-S_{1,0}] =  L^{ -\ell +\lambda_-} [-S_{0,1}],$$ 
 and over $S_1$
 $$L^{  \lambda_+}  [-S_{1,0}]  =  L^{ \lambda_-} [-S_{0,1}].$$
 \end{proposition}
 
 These isomorphisms are useful in understanding the Nahm flows.

\subsection{Reconstructing $\pE$}

So far a lot of this is basically reproducing results for the caloron; what is new here is a twist in the geometry induced by the fact that the fibers of  $Z_0'\rightarrow\bO(2)$ over $\eta = 0$ are chains of two $\bP^1$s, while elsewhere they are a single $\bP^1$; alternately, by the fact that the function $\eta$ factors as $\xi\psi$, and so there are two components $\Delta_\xi, \Delta_\psi$ to $\eta= 0$. The section $\xi$ of $L^\ell(1)$ gives an isomorphism
$$\xi:  \pE_{ 1,0}^{0,0}\rightarrow  \pE_{0 ,0}^{ 1,0}[-\Delta_\xi]\otimes L^{ \ell}(1).$$
Composing it with the inclusion $ \pE_{0 ,0}^{ 1,0}[-\Delta_\xi]\otimes L^{ \ell}(1)\rightarrow \pE_{0 ,0}^{ 1,0}\otimes L^{ \ell}(1)$, and 
then taking $R^1\pi_*,$ one obtains 
\begin{equation}  
\widehat \pB_{h,t}: \pQ_{1,0}^{0,0}\rightarrow  \pQ_{0 ,0}^{ 1,0}\otimes L^{ \ell}(1). \label{hatheadtail}
\end{equation}
Likewise, multiplication by $\psi$ induces 
\begin{equation}  \widehat \pB_{t,h}: \pQ^{1,0}_{0,0}\rightarrow  \pQ^{0 ,0}_{ 1,0}\otimes L^{ -\ell}(1).\label{hattailhead} \end{equation}

The failure of $\widehat \pB_{h,t}$ to be an isomorphism is governed by vanishing of the inclusion over $\Delta_\xi.$ Note that the fibres of $Z_0'\rightarrow \bO(2) $ over $\eta = 0$ intersect $\Delta_\xi$ in a line $\bP^1_\xi$, and $\Delta_\psi$ in a $\bP^1_\psi$; on $\bP^1_\xi$, the tensoring by $\pO(-\Delta_\xi)$ is tensoring by an $\pO(1)$, and on $\bP^1_\psi$,   by an $\pO(-1)$. 

The sheaves $\pQ_{ 1,0}^{0,0},\pQ_{0 ,0}^{ 1,0}$ measure when there are jumping lines (lines over which the bundle is non-trivial) in the structure of $\pE$ over $\bO(2)$. Over $\eta= 0$, $\pQ_{ 1,0}^{0,0}$ picks out the jumping lines in $\Delta_\psi$, and $\pQ_{0 ,0}^{ 1,0}$ picks out those in  $\Delta_\xi$. The two cases are interchanged by the real structure. Thus, the intersection of the spectral curve with $\eta = 0$ 
gets partitioned into two divisors of degree $k$ on the curve: one, $S_{0,\xi}$, where the jumping line lies in $\Delta_\xi$, and the other, $S_{0,\psi}$, where the jumping line lies in $\Delta_\psi$. The map $\widehat \pB_{h,t}$ is not  an isomorphism over $S_{0,\psi}$, and similarly, the map $\widehat \pB_{t,h}$ is not  an isomorphism over $S_{0,\xi}$; in fact, instead there are isomorphisms
\begin{equation}\label{shiftup}  
\widehat \pB_{h,t}: \pQ_{1,0}^{0,0}\rightarrow  \pQ_{0 ,0}^{ 1,0}\otimes L^{ \ell}(1)[-S_{0,\psi}],  
\end{equation}
 
\begin{equation} \label{shiftdown} \widehat \pB_{t,h}: \pQ^{1,0}_{0,0}\rightarrow  \pQ^{0 ,0}_{ 1,0}\otimes L^{ -\ell}(1)[-S_{0,\xi}]. \end{equation}We note that the compositions
$\widehat \pB_{h,t}\circ \widehat \pB_{t,h}$, $\widehat \pB_{t,h}\circ \widehat \pB_{h,t}$ are sections of $\pO(2)[-S_{0,\xi}-S_{0,\psi}]$ corresponding to multiplication by $\xi\psi = \eta$, vanish over $\eta= 0$ on the curve, and have as well a double pole over $\zeta=\infty$, that is
$$S_0 \cap \{\eta= 0\} =  S_{0,\psi}\cup S_{0,\xi},$$
 as it should, with $ S_{0,\psi}, S_{0,\xi}$ interchanged by the real structure.

As noted for the restriction of the bundle to $X_0'$, the fact that the shifts between $\pF^0_{p ,i}, \pF^0_{p+1,i}$ and between $\pF_\infty^{-p,1}, \pF_\infty^{-p-1,1}$ induced by multiplication by $\xi$, $\psi$ are no longer isomorphisms makes the reconstruction process for the bundle $\pE$ is a bit trickier. One has,   essentially, a parametrized version of what was done over $X_0$:
  \begin{equation} \label{cal-E-sequence}
  0\rightarrow \pE \rightarrow \pP^{0,0}_{0,1}\oplus \pQ \oplus \pP^{0,1}_{0,0} \xrightarrow{\begin{pmatrix}1& 0 &-1\\  
  -1&(\xi,\widehat \pB_{t,h})& 0\\ 0& (\widehat \pB_{h,t},\psi)& 1\end{pmatrix}} \pQ_{0,1}^{0,1}\oplus \pQ_{ 1,0}^{0,0} \oplus \pQ_{0,0}^{1,0},\end{equation}
where we kept the notation  $\pP^{i,j}_{m,n},\pQ^{i,j}_{m,n}$ to denote the lifts of $\pP^{i,j}_{m,n},\pQ^{i,j}_{m,n}$ to $Z_0'$ from $\bO(2)$, and   $\pQ$ is defined as the quotient, over $Z_0$
  
 \begin{equation} \label{define-pQ2}\xymatrix{  
  \pQ_{ 1,0}^{0,0}( -\Gamma_0)(-2)   \ar[ddr]^(0.7){ -\widehat \pB_{h,t}}\ar[r]^(0.43){\psi\bI} &  \pQ_{ 1,0}^{0,0}(-\Gamma_\infty )  \otimes L^{-\ell} (-1) \ar[rd]  \\
 &&\pQ. \\
\pQ_{0 ,0}^{ 1,0} (-\Gamma_\infty) (-2) \ar[uur]_(0.7){-\widehat \pB_{t,h}}\ar[r]^(0.43){\xi\bI} &  \pQ_{0 ,0}^{ 1,0} (-\Gamma_0 )\otimes L^{\ell}(-1)   \ar[ur] \\ }
\end{equation}
 
This way of reconstructing $\pE$ has advantages for describing solutions to the Dirac equation and the Laplace equation for the underlying instanton $A$, shifted by the $U(1)$ instanton $-sa$ corresponding to the line bundle $L^s.$ On the twistor space, this shift amounts to considering the bundle $\pE\otimes L^{-s}$. Let $R$ denote one of the sheaves $\pQ_{0,1}^{0,1},  \pQ_{ 1,0}^{0,0}, \pQ_{0,0}^{1,0}$ on the right hand side of \eqref{cal-E-sequence}. An element of $H^0(Z_0, R(-j)\otimes L^{-s})$ maps, by the coboundary map, to an element of $H^1(Z_0, \pE(-j)\otimes L^{-s})$; by the twistor transform,
\begin{itemize}
\item
for $j= 1$, these correspond to solutions of the Dirac equation for the connection $A-sa\bI$, and 
\item
for $j = 2,$ these correspond 
to solutions of the Laplace equation in the same background.
\end{itemize}
 
With this, following the proof in Hurtubise and Murray \cite[Prop 1.17]{HurtubiseMurray}, (which in turn follows Hitchin \cite{Hitchin:1983ay}) we have
\begin{theorem} \label{vanishing} For solutions to the Laplace equation
\begin{itemize}
\item For $s\in (\lambda_-, \lambda_+)$ if $m>0$, and $s\in [\lambda_-, \lambda_+]$ if $m=0$, elements of $H^0(S_1, \pQ_{0,1}^{0,1}\otimes L^{-s}(-2))$ correspond to  solutions to the Laplace equation on $X$ decaying at infinity, in the  background of the connection $A-sa\bI$;
\item For $s\in [\lambda_+, \lambda_- +\ell]$,  elements of $H^0(S_0, \pQ_{1,0}^{0,0}\otimes L^{-s}(-2))$ correspond to solutions to the Laplace equation on $X$ decaying at infinity, in the same background.
\item For $s\in (\lambda_+-\ell, \lambda_-  )$, elements of $H^0(S_0, \pQ_{0,0}^{1,0}\otimes L^{-s}(-2))$ correspond to  solutions to the Laplace equation on $X$ decaying at infinity, in the same background.
\end{itemize}

In consequence, since decaying solutions to the Laplace equation must vanish, 
\begin{align}   H^0(S_1, \pQ_{0,1}^{0,1}\otimes L^{-s}(-2))&  =0\ {\rm for}\ s\in  (\lambda_-, \lambda_+),\  {\rm and}\  s = \lambda_-, \lambda_+ {\rm if}\  m=0,  \\
H^0(S_0, \pQ_{1,0}^{0,0}\otimes L^{-s}(-2))& = 0 \ {\rm for}\ s\in [\lambda_+, \lambda_- +\ell],\\
 H^0(S_0, \pQ_{0,0}^{1,0}\otimes L^{-s}(-2))& = 0 \ {\rm for}\ s\in [\lambda_+-\ell, \lambda_-  ].
\end{align}

In turn, for the Dirac equation,
\begin{itemize}
\item For $s\in (\lambda_+, \lambda_- +\ell)$ if $m>0$, and $s\in [\lambda_+, \lambda_- +\ell]$ if $m=0$,
the elements of $H^0(S_0, \pQ_{1,0}^{0,0}\otimes L^{-s}(-1))$ correspond to $L^2$ solutions to the Dirac equation on $X$, in the background of the connection $A-sa\bI$;

\item For $s\in [\lambda_+, \lambda_- +\ell]$, 
  elements of $H^0(S_0, \pQ_{1,0}^{0,0}\otimes L^{-s}(-1))$ correspond to $L^2$ solutions to the Dirac equation on $X$, in the same background;
  
  \item For $s\in [\lambda_+-\ell, \lambda_-  ]$, 
 the elements of $H^0(S_0, \pQ_{0,0}^{1,0}\otimes L^{-s}(-1))$ correspond to $L^2$ solutions to the Dirac equation on $X$, in the same background.
 \end{itemize}

 From the sequence $\pO(-2)\rightarrow \pO(-1)\rightarrow \pO(-1)|_f$, where $f$ is a fiber of $\bO(2)$ over $\bP^1$, tensored by the bundles above, the vanishing theorem for the Laplace equation implies in turn that 
 \begin{align}   H^0(S_1, \pQ_{0,1}^{0,1}\otimes L^{-s}(-1)) & =\bC^{k+m}\ {\rm for}\ s\in  (\lambda_-, \lambda_+),\ {\rm and}\ s= \lambda_-, \lambda_+, {\rm if} m=0,\\
H^0(S_0, \pQ_{1,0}^{0,0}\otimes L^{-s}(-1)) &= \bC^k  \ {\rm for}\ s\in (\lambda_+, \lambda_- +\ell),\\
 H^0(S_0, \pQ_{0,0}^{1,0}\otimes L^{-s}(-1))& = \bC^k \ {\rm for}\ s\in (\lambda_+-\ell, \lambda_-  ).
\end{align}
\end{theorem} 
 This concords with the index calculation for Dirac operators of \cite{Cherkis:2016gmo}. For $s$ in the interior of the intervals,   the solutions have exponential decay, with the exponent of the bound given by minus the distance between $s$ and the closest ends of the interval  to which $s$ belongs given in the theorem.

\subsection{Constructing the Nahm flows} 

The bow solutions corresponding to the instanton can be constructed following more or less exactly the scheme for monopoles given in \cite[Section 2]{HurtubiseMurray}. In brief, with our genericity assumptions:
\begin{itemize}
\item One defines a rank $k$ bundle $N_{0,-}$ over  $[\lambda_+-\ell, \lambda_-  ]$ by  
$$N_{0,-}(s) = H^0(S_0, \pQ_{0,0}^{1,0}\otimes L^{-s}(-1));$$ 
a rank $k$ bundle $N_{0,+}$ over $[\lambda_+, \lambda_- +\ell]$ by 
$$N_{0,+}(s) = H^0(S_0, \pQ_{1,0}^{0,0}\otimes L^{-s}(-1));$$ 
and a rank $k+m$ bundle $N_1$ over $(\lambda_-, \lambda_+)$ by 
$$N_1(s) = H^0(S_1, \pQ_{0,1}^{0,1}\otimes L^{-s}(-1)).$$
One can take limits, and extend to a bundle over $[\lambda_-, \lambda_+]$
\item At the boundary point $\lambda_-$  the sheaves $\pQ_{0,0}^{1,0}\otimes L^{-s}(-1)), \pQ_{0,1}^{0,1}\otimes L^{-s}(-1)$ are identified with the restrictions to $S_0$, $S_1$ of the sheaves $L^{\lambda_--s}\otimes M$, where $M$ is some algebraic sheaf on $\bO(2)$. At $\lambda_-=s$, their sections become the restrictions of global sections of $\pP_{0 ,1}^{0 ,0}\otimes L^{-\lambda_-} = M$, and so can `pass' from one spectral curve to the other. This provides the glueing of $N_{0,-}$ to $N_1$.
\item The same thing happens at the boundary point $\lambda_+$:  this time the sheaf $\pP^{0 ,1}_{0 ,0}\otimes L^{-\lambda_+}$
is the intermediary.
\item This gives us our generalized bundle $N.$ We can now define an {$\pO(2)$-valued} endomorphism $\underline A(\zeta,s) =\underline A_0(s) + \underline A_1(s)\zeta + \underline A_2(s)\zeta$   as the action on sections induced by multiplication by $\eta$ on the $\pQ$s. 
\item A connection of the form $\nabla_s =  d_s + \underline A_0(s) + \underline A_1(s)\zeta/2 \buildrel{def}\over{=} d_s + \underline A_+(s,\zeta)$ can be defined on the bundle $N$, and using this to trivialise the bundle $N$, and so turn the endomorphism  $\underline A(\zeta,s) $ into a matrix  $ A(\zeta,s)$,   one then has that the Nahm equations $d_sA(\zeta,s) = [A_+(s,\zeta), A(s,\zeta)]$ are satisfied.
\item Finally, using the real structures on the $Q$ induced from those on $E$, one has a real structure on $N$, with respect to which our solutions to the Nahm's equations are skew-Hermitian.
\item The flows are the same as those obtained directly from the instanton via the Down transform \cite{Third}, a generalization of the ADHM-Nahm transform.
\end{itemize}

All of this follows a well established pattern: for the flows, it appears in Hitchin \cite{Hitchin:1983ay};  for the glueing of intervals, in Hurtubise and Murray \cite{HurtubiseMurray}; and for the equivalence with the Nahm transform, in a form closest to what we have here, in  Charbonneau and Hurtubise \cite{Charbonneau:2007zd}. There remains the bifundamental part of the bow data, which closes up the flow on our intervals to flows on the bow, with bifundamental maps along the edges of the bow. Indeed, for $a\in (\lambda_+, \lambda_-+\ell)$, the sheaf maps \eqref{hatheadtail} and \eqref{hattailhead} of multiplication by $\xi, \psi$ give, respectively, maps on sections
\begin{equation} \label{headtail1} \widehat \pB_{h,t}: N_{0,+}(a) = H^0(S_0, \pQ_{ 1,0}^{0,0}\otimes L^{-a} (-1)) \rightarrow  H^0(S_0, \pQ_{0 ,0}^{ 1,0}\otimes L^{ -a+\ell}), \end{equation}
\begin{equation} \label{tailhead1} \widehat \pB_{h,t}: N_{0,-}(a-\ell)=  H^0(S_0,  \pQ^{ 1,0}_{0,0}\otimes L^{-a+\ell}(-1))\rightarrow   H^0(S_0, \pQ^{0 ,0}_{ 1,0}\otimes L^{-a}).\end{equation}
Using our vanishing theorem, one can identify 
\begin{multline*}
H^0(S_0, \pQ_{0 ,0}^{ 1,0}\otimes L^{ -a+\ell}) = H^0(S_0, \pQ_{0 ,0}^{ 1,0}\otimes L^{ -a+\ell}(-1))\otimes H^0(S_0,\pO(1))\\ 
= N_{0,-}(a-\ell)\otimes (\bC\oplus \zeta\bC)
\end{multline*}
\begin{multline*}
H^0(S_0, \pQ^{0 ,0}_{ 1,0}\otimes L^{-a}) = H^0(S_0, \pQ^{0 ,0}_{ 1,0}\otimes L^{-a}(-1)) \otimes H^0(S_0,\pO(1))\\ =  N_{0,+}(a)\otimes (\bC\oplus \zeta\bC),
\end{multline*}
and so one can write \eqref{headtail1} and \eqref{tailhead1} as
\begin{equation} \label{head-tail} \pB_{h,t}(\zeta, a) =  \pB_{h,t,0}(a)+   \pB_{h,t,1}(a)\zeta : N_{0,+}(a)   \rightarrow N_{0,-}(a-\ell) , \end{equation}
for the multiplication by $\xi.$ Likewise, multiplication by $\psi$ induces 
\begin{equation} \label{tail-head}  \pB_{t,h}(\zeta,a-\ell) =  \pB_{t,h,0}(a-\ell)+  \pB_{t,h,1}(a-\ell)\zeta:N_{0,-}(a-\ell)\rightarrow   N_{0,+}(a).\end{equation}
The fact that $\xi\psi = \eta$ translates into the relations
\begin{align}
\pB_{h,t}(\zeta, a)  \pB_{t,h}(\zeta,a-\ell) =A(\zeta,a-\ell)& : N_{0,-}(a-\ell)\rightarrow N_{0,-}(a-\ell),\\
 \pB_{t,h}(\zeta,a-\ell)\pB_{h,t}(\zeta, a) =  A(\zeta,a )&: N_{0,+}(a)\rightarrow  N_{0,+}(a).
 \end{align}
 
 Summarising:
 
 \begin{theorem} 
 Through the sequence (\ref{cal-E-sequence}), and the construction given in this section, a bundle $\pE$ on the twistor space $Z_0'$ corresponding to a Taub-NUT instanton defines a bow solution as in subsection \ref{Bow} above.
 
 This solution coincides with the one obtained through the Down transform, a generalization of the ADHM-Nahm transform, from instanton to bow solution.
 
 With our genericity assumption (\ref{genericity}) the instanton (or the bow solution to the following Nahm's equations) are equivalent to {\it spectral data} on $\bO(2)$:
 \begin{itemize}
 \item {\it Spectral curves} $S_0, S_1$ in $\bO(2)$, of degree $2k, 2k+2m$ respectively, both real.
 \item  A partition of the intersection $S_0\cap S_1$ into two divisors $S_{0,1}, S_{1,0}$ interchanged by the real structure.
 \item Over $S_0$, an isomorphism $L^{ \lambda_+}  [-S_{1,0}] =  L^{ -\ell +\lambda_-} [-S_{0,1}],$  
 and over $S_1$ $L^{  \lambda_+}  [-S_{1,0}]  =  L^{ \lambda_-} [-S_{0,1}].$
  \item  Sheaves $\pQ^{0 ,0}_{ 1,0}, \pQ_{0 ,0}^{ 1,0}$ on $S_0$, with 
 \begin{itemize}
 \item Holomorphic Euler characteristic $k$;
 \item  Isomorphisms \eqref{shiftup} and \eqref{shiftdown}   $\pQ_{1,0}^{0,0} =  \pQ_{0 ,0}^{ 1,0}\otimes L^{ \ell}(1)[-S_{0,\psi}]$ and  $\pQ^{1,0}_{0,0}= \pQ^{0 ,0}_{ 1,0}\otimes L^{ -\ell}(1)[-S_{0,\xi}] $ over $S_0$;
 \item A vanishing theorem \ref{vanishing};
  \item A quaternionic structure $\pQ^{0 ,0}_{ 1,0}\rightarrow \pQ_{0 ,0}^{ 1,0}$, $\pQ_{0 ,0}^{ 1,0}\rightarrow \pQ^{0 ,0}_{ 1,0}$, lifting the real structure on $\bO(2)$.

 \end{itemize}
  \item A sheaf $\pQ^{0 ,1}_{ 0,1}$ on $S_1$, with 
  \begin{itemize}
 \item Holomorphic Euler characteristic $k+m$;
 \item A vanishing theorem \ref{vanishing};
 \item A quaternionic structure on $\pQ^{0 ,1}_{ 0,1}$, lifting the real structure on $\bO(2)$.
 \end{itemize}
 \end{itemize} 
 \end{theorem}

 \subsection{Twistors and monads} 
 
 The bundles $\bE$ on twistor space corresponding to an instanton can be obtained by a version of the monad construction; however, the non-algebraic nature of the line bundles $L^\mu$ entering into the construction of $\pE$ makes a global construction difficult, but one can build two monads, one over $\zeta\neq \infty$, and one over $\zeta\neq 0$ which are quasi-isomorphic on the overlap, in that there is a monad morphism between the two inducing an isomorphism on the cohomology. 
 
 We had for our instanton bundle $\pE$, a sequence of sheaves (\ref{cal-E-sequence}) determining $\pE$. With our genericity assumptions we have seen above, the constituent sheaves are:
  \begin{align}\pP_{0,0}^{0 ,1} &= L^{ \lambda_-}(2k+m)\otimes \pI_{S_{1,0}},\\
\pP_{0 ,1}^{0 ,0}&=  L^{ \lambda_+}(2k+m)\otimes \pI_{S_{0,1}},\\
\pQ_{0,0}^{ 1 ,0}  &= L^{  \lambda_-}(2k+m)|_{S_0}[-S_{1,0}] =  L^{ -\ell +\lambda_+}(2k+m)|_{S_0}[-S_{0,1}],\\
\pQ_{0,1}^{0 ,1}&=   L^{ \lambda_+}(2k+m)|_{S_1}[-S_{0,1}]  =  L^{\lambda_-}(2k+m)|_{S_1}[-S_{1,0}],\\
\pQ_{1,0}^{0 ,0}  &= L^{ \ell + \lambda_-}(2k+m)|_{S_0}[-S_{1,0}] =  L^{ \lambda_+}(2k+m)|_{S_0}[-S_{0,1}],
 \end{align}
as well as the sheaf $\pQ$ defined using them via (\ref{define-pQ2}).
 We would like to get some resolutions of these sheaves, in a fairly compatible way, achieving over the twistor space what we had over $\zeta = 0$.  One issue is that the line bundle $L^\mu$ has no sections, not even meromorphic ones, over $\bO(2)$. We therefore  tensor our sheaves by a suitable power of $L$, to make the resulting bundle algebraic.
 
  We begin with 
 $$ \pQ_{0,0}^{ 1 ,0}\leftarrow \pP_{0,0}^{0 ,1}\rightarrow \pQ_{0,1}^{0 ,1};$$
 tensoring this by $L^{- \lambda_-}$, we get
 $$\pO(2k+m)|_{S_0}[-S_{1,0}] \leftarrow \pO(2k+m)\otimes \pI_{S_{1,0}}\rightarrow \pO(2k+m)|_{S_1}[-S_{1,0}].$$
This extends to  a ruled surface compactification $\overline{\pO(2)}= \bP(\pO(2)\oplus \pO)$ of $\pO(2)$:
  $$\pO(2k+m)|_{S_0}[-S_{1,0}] \leftarrow \pO(2k+m, k+m-1)\otimes \pI_{S_{1,0}}\rightarrow \pO(2k+m)|_{S_1}[-S_{1,0}],$$
 where the notation $(m,n)$ refers to tensoring by the line bundles $\pO(mC_\zeta + n C_\infty)$, with $C_\zeta$ a fiber of the projection to $\pi: \overline{\pO(2)}\rightarrow \bP^1$, and $C_\infty$ the divisor at infinity one is adding to compactify. We note that $C_\infty$ does not intersect the spectral curves. 
 On the fiber product $\overline{\pO(2)}\times_{\bP^1} \overline{\pO(2)}$, one has a resolution of the diagonal
 $$0\rightarrow   \pO(-2,-1)\boxtimes \pO(-2-1) \xrightarrow{\eta-\eta'} \pO \rightarrow \pO|_\Delta\rightarrow 0.$$
 Tensoring this with the lift of  our sheaves $W$ from one factor, then pushing down to the other will give a resolution of the form
 $$0\rightarrow \pi^*\pi_*W(-2) \rightarrow \pi^*\pi_*W \rightarrow W \rightarrow 0.$$
 
 For the first sheaf, from the Vanishing Theorem~\ref{vanishing}, $H^0(S_0, \pO(2k+m-2) [-S_{1,0}]) = 0$, one obtains $\pi_*( \pO(2k+m-2)|_{S_0} [-S_{1,0}])= \pO(-1)^{\oplus k}, \pi_*( \pO(2k+m)|_{S_0} [-S_{1,0}])= \pO(1)^{\oplus k}$, and our resolution becomes, when restricted to the complement of $C_\infty$,
 $$\pU_{0,-}=\pO(-1)^{\oplus k}\xrightarrow{ A^0(\zeta,\lambda_-)-\eta\bI} \pV_{0,-}= \pO(1)^{\oplus k} \rightarrow \pO(2k+m)|_{S_0}[-S_{1,0}]\rightarrow  0.$$
 The matrix $ A^0(\zeta,\lambda_-)$ is up to conjugation the matrix involved in our solution to the Nahm equations.
 
 For the third, one has the result in Murray-Hurtubise, \cite[Prop.~2.21]{HurtubiseMurray}, that the space of sections of $\pO(2k+m)|_{S_1}[-S_{1,0}]$ splits as a sum of the sections on $S_0$, plus a family of sections vanishing on all of $S_0$, and therefore sections of $\pO(m)$. The pushdown of the first yields  $\pO(1)^{\oplus k}$, as before; the second one gives $\pO(m) \oplus 
 \pO(m-2)\oplus...\oplus \pO(-m+2)$. The resolution is then 
\begin{multline*}
\pU_{1}= \pO(-1)^{\oplus k} \oplus \pO(m-2) \oplus 
 \pO(m-4 )\oplus...\oplus \pO(-m )\xrightarrow{  A^1(\zeta)-\eta\bI}  \\
\pV_{1}=  \pO(1)^{\oplus k}\oplus \pO(m) \oplus 
 \pO(m-2)\oplus...\oplus \pO(-m+2) \rightarrow \pO(2k+m)|_{S_1}[-S_{1,0}]\rightarrow  0.\end{multline*}
 Here $  A^1(\zeta)-\eta\bI$ is given by:
 $$A^1(\zeta,\lambda_-)-\eta\bI = \begin{pmatrix}  A^0(\zeta,\lambda_-)-\eta\bI& C_1(\zeta)e_+ \\ -e_-^T A'(\zeta) &(\Sh-\eta\bI) + C_1(\zeta)e_+  \end{pmatrix},$$
 with $\Sh, e_-, e_+$ as before.
 
 For the middle sheaf, one can again decompose as sections on the spectral curve $S_0$, and sections vanishing on $S_0$, giving   similar pushdowns as for the third term, and the resolution
\begin{multline*}
\pU_-=\pO(-1 )^{\oplus k}   \oplus \pO(m-2 ) \oplus  
 \pO(m-4  )\oplus... \pO(-m+2  ){\xrightarrow{ \widetilde A^1(\zeta,\lambda_-)-\eta\bI}}\\
\pV_- =\pO(1)^{\oplus k} \oplus \pO(m) \oplus 
 \pO(m-2)\oplus...\oplus \pO(-m+2) \rightarrow    \pO(2k+m)\otimes \pI_{S_{1,0}} \rightarrow 0. 
 \end{multline*}
Here $\widetilde A^1(\zeta,\lambda_-)-\eta\bI$ is $  A^1(\zeta,\lambda_-)-\eta\bI $ with the last column removed.
We then have, shifting back by $L^{\lambda_-}$, the diagram:
\begin{equation} \label{resolution-twistor-PQ-TN-1}
 \xymatrix{    
 & \pU_{0,-}\otimes L^{\lambda_-}\ar[rr]^{ A^0(\zeta,\lambda_-)-\eta\bI} &&\pV_{0,-}\otimes L^{\lambda_-} \ar[r]  &\pQ_{0 ,0}^{ 1,0} \\
\pU_{-}\otimes L^{\lambda_-}\ \ \ar[dr]^{-X_{ -,1}  } \ar[rr]^{\widetilde A^1(\zeta,\lambda_-)-\eta\bI}\ar[ur]^{-X_{-,0} }&&  \pV_{-}\otimes L^{\lambda_-}\ar[dr]^{Y_{-,1}}\ar[ur]^{Y_{-,0}}\ar[r]  &\pP_{0,0}^{0,1}\ar[dr] \ar[ur]\\
 & \pU_{1}\otimes L^{\lambda_-} \ar[rr]^{A^1(\zeta,\lambda_-)-\eta\bI} &&\pV_{1}\otimes L^{\lambda_-}\ar[r]  &\pQ_{0,1}^{0,1}. 
 }
\end{equation}
The matrices $X_{-,0}, Y_{-,0}$ are projections onto the first $k$ coordinates;   $X_{-,1}$ is inclusion into the first $k+m-1$ coordinates;  $Y_{-,1}$ is the identity.
Likewise, one has 
     \begin{equation} \label{resolution-twistor-PQ-TN-2}\xymatrix{    
      & \pU_{1}\otimes L^{\lambda_+} \ar[rr]^{A^1(\zeta,\lambda_+)-\eta\bI} &&\pV_{1}\otimes L^{\lambda_+}\ar[r]  &\pQ_{0,1}^{0,1} \\
\pU_{+}\otimes L^{\lambda_+}\ \ \ar[ur]^{-X_{ +,1}  } \ar[rr]^{\widetilde A^1(\zeta,\lambda_+)-\eta\bI}\ar[dr]^{-X_{+,0} }&&  \pV_{+}\otimes L^{\lambda_+}\ar[ur]^{Y_{+,1}}\ar[dr]^{Y_{+,0}}\ar[r]  &\pP_{0,0}^{0,1}\ar[dr] \ar[ur]\\
 & \pU_{0,+}\otimes L^{\lambda_+}\ar[rr]^{ A^0(\zeta,\lambda_+)-\eta\bI} &&\pV_{0,+}\otimes L^{\lambda_+} \ar[r]  &\pQ^{0 ,0}_{ 1,0}. \\
 }
\end{equation}
The definitions of $A^0(\zeta,\lambda_+), A^1(\zeta,\lambda_+)-\eta\bI, \widetilde A^1(\zeta,\lambda_+)-\eta\bI,$ $ X_{ +,1}, Y_{ +,1}, X_{ +,0}, Y_{+,0}$ are the exact analogs of the matrices  $A^0(\zeta,\lambda_-),  A^1(\zeta,\lambda_-)-\eta\bI, \widetilde A^1(\zeta,\lambda_-)-\eta\bI, X_{ -,1}, Y_{ -,1}, X_{ -,0}, Y_{-,0}$. 

There remains the sheaf $\pQ$, defined in \eqref{define-pQ2}. Unlike the others, it is not a lift from $\bO(2)$, but must be defined on $Z_0.$  Its support, on the other hand, is on the lift of the spectral curve $S_0$, and it is defined in terms of the sheaves $\pQ_{0 ,0}^{ 1,0}, \pQ^{0 ,0}_{ 1,0}$. However, since the fibers under  $\pi: \bO(2)\rightarrow \bP^1$ of the  support of these sheaves are discrete, local sections of  $\pQ_{i,j}^{ k,l}$ can always be represented by local sections of $\pi^*\pi_*\pQ_{i,j}^{ k,l}$, giving a resolution:
 \begin{equation} \label{define-pQ2-twistor}\xymatrix{  
\pU_{0,-}\otimes L^{\lambda_+ }(-\Gamma_\infty)   \ar[ddr]_(0.7){- \pB_{t,h}}\ar[r]^(0.43){\xi\bI} &  \pV_{0,-}\otimes L^{\lambda_++\ell} (-\Gamma_0 ) (-1)   \ar[dr] \\ 
 &&\pQ. \\
 \pU_{0,+}\otimes L^{\lambda_+ +\ell}( -\Gamma_0)  \ar[uur]^(0.7){ -  \pB_{h,t}}\ar[r]^(0.43){\psi\bI} &  \pV_{0,+}\otimes L^{\lambda_+ }(-\Gamma_\infty )  (-1) \ar[ru]  \\}
\end{equation}

 \subsubsection {Monads on open sets}
 
We now must sew these pieces together to get a monad for bundles. We note that the pieces already assembled involve the bundles $L^{\lambda_+}, L^{\lambda_-}$. These bundles $L^\mu$, over $\bO(2)$, have no meromorphic sections for $\mu\neq 0$, and similarly there is no meromorphic map from $L^{\lambda_+}$ to $L^{\lambda_-}$. On the other hand, over the spectral curves, things are more flexible. In any case, let us cover $\bO(2)$ by two standard open sets $D_\alpha, D_\beta$ corresponding to $\zeta \neq \infty, \zeta\neq 0$ respectively. Let  $Z_\alpha, Z_\beta$ be their inverse images in the twistor space. One has trivialisations of $L^\mu$ over $D_\alpha, D_\beta$, related by $\exp(\mu \eta/\zeta)$ on the overlap.

We look at the maps $ \pP_{0,0}^{0,1}\rightarrow \pQ^{0,1}_{0,1}\leftarrow \pP^{0,0}_{0,1}$ where now we don't twist by a power of $L$; let us write the resolution of  $\pQ^{0,1}_{0,1}$ as
  \begin{equation}  \xymatrix{  
       \widetilde \pU_{1}  \ar[rr]^{  A^1(\zeta,0)-\eta\bI} &&\widetilde\pV_{1}  \ar[rr]  &&\pQ_{0,1}^{0,1}}.\end{equation}
Here we suppose for simplicity that $0\in (\lambda_-,\lambda_+)$, and so the vanishing theorem applies there. This yields, in the same way as for $\pQ^{1,0}_{0,0}$ above, that  $\widetilde \pU_{1}  = \pO(-1)^{k+m}, \widetilde \pV_{1}  = \pO( 1)^{k+m}$.  
 Using the trivializations of $L^{\lambda_+}, L^{\lambda_-}$, we can compactify  fiberwise over $\zeta\neq \infty, 0$, lift up and push down on the restriction of the fiber product $\overline{\pO(2)}\times_{\bP^1} \overline{\pO(2)}$ to $\zeta\neq \infty, 0$ to get a diagram over $D_\alpha $:
         
         \begin{equation}  \xymatrix{  
   \pU_{-}\otimes L^{\lambda_-}\ \ \ar[dr]^{-\widetilde X_{ -,1}  } \ar[rr]^{\widetilde A^1(\zeta,\lambda_-)-\eta\bI} &&  \pV_{-}\otimes L^{\lambda_-}\ar[dr]^{\widetilde Y_{-,1}} \ar[rr]  &&\pP_{0,0}^{0,1}\ar[dr] \\      
        & \widetilde \pU_{1}  \ar[rr]^{  A^1(\zeta,0)-\eta\bI} &&\widetilde \pV_{1} \ar[rr]  &&\pQ_{0,1}^{0,1}, \\
\pU_{+}\otimes L^{\lambda_+}\ \ \ar[ur]^{ -\widetilde X_{ +,1}  } \ar[rr]^{\widetilde A^1(\zeta,\lambda_+)-\eta\bI} && \pV_{+}\otimes L^{\lambda_+}\ar[ur]^{ \widetilde Y_{+,1}} \ar[rr]  &&\pP^{0,0}_{0,1}  \ar[ur]}\end{equation}
with a similar picture over $D_\beta.$

The resolutions of    $\pQ^{1,0}_{0,0}$ are in some sense a bit simpler, essentially because the matrices $A^0(\zeta,s)$ are regular on the whole interval $[\lambda_+-\ell, \lambda_-]$, and the relevant vanishing theorem applies for all $s$ in this interval; this means that one can resolve  $\pQ^{1,0}_{0,0}$ in terms of a sum of $L^s$ for any $s$ in the interval; furthermore, the relevant matrices $A^0(\zeta,s)$ evolve by ($\zeta$-dependent) conjugation:  $A^0(\zeta,s) =  h_{\alpha, s, s'}A^0(\zeta,s') h_{\alpha, s, s'}^{-1}$. One can then use this to write  the map from $\pQ$ to $\pQ^{0,1}_{0,1}$   as
  \begin{equation} 
  \xymatrix@C=0.9em{  
\pU_{0,-}\otimes L^{\lambda_+}(-\Gamma_\infty)    \ar[ddr]^(0.7){- \pB_{t,h}}\ar[r]^(0.44){\xi\bI} &  \pV_{0,-}\otimes L^{\lambda_++\ell} (-\Gamma_0 ) (-1)    \ar[drr] \\ 
 &&&\pQ \ar[ddr] \\
  \pU_{0,+}\otimes L^{\lambda_++\ell }( -\Gamma_0) \ar[dr]^(0.6){-h_{\alpha,-,\lambda_-,\lambda_+-\ell} }  \ar[uur]_(0.7){ -  \pB_{h,t}}\ar[r]^(0.44){\psi\bI} &  \pV_{0,+}\otimes L^{\lambda_+}(-\Gamma_\infty )  (-1)\ar[dr]^(0.6){\ \ \xi h_{\alpha,-,\lambda_-,\lambda_+-\ell}} \ar[rru]  \\
  & \pU_{0,-}\otimes L^{\lambda_-}\ar[r]^(0.44){A^0(\zeta,\lambda_-)-\eta\bI}  &\pV_{0,-}\otimes L^{\lambda_-} \ar[rr]  &&\pQ_{0 ,0}^{ 1,0},
  }
\end{equation}
with extra terms
$h_{\alpha, \lambda_-, \lambda_+-\ell}$, $h_{\beta, \lambda_-, \lambda_+-\ell}$.

One can then use this to glue our pieces of the monad, to get a monad over 
$Z_\alpha$:
  \begin{equation} \label{define-pQ2-monad}
  \xymatrix@C=0.9em{  
        &\widetilde \pU_{1}\otimes L^{\lambda_-} \ar[r]^{  A^1(\zeta,\lambda_-)-\eta\bI} &\widetilde \pV_{1}\otimes L^{\lambda_-}\ar[rr]  &&\pQ_{0,1}^{0,1} \\
\pU_{+}\otimes L^{\lambda_+}\ \ \ar[ur]^{-\widetilde X_{ +,1}  } \ar[r]^{\widetilde A^1(\zeta,\lambda_+)-\eta\bI}\ar[dr]^{-X_{+,0} }&  \pV_{+}\otimes L^{\lambda_+}\ar[ur]^{ \widetilde Y_{+,1}}\ar[dr]^{Y_{+,0}}\ar[rr] &&\pP^{0,0}_{0,1}\ar[dr] \ar[ur]\\
 & \pU_{0,+}\otimes L^{\lambda_+}\ar[r]^{ A^0(\zeta,\lambda_+)-\eta\bI}  &\pV_{0,+}\otimes L^{\lambda_+} \ar[rr]  &&\pQ^{0 ,0}_{ 1,0} \\  
\pU_{0,-}\otimes L^{\lambda_+}(-\Gamma_\infty) \ar[ur]^\bI  \ar[ddr]_(0.7){- \pB_{t,h}}\ar[r]^(0.43){\xi\bI} &  \pV_{0,-}\otimes L^{\lambda_++\ell} (-\Gamma_0 ) (-1)\ar[ur]_{\psi\bI}   \ar[drr]
\ar[rddd]^(0.4){\pB_{t,h}}  \\ 
 &&&\pQ\ar[uur]\ar[ddr] \\
 \pU_{0,+}\otimes L^{\lambda_++\ell }( -\Gamma_0) \ar[dr]^(0.6){-h_{\alpha,-,\lambda_-,\lambda_+-\ell} }  \ar[uur]^(0.7){ -  \pB_{h,t}}\ar[r]^(0.43){\psi\bI} &  \pV_{0,+}\otimes L^{\lambda_+}(-\Gamma_\infty )  (-1)\ar[dr]^(0.6){\ \ \xi h_{\alpha,-,\lambda_-,\lambda_+-\ell}} \ar[rru] \ar[ruuu]_(0.4){\pB_{h,t}}  \\
  & \pU_{0,-}\otimes L^{\lambda_-}\ar[r]^{A^0(\zeta,\lambda_-)-\eta\bI}  &\pV_{0,-}\otimes L^{\lambda_-} \ar[rr]  &&\pQ_{0 ,0}^{ 1,0} \\
\pU_{-}\otimes L^{\lambda_-}\ \ \ar[dr]^{-X_{ -,1}  } \ar[r]^{\widetilde A^1(\zeta,\lambda_-)-\eta\bI}\ar[ur]^{-\widetilde X_{-,0} } &  \pV_{-}\otimes L^{\lambda_-}\ar[dr]^{\widetilde Y_{-,1}}\ar[ur]^{Y_{-,0}}\ar[rr]  &&\pP_{0,0}^{0,1}\ar[dr] \ar[ur]\\
 & \widetilde \pU_{1}  \ar[r]^(0.43){A^1(\zeta,0)-\eta\bI}  &\widetilde \pV_{1} \ar[rr]  &&\pQ_{0,1}^{0,1}. 
  }
\end{equation}
 One has a similar monad over $Z_\beta$.

 \subsubsection{Glueing the two monads}
 
 This procedure gives monads for the bundle $\pE$ over two open sets $Z_\alpha, Z_\beta$ of the twistor space. The question then arises of how to glue them over the intersection.   This turns out to be fairly simple: one has resolutions over the open sets which basically only differ by the trivialization one has given to the line bundles $L^\ell$. One then just has a map from  the monads over $Z_\alpha$ to the monad over $Z_\beta$ of the form $\exp (\ell\eta/\zeta)$  on each factor $L^\ell$. 
 
We note that while the bundles $L^\mu$ which are involved in the definition of the monad are globally defined, there are no global morphisms, not even meromorphic, between $L^\mu$ and $L^\nu$ for $\mu\neq \nu$; on the other hand, on the intersection of  our open sets, there are morphisms, and these allow an identification of the cohomology bundle. In short, one has a choice: either a global monad, but infinite dimensional, given by Nahm's equations and the Nahm construction, or a pair of finite rank monads, each can only be given locally over the twistor sphere, but then can be glued together, however, non-algebraically.

\section{Conclusion}

The equivalence of the monad directly producing the instanton with the monad encoded by the bow solution establishes the one-to-one correspondence between instantons and  bow solutions.  

In fact, the structure of the monad itself is intimately related to the bow, with its various parts associated with bow elements.  The vertical direction in diagram~\eqref{resolution-PQ-TN-10} corresponds to the bow interval, parameterized by $s.$ Each $\lambda$-point of the bow representation corresponds to a horizontal sequence in the monad, the ones resolving a sheaf $\pP$, such as the second line in \eqref{resolution-PQ-TN-10}; each  subinterval of the bow corresponds to a  horizontal sequence resolving a sheaf $\pQ$; finally, the `bifundamental' data correspond also to a piece of the diagram, given for example by lines  four and five of \eqref{resolution-PQ-TN-10}. This picture holds both for the fixed complex structure (e.g. the restriction to $\zeta= 0$) or the full twistor space description.

Indeed, we have a decomposition of the monad diagram into `modules', not in the mathematical sense, but in the sense of parts or pieces, for example (\ref {resolution-twistor-PQ-TN-1}), (\ref {resolution-twistor-PQ-TN-2}), and (\ref{define-pQ2}), each saying how  solutions  to Nahm's equations patch together at the end of intervals, either at $\lambda_\pm$ or at the edges.

{\it Other unitary groups:} To streamline our exposition we chose to focus our discussion on $SU(2)$ instantons on a single-centered Taub-NUT space.  Our discussion applies  directly to $U(N)$ instantons,  so long as the instanton has generic holonomy at infinity, and indeed the constructions use our pieces of diagram of type (\ref {resolution-twistor-PQ-TN-1}) repeatedly, as well as one piece of type (\ref{define-pQ2}). In a sense, this is a hybrid of the results above and those already found in \cite{HurtubiseMurray}.

 {\it Multi-Taub-NUT space:} Even more generally, one can apply the result to $U(N)$ instantons on multi-Taub-NUT $\mathrm{TN}_k.$ The relevant compact space $X$ in this case is the blow-up of $\mathbb{P}^1\times\mathbb{P}^1$ at $k+1$ points. The resulting monad again consists of the blocks we already discussed in detail.  Namely, it contains $k$ blocks each of type (\ref{define-pQ2}), and $N$ of type (\ref {resolution-twistor-PQ-TN-1}).  Just as the monads above, the resulting monad is periodic, i.e. it has its top line identified with its bottom line.

Our approach is quite general, allowing one to further explore instantons on other base spaces, such as, for example, D$_k$ ALF space or ALG spaces.  The former should lead to a monad organized according to a D-type bow, while the latter should lead to an entirely new object, which might deserve to be called a ``sling''.  If bows give a useful generalization of quivers \cite{Nakajima:2016guo}, slings generalize quivers further and should also lead to important application in quantum gauge theory and in geometric representation theory.  Another direction for our approach is exploring instantons with more general structure group, which is significant for the geometric Langlands correspondence for complex surfaces \cite{Witten:2009at}.

\section*{Acknowledgements}
The work of SCh was partially supported by the Simons Foundation grant {\footnotesize\#}245643, and that of JH by NSERC grant  {\footnotesize\#} RGPIN-2015-04393. 
The authors are grateful to the Isaac Newton Institute for Mathematical Sciences, Cambridge, for support and hospitality during the programme Metric and Analytic Aspects of Moduli Spaces where part of the work on this paper was undertaken, supported by EPSRC grant no EP/K032208/1.  SCh thanks the Berkeley Center for Theoretical Physics for hospitality during the final stages of this work. 
The authors also thank  the
Banff International Research Station for hospitality during the workshop
The Analysis of Gauge-Theoretic Moduli Spaces where this work was completed.

\providecommand{\bysame}{\leavevmode\hbox to3em{\hrulefill}\thinspace}
\providecommand{\MR}{\relax\ifhmode\unskip\space\fi MR }
\providecommand{\MRhref}[2]{%
  \href{http://www.ams.org/mathscinet-getitem?mr=#1}{#2}
}
\providecommand{\href}[2]{#2}


\begin{thebibliography}{AHDM78}

\bibitem[AHDM78]{ADHM}
M.F. Atiyah, N.J. Hitchin, V.G. Drinfeld, and Yu.I. Manin, \emph{{Construction
  of Instantons}}, Physics Letters A \textbf{65} (1978), no.~3, 185 -- 187,
  \href {http://dx.doi.org/http://dx.doi.org/10.1016/0375-9601(78)90141-X}
  {\path{doi:http://dx.doi.org/10.1016/0375-9601(78)90141-X}},
  \url{http://www.sciencedirect.com/science/article/pii/037596017890141X}.

\bibitem[BFN16]{Braverman:2016wma}
A.~Braverman, M.~Finkelberg, and H.~Nakajima, \emph{{Towards a Mathematical
  Definition of Coulomb Branches of $3$-dimensional $\mathcal N=4$ Gauge
  Theories, II}}, \href {http://arxiv.org/abs/1601.03586}
  {\path{arXiv:1601.03586}}.

\bibitem[BFN17]{Braverman:2017ofm}
\bysame, \emph{{Ring Objects in the Equivariant Derived Satake Category Arising
  from Coulomb Branches}}, \href {http://arxiv.org/abs/1706.02112}
  {\path{arXiv:1706.02112}}.

\bibitem[Biq91]{Biquard91}
O.~Biquard, \emph{{Fibr\'es Paraboliques Stables et Connexions Singuli\`eres
  Plates}}, Bull. Soc. Math. France \textbf{119} (1991), no.~2, 231--257,
  \url{http://www.numdam.org/item?id=BSMF_1991__119_2_231_0}. \MR{1116847}

\bibitem[Biq97]{Biquard97}
\bysame, \emph{{Fibr\'es de {H}iggs et Connexions Int\'egrables: le cas
  logarithmique (diviseur lisse)}}, Ann. Sci. \'Ecole Norm. Sup. (4)
  \textbf{30} (1997), no.~1, 41--96, \href
  {http://dx.doi.org/10.1016/S0012-9593(97)89915-6}
  {\path{doi:10.1016/S0012-9593(97)89915-6}}. \MR{1422313}

\bibitem[Buc87]{Buchdahl}
N.~P. Buchdahl, \emph{{Stable {$2$}-bundles on {H}irzebruch Surfaces}}, Math.
  Z. \textbf{194} (1987), no.~1, 143--152, \href
  {http://dx.doi.org/10.1007/BF01168013} {\path{doi:10.1007/BF01168013}}.
  \MR{871226}

\bibitem[CH08]{Charbonneau:2006gu}
B.~Charbonneau and J.~Hurtubise, \emph{{Calorons, Nahm's Equations on $S^1$ and
  Bundles over $\mathbb {P}^1 \times \mathbb {P}^1$}}, Commun. Math. Phys.
  \textbf{280} (2008), 315--349, \href {http://arxiv.org/abs/math/0610804}
  {\path{arXiv:math/0610804}}, \href
  {http://dx.doi.org/10.1007/s00220-008-0468-7}
  {\path{doi:10.1007/s00220-008-0468-7}}.

\bibitem[CH10]{Charbonneau:2007zd}
\bysame, \emph{{The Nahm Transform for Calorons}}, The many facets of geometry,
  Oxford Univ. Press, Oxford, 2010, pp.~34--70, \href
  {http://dx.doi.org/10.1093/acprof:oso/9780199534920.003.0004}
  {\path{doi:10.1093/acprof:oso/9780199534920.003.0004}}. \MR{2681686}

\bibitem[Che09]{Cherkis:2008ip}
S.~A. Cherkis, \emph{{Moduli Spaces of Instantons on the Taub-NUT Space}},
  Commun. Math. Phys. \textbf{290} (2009), 719--736, \href
  {http://arxiv.org/abs/0805.1245} {\path{arXiv:0805.1245}}, \href
  {http://dx.doi.org/10.1007/s00220-009-0863-8}
  {\path{doi:10.1007/s00220-009-0863-8}}.

\bibitem[Che11]{Cherkis:2010bn}
\bysame, \emph{{Instantons on Gravitons}}, Commun. Math. Phys. \textbf{306}
  (2011), 449--483, \href {http://arxiv.org/abs/1007.0044}
  {\path{arXiv:1007.0044}}, \href {http://dx.doi.org/10.1007/s00220-011-1293-y}
  {\path{doi:10.1007/s00220-011-1293-y}}.

\bibitem[CLHSa]{Second}
S.~A. Cherkis, A.~Larrain-Hubach, and M.~Stern, \emph{{Instantons on
  multi-Taub-NUT Spaces II: Bow Construction}}, {in preparation}.

\bibitem[CLHSb]{Third}
\bysame, \emph{{Instantons on multi-Taub-NUT Spaces III: Down Transform,
  Completeness, and Isometry}}, {in preparation}.

\bibitem[CLHS16]{Cherkis:2016gmo}
\bysame, \emph{{Instantons on multi-Taub-NUT Spaces I: Asymptotic Form and
  Index Theorem}}, \href {http://arxiv.org/abs/1608.00018}
  {\path{arXiv:1608.00018}}.

\bibitem[COS11]{Cherkis:2011ee}
S.~A. Cherkis, C.~O'Hara, and C.~Saemann, \emph{{Super Yang-Mills Theory with
  Impurity Walls and Instanton Moduli Spaces}}, Phys. Rev. \textbf{D83} (2011),
  126009, \href {http://arxiv.org/abs/1103.0042} {\path{arXiv:1103.0042}},
  \href {http://dx.doi.org/10.1103/PhysRevD.83.126009}
  {\path{doi:10.1103/PhysRevD.83.126009}}.

\bibitem[Don84]{Donaldson:1985id}
S.~K. Donaldson, \emph{{Nahm's Equations and the Classification of Monopoles}},
  Commun. Math. Phys. \textbf{96} (1984), 387--407, \href
  {http://dx.doi.org/10.1007/BF01214583} {\path{doi:10.1007/BF01214583}}.

\bibitem[GM88]{Garland-Murray}
H.~Garland and M.~K. Murray, \emph{{Kac-{M}oody Monopoles and Periodic
  Instantons}}, Comm. Math. Phys. \textbf{120} (1988), no.~2, 335--351,
  \url{http://projecteuclid.org/euclid.cmp/1104177753}. \MR{973538}

\bibitem[Gri85]{Griffiths}
P.~A. Griffiths, \emph{{Linearizing Flows and a Cohomological Interpretation of
  Lax Equations}}, American Journal of Mathematics \textbf{107} (1985), no.~6,
  1445--1484, \url{http://www.jstor.org/stable/2374412}.

\bibitem[GW09]{Gaiotto:2008ak}
D.~Gaiotto and E.~Witten, \emph{{S-Duality of Boundary Conditions In N=4 Super
  Yang-Mills Theory}}, Adv. Theor. Math. Phys. \textbf{13} (2009), no.~3,
  721--896, \href {http://arxiv.org/abs/0807.3720} {\path{arXiv:0807.3720}},
  \href {http://dx.doi.org/10.4310/ATMP.2009.v13.n3.a5}
  {\path{doi:10.4310/ATMP.2009.v13.n3.a5}}.

\bibitem[Hit79]{Hitchin79}
N.~J. Hitchin, \emph{{Polygons and Gravitons}}, Math. Proc. Cambridge Phil.
  Soc. \textbf{85} (1979), 465--476, \href
  {http://dx.doi.org/10.1017/S0305004100055924}
  {\path{doi:10.1017/S0305004100055924}}.

\bibitem[Hit83]{Hitchin:1983ay}
\bysame, \emph{{On the Construction of Monopoles}}, Comm. Math. Phys.
  \textbf{89} (1983), no.~2, 145--190,
  \url{http://projecteuclid.org/euclid.cmp/1103922679}. \MR{709461}

\bibitem[HM89]{HurtubiseMurray}
J.~Hurtubise and M.~K. Murray, \emph{{On the Construction of Monopoles for the
  Classical Groups}}, Comm. Math. Phys. \textbf{122} (1989), no.~1, 35--89,
  \url{http://projecteuclid.org/euclid.cmp/1104178316}. \MR{994495}

\bibitem[Hor64]{Horrocks}
G.~Horrocks, \emph{{Vector Bundles on the Punctured Spectrum of a Local Ring}},
  Proceedings of the London Mathematical Society \textbf{s3-14} (1964), no.~4,
  689--713, \href {http://dx.doi.org/10.1112/plms/s3-14.4.689}
  {\path{doi:10.1112/plms/s3-14.4.689}}.

\bibitem[Hur89]{Hurtubise:1989wh}
J.~Hurtubise, \emph{{The Classification of Monopoles for the Classical
  Groups}}, Commun. Math. Phys. \textbf{120} (1989), 613--641, \href
  {http://dx.doi.org/10.1007/BF01260389} {\path{doi:10.1007/BF01260389}}.

\bibitem[Jar04]{Jardim}
M.~Jardim, \emph{{A Survey on Nahm Transform}}, Journal of Geometry and Physics
  \textbf{52} (2004), 313--327, \href {http://arxiv.org/abs/math/0309305}
  {\path{arXiv:math/0309305}}, \href
  {http://dx.doi.org/10.1016/j.geomphys.2004.03.006}
  {\path{doi:10.1016/j.geomphys.2004.03.006}}.

\bibitem[Mur84]{murray-monopoles}
M.~K. Murray, \emph{{Nonabelian Magnetic Monopoles}}, Comm. Math. Phys.
  \textbf{96} (1984), no.~4, 539--565,
  \url{http://projecteuclid.org/euclid.cmp/1103941913}. \MR{775045}

\bibitem[Nah80]{Nahm:1979yw}
W.~Nahm, \emph{{A Simple Formalism for the BPS Monopole}}, Phys. Lett.
  \textbf{90B} (1980), 413--414, \href
  {http://dx.doi.org/10.1016/0370-2693(80)90961-2}
  {\path{doi:10.1016/0370-2693(80)90961-2}}.

\bibitem[Nah83a]{Nahm1983}
\bysame, \emph{{All Self-Dual Multimonopoles for Arbitrary Gauge Groups}},
  pp.~301--310, Springer US, Boston, MA, 1983, \href
  {http://dx.doi.org/10.1007/978-1-4613-3509-2_21}
  {\path{doi:10.1007/978-1-4613-3509-2_21}}.

\bibitem[Nah83b]{Nahm:1982jt}
\bysame, \emph{{The Algebraic Geometry of Multimonopoles}}, pp.~456--466,
  Springer Berlin Heidelberg, Berlin, Heidelberg, 1983, \href
  {http://dx.doi.org/10.1007/3-540-12291-5_74}
  {\path{doi:10.1007/3-540-12291-5_74}}.

\bibitem[NS00]{Nye-Singer}
T.~M.~W. Nye and M.~A. Singer, \emph{{An $L^{2}$ Index Theorem for Dirac
  Operators on $S^1\times \mathbb {R}^{3}$}}, J. Funct. Anal. \textbf{177}
  (2000), no.~1, 203--218, \href {http://arxiv.org/abs/math/0009144}
  {\path{arXiv:math/0009144}}, \url{https://doi.org/10.1006/jfan.2000.3648}.

\bibitem[NT17]{Nakajima:2016guo}
H.~Nakajima and Y.~Takayama, \emph{{Cherkis Bow Varieties and Coulomb Branches
  of Quiver Gauge Theories of Affine Type $A$}}, Sel. Math. New Ser.
  \textbf{23} (2017), \href {http://arxiv.org/abs/1606.02002}
  {\path{arXiv:1606.02002}}, \href
  {http://dx.doi.org/10.1007/s00029-017-0341-7}
  {\path{doi:10.1007/s00029-017-0341-7}}.

\bibitem[Tak16]{Takayama:2016}
Y.~Takayama, \emph{{Nahm's Equations, Quiver Varieties and Parabolic Sheaves}},
  Publ. Res. Inst. Math. Sci. \textbf{52} (2016), no.~1, 1--41, \href
  {http://dx.doi.org/10.4171/PRIMS/172} {\path{doi:10.4171/PRIMS/172}}.
  \MR{3542891}

\bibitem[Wit09]{Witten09}
E.~Witten, \emph{{Branes, Instantons, and Taub-NUT Spaces}}, JHEP \textbf{06}
  (2009), 067, \href {http://arxiv.org/abs/0902.0948} {\path{arXiv:0902.0948}},
  \href {http://dx.doi.org/10.1088/1126-6708/2009/06/067}
  {\path{doi:10.1088/1126-6708/2009/06/067}}.

\bibitem[Wit10]{Witten:2009at}
\bysame, \emph{Geometric {L}anglands from six dimensions}, A celebration of the
  mathematical legacy of {R}aoul {B}ott, CRM Proc. Lecture Notes, vol.~50,
  Amer. Math. Soc., Providence, RI, 2010, pp.~281--310, \href
  {http://arxiv.org/abs/0905.2720} {\path{arXiv:0905.2720}}. \MR{2648898}

\end{thebibliography}
  \end{document}